\documentclass[english,letterpaper]{smfart}

\usepackage[utf8]{inputenc}
\usepackage[T1]{fontenc}

\usepackage{lmodern}
\usepackage{textcomp}

\usepackage[leqno]{amsmath}
\usepackage{amsfonts,%
  amssymb,%
  amsthm,%
  amsxtra%
}

\usepackage{mathrsfs} 

\usepackage[all,tips,2cell]{xy}
\SelectTips{cm}{}
\UseAllTwocells

\usepackage{xspace} 
\usepackage{ifpdf}
\usepackage{ifthen}
\usepackage{color}
\ifthenelse{\boolean{pdf}}{%
  \definecolor{theblue}{rgb}{0.02,0.04,0.8}%
  \definecolor{thered}{rgb}{0.8,0.04,0.07}%
  \definecolor{thegreen}{rgb}{0.06,0.44,0.08}%
  \definecolor{thegrey}{gray}{0.5}%
  \definecolor{theshade}{gray}{0.92}%
  \usepackage[%
  colorlinks=true,
  citecolor=theblue,
  linkcolor=thered,
  backref=page,
  plainpages=false,
  pdfpagelabels,bookmarks]{hyperref}}{%
  \usepackage{hyperref}%
}

\swapnumbers
\newtheorem*{theorem*}{Theorem}
\newtheorem{theorem}{Theorem}[subsection]
\newtheorem{proposition}[theorem]{Proposition}
\newtheorem*{proposition*}{Proposition}
\newtheorem{lemma}[theorem]{Lemma}
\newtheorem{corollary}[theorem]{Corollary}
\newtheorem{lemma-def}[theorem]{Lemma-Definition}
\newtheorem{claim}[theorem]{Claim}

\theoremstyle{definition}
\newtheorem{definition}[theorem]{Definition}

\newtheorem{paragr}[theorem]{} 

\theoremstyle{remark}
\newtheorem{remark}[theorem]{Remark}
\newtheorem*{remark*}{Remark}

\numberwithin{equation}{theorem}%
\newcommand{\braid}[2]{\lbrace #1, #2\rbrace}
\newcommand{\cprod}[1]{\wedge^{#1}} 

\newcommand{\eqdef}{\overset{\text{\normalfont\tiny def}}{=}}
\newcommand{\coloneq}{\mathrel{\mathop:}=}
\renewcommand{\phi}{\varphi}
\renewcommand{\epsilon}{\varepsilon}

\newcommand{\Cech}{\v{C}ech\xspace}

\newcommand{\loccit}{loc.\ cit.\xspace}

\newcommand{\lto}{\longrightarrow} 
\newcommand{\lmto}{\longmapsto}
\newcommand{\iso}{\simeq}
\newcommand{\isoto}{\overset{\sim}{\rightarrow}} 
\newcommand{\lisoto}{\overset{\sim}{\longrightarrow}} 
\newcommand{\coin}{\equiv} 

\newcommand{\coho}[1]{\operatorname{\mathrm{#1}}}%
\renewcommand{\H}{\coho{H}} 
\renewcommand{\L}{\coho{L}} 
\newcommand{\shH}{\operatorname{\underline{\mathrm{H}}}} 

\makeatletter
\def\varLim@#1#2{%
  \vtop{\m@th\ialign{##\cr
    \hfil$#1\operator@font Lim$\hfil\cr
    \noalign{\nointerlineskip\kern1.5\ex@}#2\cr
    \noalign{\nointerlineskip\kern-\ex@}\cr}}%
}
\def\varinjLim{%
  \mathop{\mathpalette\varLim@{\rightarrowfill@\textstyle}}\nmlimits@
}
\def\varprojLim{%
  \mathop{\mathpalette\varLim@{\leftarrowfill@\textstyle}}\nmlimits@
}
\makeatother

\DeclareMathOperator{\Aut}{Aut}
\newcommand{\shAut}{\operatorname{\underline{\mathrm{Aut}}}}
\newcommand{\shcatAut}{\operatorname{\mathscr{A}\mspace{-3mu}\mathit{ut}}}
\DeclareMathOperator{\B}{B}
\DeclareMathOperator{\cech}{\check{C}}
\DeclareMathOperator{\Coker}{Coker}
\DeclareMathOperator{\cosk}{cosk}
\DeclareMathOperator{\sk}{sk}
\newcommand{\del}{\partial} 
\newcommand{\Desc}{\mathrm{Desc}}

\DeclareMathOperator{\catExt}{\mathsf{Ext}}
\DeclareMathOperator{\shcatExt}{\mathscr{E}\mspace{-3mu}\mathit{xt}}
\DeclareMathOperator{\Hom}{Hom} 

\DeclareMathOperator{\catHom}{\mathsf{Hom}}
\DeclareMathOperator{\twocatHom}{\underline{\mathsf{Hom}}}
\DeclareMathOperator{\shHom}{\underline{\mathrm{Hom}}}
\DeclareMathOperator{\shcatHom}{\mathscr{H}\mspace{-4mu}\mathit{om}}
\DeclareMathOperator{\id}{id} 
\DeclareMathOperator{\im}{Im}
\DeclareMathOperator{\Id}{Id}
\DeclareMathOperator{\Ker}{Ker} 
\newcommand{\shcatKer}{\mathscr{K}\mspace{-4mu}\mathit{er}}
\newcommand{\colim}{\varinjlim}
\renewcommand{\lim}{\varprojlim}

\newcommand{\Lim}{\varprojLim}
\DeclareMathOperator{\N}{N}
\DeclareMathOperator{\W}{\overline{W}}

\newcommand{\cat}[1]{\mathsf{#1}} 
\newcommand{\twocat}[1]{\underline{\mathsf{#1}}}  
\newcommand{\bicat}[1]{\underline{\mathsf{#1}}} 
\DeclareMathOperator{\Ob}{Ob}
\DeclareMathOperator{\Mor}{Mor}

\newcommand{\Set}{\cat{Set}}
\newcommand{\Mod}{\cat{Mod}}

\newcommand{\site}[1]{\cat{#1}}

\newcommand{\sS}{\site{S}}
\newcommand{\s}{\sS}
\newcommand{\rng}[1][\s]{\mathcal{O}_{#1}}

\newcommand{\pre}[1]{{\site{#1}^\wedge}}

\newcommand{\prshv}{\pre{\s}}

\newcommand{\grpd}[1]{\cat{#1}}

\newcommand{\smp}[1]{\underline{#1}}

\newcommand{\fib}[1]{\mathscr{#1}}
\newcommand{\fibf}{\fib{F}}

\newcommand{\stack}[1]{\mathscr{#1}}
\newcommand{\stb}{\stack{B}}
\newcommand{\stc}{\stack{C}}
\newcommand{\stf}{\stack{F}}

\newcommand{\stx}{\stack{X}}
\newcommand{\sty}{\stack{Y}}

\newcommand{\stwm}{\stack{W}\mspace{-6mu}\stack{M}}

\newcommand{\grstack}[1]{\mathscr{#1}}
\newcommand{\gra}{\grstack{A}}

\newcommand{\grc}{\grstack{C}}
\newcommand{\grd}{\grstack{D}}
\newcommand{\grg}{\grstack{G}}
\newcommand{\grh}{\grstack{H}}
\newcommand{\grk}{\grstack{K}}

\newcommand{\twofib}[1]{\mathfrak{#1}}

\newcommand{\twostack}[1]{\mathfrak{#1}}
\newcommand{\bistack}[1]{\mathfrak{#1}}
\newcommand{\tstf}{\twostack{F}}

\newcommand{\CM}{\bicat{XMod}} 
\newcommand{\cm}{\bistack{XMod}} 
\newcommand{\picCM}{\bicat{PicXMod}} 
\newcommand{\piccm}{\bistack{PicXMod}} 
\newcommand{\ch}{\bistack{Ch}^{[-1,0]}}
\newcommand{\tors}{\operatorname{\textsc{Tors}}}

\newcommand{\stacks}{\operatorname{\textsc{Stacks}}}
\newcommand{\grstacks}{\operatorname{\textsc{Gr\mbox{-}Stacks}}}
\newcommand{\picstacks}{\twostack{Pic}}

\author{Ettore Aldrovandi}
\address{Department of Mathematics\\
  Florida State University\\
  1017 Academic Way\\
  Tallahassee, FL 32306-4510, USA}
\email{aldrovandi@math.fsu.edu}
\author{Behrang Noohi}
\address{Department of Mathematics\\
  Florida State University\\
  1017 Academic Way\\
  Tallahassee, FL 32306-4510, USA}
\email{noohi@math.fsu.edu}
\title{Butterflies I: morphisms of 2-group stacks}

\begin{document}
\frontmatter
\begin{abstract}
  Weak morphisms of non-abelian complexes of length 2, or crossed
  modules, are morphisms of the associated 2-group stacks, or
  gr-stacks.  We present a full description of the weak morphisms
  in terms of diagrams we call butterflies. We give a complete
  description of the resulting bicategory of crossed modules,
  which we show is fibered and biequivalent to the 2-stack of 2-group
  stacks. As a consequence we obtain a complete characterization
  of the non-abelian derived category of complexes of length
  2.  Deligne's analogous theorem in the case of
  Picard stacks and abelian sheaves becomes an immediate corollary.
  Commutativity laws on 2-group stacks are also analyzed in terms
  of butterflies, yielding new characterizations of braided,
  symmetric, and Picard 2-group stacks.  Furthermore, the
  description of a weak morphism in terms of the corresponding
  butterfly diagram allows us to obtain a long exact sequence in
  non-abelian cohomology, removing a preexisting fibration
  condition on the coefficients short exact sequence.
\end{abstract}

\maketitle%
\tableofcontents%

\mainmatter
\section{Introduction}

\label{sec:introduction}

The notion of $n$-group or, loosely speaking, $n$-categorical
group is by now quite established in mathematics.  This paper is
the first of a series aimed at a systematic study of the
$n$-category of $n$-groups.  We do it in a geometric fashion,
working with $n$-groups over a general Grothendieck site, therefore
we should rather be saying that we are studying the $n$-stack of
$n$-groups. (This is a bit imprecise, though, and it will be
appropriately qualified later, especially concerning laxness of
the various $n$-categorical constructions we consider.) Torsors
over $n$-groups are also included in our study.

The $n$-categorical aspect is emphasized, since what is of
fundamental importance are the morphisms (1-morphisms and higher)
between $n$-groups. In particular, while it is relatively
harmless to consider $n$-groups as ``strict,'' in the categorical
sense, we now understand it is not so when dealing with their
morphisms. In other words, \emph{weak} morphisms are the
important ones---and they cannot be made strict. Thus one of our
main points is to characterize precisely and provide an explicit
and---we believe---very manageable construction of these weak
morphisms: the butterflies of this series' title.  In doing so,
we are also in position to obtain a very complete and concrete
description of \emph{torsors} over said $n$-groups, in particular
concerning the functorial aspects, where previously only strict
morphisms were considered.

To explain things in slightly more detail, it is useful to do it
in the case $n=2$, which is covered in the present paper
and its sequel \cite{ButterfliesII}, which deals with the
torsors, whereas the situation for $n=3$ and higher is dealt with
in \cite{ButterfliesIII,ButterfliesIV}.  The reader shall be
aware that many remarks and observations apply to higher $n$'s as
well.

What we do in this paper is present a general framework for
morphisms of 2-groups on a Grothendieck site.  It is roughly
divided in two parts.  The first, more general part deals with
the 2-categorical aspects; the second, discusses what may even be
called applications of the theory of morphisms of 2-groups built
in the first part. We discuss exact sequences of 2-groups, the
long exact sequence in non-abelian cohomology, and devote some
space to 2-groups in abelian categories, that is Picard
groupoids, as in~\cite{0259.14006}. This is the fundamental
motivating example, and in a sense, one could characterize the
present (and similar) investigations as quest for a geometric
approach to the (non-abelian) derived category in the same spirit
as \loccit

\subsection{General beginning remarks}
\label{sec:gener-beginn-remarks}

First, a question of terminology: in this introduction, and when
used in the rest of the paper, the term ``2-group'' without
further qualification usually means ``2-group stack.''  This is
(should be?) a modern substitute for the term ``gr-stack.'' Over
a point, a gr-stack becomes a gr-category, and this is also
called a 2-group (over a point). Not having been able to complete
the transition ourselves, we too pervasively use the terms
gr-stack and gr-category in the main body.

Whichever way it is called, a 2-group is usually by nature a lax
object, in the sense that its algebraic operations involve higher
coherences; for example, multiplication is only associative up to
coherent isomorphism. This certainly the case for 2-groups
\emph{stacks} over a site.

Over a point, it has been known for
quite some time (see \cite{sinh:gr-cats}) that a 2-group
$\mathscr{G}$ can always be made strict.  This means that
$\mathscr{G}$ is equivalent, as a category, to some $\grpd{G}$
such that for the latter the group operations hold with
equalities, and not just up to isomorphism, and the equivalence
is an additive functor.  Standard arguments then imply that
$\grpd{G}$ is the strict gr-category determined by---and
determining---a crossed module
\begin{equation*}
  G^{\bullet} \colon G^{-1} \overset{\delta}{\lto} G^0.
\end{equation*}
$G^\bullet$ is the reduced Moore complex obtained from the nerve
$\smp{G}_\bullet = \N_\bullet \grpd{G}$, which is a simplicial
group.  Over a site $\s$ something very similar can be achieved,
namely it is possible to find a (sheaf of) crossed
module(s) $G^\bullet$ such that $\grstack{G}$ is now equivalent
to the stack associated to the (sheaf of) groupoid(s)
$\grpd{G}$---so this is a prestack---determined by $G^\bullet$.
This is certainly well known, but since it plays an important
rôle in our arguments throughout, we have decided to provide an
explicit statement with proof (see~\ref{prop:9}).

On the other hand, when we turn our attention to morphisms
between 2-groups, it is not possible to make them strict, and
this is already true---and well known---in the set-theoretic
case, that is, over a point.  To express it in a more precise
way, let $F\colon \mathscr{H}\to \mathscr{G}$ be an additive
functor between 2-groups. Assuming $\mathscr{H}$ and
$\mathscr{G}$ are equivalent to strict 2-groups $\grpd{H}$ and
$\grpd{G}$, in general $F$ will not be isomorphic to a strict
morphism $\grpd{H}\to \grpd{G}$ or, to put it in a different but
equivalent way, cannot be realized as a morphism of crossed
modules $H^\bullet \to G^\bullet$.  Thus we must grapple with the
problem that the groupoid $\shcatHom (\mathscr{H}, \mathscr{G})$
of additive functors is much larger than that of strict morphisms
from $\grpd{H}$ to $\grpd{G}$. It follows that in order to work
with strict 2-groups and still retain all the features of the
2-category of 2-groups, one must find a ``derived'' version
$\mathbf{R}\catHom (\grpd{H}, \grpd{G})$ of the groupoid of
strict morphisms from $\grpd{H}$ to $\grpd{G}$. The requirement
is that this new groupoid be equivalent to $\shcatHom
(\mathscr{H}, \mathscr{G})$. It objects are the \emph{weak
  morphisms} from $\grpd{H}$ to $\grpd{G}$.

In the set-theoretic case, this problem was solved by the
second-named author in refs.\ \cite{Noohi:notes,Noohi:weakmaps}.
The existence of $\mathbf{R}\catHom (\grpd{H}, \grpd{G})$ was
established by methods of homotopical algebra. Then, a very
concrete description of it was given in terms of group diagrams
called \emph{butterflies.}

A butterfly from $H^\bullet$ to $G^\bullet$ is a diagram of the
form
\begin{equation}
  \label{eq:1}
  \vcenter{%
  \xymatrix@=1pc{%
    H^{-1}\ar[dd] \ar[dr]  & &G^{-1} \ar[dl] \ar[dd]\\
    & E\ar[dl] \ar[dr] &  \\
    H^0 & & G^0
  }}\tag{*}
\end{equation}
where the NW-SE sequence is a complex, the NE-SW sequence is
exact, i.e.\ a group extension, plus some other conditions which
will be explained later. There is also a notion of morphism of
butterflies: it is induced by a group isomorphism $E\isoto E'$
compatible with the rest of the various maps, so one has a
groupoid.  Strict morphisms corresponds to butterfly diagrams
whose NE-SW sequence is split, with the group at the center equal
to the semi-direct product $H^0\ltimes G^{-1}$.

Incidentally, the adjective ``derived'' used earlier is not
entirely out of place: an equivalent formulation of the butterfly
is that a weak morphism can be realized as a triangle
\begin{equation}
  \label{eq:2}
  \vcenter{%
  \xymatrix@C-2.5pc{%
    &  [H^{-1}\times G^{-1} \to E ] \ar @(l,ur) [dl]_(0.7)\sim
    \ar @(r,ul) [dr] & \\ 
    [H^{-1} \to H^0] & &  [G^{-1} \to G^0]
  }}\tag{+}
\end{equation}
where each arrow is a strict morphism of crossed modules, and the
left one is a quasi-isomorphism (it preserves homotopy groups).

Since butterflies can be nicely composed by juxtaposition, one
obtains a nice bicategory of strict 2-groups and weak morphisms
which carries the right kind of homotopical information.

The actual choice of the definition of weak morphism is discussed
in section~\ref{sec:remarks-weak-morph}, where it is compared to
the topological one given in refs.\
\cite{Noohi:notes,Noohi:weakmaps}.  One may also wonder whether
it is possible to further weaken the notion of weak morphism as
given here. We observe (still in
section~\ref{sec:remarks-weak-morph}) that it is related to a
categorification of Morita's theory and is, however, essentially
a theory of 2-stacks. Furthermore it leads to non-additive
functors.

\subsection{The content of the paper}
\label{sec:content-paper}
The following is a rather long discussion of the main ideas in
this paper.  For a quick glance at the content's description, the
reader is invited to read section~\ref{sec:organization-paper}
first.

\subsubsection{Butterflies and 2-groups}
\label{sec:butterflies-2-groups}

The first part of this paper, and in part its immediate sequel
\cite{ButterfliesII}, center on the same question of finding a
correct model for the 2-category of 2-group stacks: we want to
strictify the objects, but retain an accurate information on the
morphisms. At a minimum, the ``model'' in question is a
bicategory, whose objects are the crossed modules over the site,
equipped with a biequivalence with the 2-category of
2-groups. More is true when working over a general site, as we
now explain.

First, let us be clear about the 2-category of 2-groups, let us
denote it by $\twocat{Gr\mbox{-}Stacks}(\s)$.  This is an honest
2-category, and it is a sub-2-category of the 2-category
$\twocat{Stacks} (\s)$. Moreover, considering the site $\s /U$,
for each object $U\in \s$, and the 2-category
$\twocat{Gr\mbox{-}Stacks}(\s/ U)$ yields a fibered 2-category
over the site $\s$ (in the sense of \cite{MR0364245}). It turns
out this fibered 2-category, denoted $\grstacks (\s)$, actually
is a 2-stack over $\s$.  This fact is certainly well known to
experts, but, being unable to find a published account of it, we
have included it here (cf.\ Appendix~\ref{sec:2-stack-gr}).

We have already observed that given a 2-group $\grg$ over $\s$ we
can find a crossed module $G^\bullet$ such that its associated
stack is equivalent to $\grg$ (cf.\ Proposition~\ref{prop:9}). It
is convenient to denote by $[G^{-1}\to G^0]\sptilde$ the stack
associated to $G^\bullet$. Note that the groupoid $\grpd{G}$
determined by $G^\bullet$ is in general only a prestack: in going
from the 2-group to the crossed module, the price we pay is to
lose the stack condition, namely the gluing conditions on
objects.

With the obvious changes, the notion of butterfly
diagram~\eqref{eq:1} still makes sense over $\s$, as well as that
of morphism of butterflies.  Let $\cat{B}(H^\bullet, G^\bullet)$
be the groupoid of butterflies from $H^\bullet$ to
$G^\bullet$. In the main text we prove (cf.\
Theorem~\ref{thm:2}):
\begin{theorem*}
  There is an equivalence
  \begin{equation*}
    \cat{B}(H^\bullet, G^\bullet)\lisoto 
    \catHom_{\twocat{Gr\mbox{-}Stacks} (\s)} ({H^\bullet}\sptilde, {G^\bullet}\sptilde).
  \end{equation*}
\end{theorem*}
The right-hand side is more or less by definition the groupoid of
\emph{weak morphisms} from $H^\bullet$ to $G^\bullet$.  Moreover,
both sides have fibered analogs
\begin{equation*}
  \stack{B}(H^\bullet, G^\bullet)\,,\quad
  \shcatHom_{\grstacks (\s)} ({H^\bullet}\sptilde, {G^\bullet}\sptilde)
\end{equation*}
over $\s$ which are stacks (in groupoids).

For three crossed modules $K^\bullet$, $H^\bullet$, and
$G^\bullet$ there is a non-associative composition
\begin{equation*}
  \cat{B}(K^\bullet, H^\bullet)\times
  \cat{B}(H^\bullet, G^\bullet)\lto
  \cat{B}(K^\bullet, G^\bullet),
\end{equation*}
again given by juxtaposition, and similarly for $\stack{B}(-,-)$
replacing $\cat{B}(-,-)$. Combining with the fact that for every
2-group $\grg$ we can find a crossed module $G^\bullet$ such that
there is an equivalence $\grg\iso {G^\bullet}\sptilde$, we obtain
the following statement (roughly corresponding to
Theorem~\ref{thm:5}):
\begin{theorem*}
  There is a bicategory $\cm (\s)$ fibered over $\s$ whose
  objects are crossed modules, 1-morphisms are butterflies, and
  2-morphisms are morphisms of butterflies. Moreover $\cm (\s)$
  is a bistack over $\s$ and the correspondence $G^\bullet \to
  {G^\bullet}\sptilde$ induces a biequivalence
  \begin{equation*}
    \cm (\s) \lisoto \grstacks (\s).
  \end{equation*}
\end{theorem*}
The right-hand side above is a genuine 2-stack, i.e.\ it is
fibered in 2-categories, which is considered as being fibered in
bicategories in the obvious way.

This theorem has the rather striking consequence that crossed
modules can be glued relative to 2-descent data formulated in
terms of butterflies.

We would like to remark that the previous theorem gives
us the right take on on the strict/weak question. Namely, on one
hand we have strict 2-groups: they are simpler to deal with, but
somewhat non geometric, in the sense that the stack condition does
not hold, and therefore there is no gluing on their objects, in
general.  Strictness of the group law entails that they must comprise
a bicategory: their morphisms compose in a non-associative way.
At the same time morphisms can be described rather concretely, in
terms of diagram, but observe that they necessarily cannot be
functors relative to the (strict) group law. On the
other hand, weak 2-groups, that is 2-group stacks, are seemingly
more complicated, but they are the true geometric, in that their
objects glue and can be given a description in terms of
torsors. Weakness of the group law allows them to collectively
form a genuine 2-category:  morphisms between weak 2-groups are
\emph{functors} relative to the weak group law, and therefore
they obviously compose in an associative way.

\subsubsection{Exact sequences and abelian categories}
\label{sec:exact-sequ-abel}

There is a number of immediate applications ensuing from the
notion of butterfly diagram, which are discussed in the rest of
the paper.
The reason for including them in this paper is
that they are closer to the general theory developed in the first
part of the paper.  In particular, they result from the analysis
of the homotopy kernel and homotopy fiber of a butterfly, that is
of a weak morphism.

The motivating example is again~\cite{0259.14006}: weak morphisms of
length-two complexes of abelian sheaves correspond to additive
functors of Picard stacks. This allows to geometrically describe
the derived category $\cat{D}^{[-1,0]}(\s)$ of complexes of
abelian sheaves over $\s$ whose homology is concentrated in
degrees $[-1,0]$.

Deligne's constructions become a special case of those of
sections~\ref{sec:butt-weak-morph} and
\ref{sec:fiber-bicat-cross}, when they are specialized to the
abelian category of abelian sheaves over $\s$. In this case, the
objects are complexes $G^{-1}\to G^0$ of abelian sheaves without
further qualifications. The associated stack $[G^{-1}\to
G^0]\sptilde$ is Picard.  Since the butterfly diagram from
$H^\bullet$ to $G^\bullet$ is a more or less canonical
representation of a weak morphism, i.e.\ an additive functor of
Picard stacks $[H^{-1}\to H^0]\sptilde \to [G^{-1}\to
G^0]\sptilde$, its cone in the derived category $\cat{D} (\s)$
(the cone is no longer in $\cat{D}^{[-1,0]} (\s)$ is visible in
the butterfly: it is the NW-SE diagonal of~\eqref{eq:1}.

In the non-abelian setting one should rather be using the
homotopy fiber construction, as it was done in
\cite{Noohi:weakmaps} in the set-theoretic context. Over a
general site $\s$, the corresponding construction is the
following (cf.\ Theorem~\ref{thm:6}).
\begin{theorem*}
  Let $F^\bullet\colon H^{-1}\to E\to G^0$ be the complex in
  degrees $[-2,0]$ corresponding to the NW-SE diagonal
  of~\eqref{eq:1}. There is a 2-stack $\tstf$ over $\s$
  associated to $F^\bullet$ whose
  homotopy groups in the sense of \cite{MR95m:18006} are
  \begin{equation*}
    \pi_i (\tstf) = \H^{-i} (F^\bullet)\,,\quad i=0,1,2,
  \end{equation*}
  and fit into the expected exact sequence
  \begin{equation*}
    \xymatrix{%
      0 \ar[r] & \pi_{2}(\twostack{F}) \ar[r]&
      \pi_{1}(\grh) \ar[r] & \pi_{1} (\grg) \ar[r] &
      \pi_{1}(\twostack{F}) \ar `r[d] `[lll] `[dlll] [dlll]\\
      & \pi_{0}(\grh) \ar[r] & \pi_{0} (\grg) \ar[r] & \pi_0
      (\twostack{F}) \ar[r] & 1
    }
  \end{equation*}
  The 2-stack $\tstf$ is part of the homotopy fiber sequence of
  2-stacks over $\s$:
  \begin{equation*}
    \grh[0] \xrightarrow{F[0]} \grg[0]
    \overset{\iota}{\lto} \tstf
    \overset{\Delta}{\lto}
    \tors ({\grh}) \overset{F_*}{\lto} \tors ({\grg}).
  \end{equation*}
\end{theorem*}
This provides an imperfect analog of the exact triangle in the
non-abelian setting. The fact that $\tstf$ is a 2-stack
corresponds to the fact that in the (abelian) derived category
situation by taking the cone we are now dealing with complexes of
length three. More serious is the fact that $\tstf$ does not
admit a group law, in general.

A special case occurs when $\pi_0 (\tstf) = *$. In this case the
butterfly~\eqref{eq:1} corresponds to an essentially surjective
morphism $\grh\to \grg$.  The homotopy fiber of such a morphism
is correspondingly simpler
\begin{equation*}
  \tstf \iso \tors (\grk),
\end{equation*}
where $\grk$ is the homotopy kernel of the morphism. It is itself
a 2-group stack and it has an explicit characterization in terms
of the group-objects occurring in the butterfly; namely we have
\begin{proposition*}[~\ref{prop:10}]
  The homotopy kernel of~\eqref{eq:1} is equivalent to the
  2-group stack associated to the crossed module
  \begin{equation*}
    H^{-1} \lto \ker (E\to G^0).
  \end{equation*}
\end{proposition*}
Thus the butterfly corresponds to the second morphism of an
extension of 2-group stacks:
\begin{equation}
  \label{eq:86}
  \grk \lto \grh \lto \grg,\tag{**}
\end{equation}
as in \cite{MR93k:18019}. Unlike \loccit, we do
\emph{not} require that the second morphism be a
fibration. Thanks to~\eqref{eq:2}, $\grh$ can be replaced with an
equivalent $\grstack{E}$. The morphism $\grstack{E}\to \grg$
\emph{is} a fibration. Appealing to a former results of Breen
(\cite{MR92m:18019}, which uses the fibration condition) we are
able to conclude that the short exact sequence~\eqref{eq:86}
without the fibration condition still induces the long exact
sequence in non-abelian cohomology (Proposition~\ref{prop:12}).

We discuss the relation with Deligne's constructions in \cite[\S
1.4]{0259.14006} in detail in
section~\ref{sec:abelian-categories}. The general idea is that in
an abelian category a crossed module is simply a complex of
length 2, therefore all constructions carry over, by simply
forgetting most of the action requirement.

If $A^\bullet \colon A^{-1}\to A^0$ is such a complex, its
associated stack is Picard. This is a very strong commutativity
condition. It is well known that on a 2-group one can impose
various degree of commutativity on the monoidal operation, so
that it becomes, in order of increasing specialization: (1)
Braided, (2) Symmetric, (3) Picard.

One thing we do is obtain all these conditions from a special
kind of butterfly diagram, which is necessarily associated to a
braided 2-group $\grg$.  We argue that a braiding on the monoidal
structure of $\grg$ is tantamount to requiring that the monoidal
structure itself
\begin{equation*}
  \otimes \colon \grg\times \grg \lto \grg
\end{equation*}
be an additive functor of 2-group stacks.  By our theorem this
must be realized by an appropriate butterfly diagram of crossed
modules
\begin{equation}
  \label{eq:87}
  \vcenter{%
    \xymatrix@R-1em{%
      G^{-1}\times G^{-1} \ar[dd] \ar@/_0.1pc/[dr]  & &
      G^{-1} \ar@/^0.1pc/[dl] \ar[dd]\\
      & P\ar@/_0.1pc/[dl] \ar@/^0.1pc/[dr] &  \\
      G^0\times G^0 & & G^0
    }}\tag{++}
\end{equation}
which must satisfy other conditions too, most notably that the
extension given by the NE-SW diagonal be split (with a fixed
splitting) when restricted to either factor.  We can go as far as
\emph{defining} a crossed module to braided if it admits such a
structure. This is reasonable in view of the following
\begin{proposition*}[~\ref{prop:15}]
  A 2-group stack $\grg$ is braided if and only if a (hence, any)
  corresponding crossed module $G^\bullet$ is braided in the
  standard sense by a braiding map $\braid{-}{-}\colon G^0\times
  G^0\to G^{-1}$, if and only if it is braided in the sense of
  admitting a butterfly diagram such as~\eqref{eq:87}.
\end{proposition*}
The symmetry and Picard conditions on the braiding structures can
be described in an entirely similar fashion. Namely, let $T
\colon \grg\times \grg\to \grg\times \grg$ be the swap morphism,
and let $\Delta\colon \grg\to \grg\times \grg$ be the diagonal
morphism. We have:
\begin{proposition*}[~\ref{prop:16} and
  \ref{sec:picard-cross-modul}] \hfill
  \begin{itemize}
  \item The braiding on $\grg$ is symmetric if and only
    if~\eqref{eq:87} is isomorphic to its own pullback by $T$;
  \item The braiding is furthermore Picard if in addition the
    induced isomorphism on the butterfly pulled back by $\Delta$
    is the identity.
  \end{itemize}
\end{proposition*}
There is of course a notion of \emph{braided butterfly} between
braided 2-groups which expresses the fact that the corresponding
morphism is a morphism of braided objects.

We are therefore able to approach Picard stacks along two
slightly different lines: as a direct byproduct of the general
theory via the additional properties imposed by the Picard
condition, as explained above, or as a special repetition of the
general theory for an abelian category.

An interesting example of the latter arises when considering
extra structures, in particular the one provided by the existence
of a sheaf of commutative rings on the base site. In a ringed
site we can talk about \emph{modules,} and in particular about
locally free ones. As an application we show that in such
situation all butterflies among complexes of locally free modules
are themselves given by locally free objects, and, more
importantly, they are always \emph{locally split,} in the sense
that upon restricting to a suitable cover they split and
correspond to a strict morphism. This culminates the discussion
in section~\ref{sec:abelian-categories}.

\subsection{Organization of the paper}
\label{sec:organization-paper}

This paper is organized as follows.
Sections~\ref{sec:background-notions}
and~\ref{sec:cross-modul-their} collect a number of facts,
notions, and propositions concerning the formalism of
(hyper)covers and descent data, stackification, gr-categories,
gr-stacks, and crossed modules.  With it we have made an attempt
at making the paper somewhat self-contained and at easing the
reader's task in hunting down the various needed results from the
literature.  The idea is for the reader to refer back to them as
needed. We have broken this rule for facts concerning 2-stacks,
for which we entirely rely on the existing references, with the
possible exception of some elementary facts concerning 2-descent
data. Section~\ref{sec:cross-modul-their}, in particular, recalls
several results of ref.\ \cite{MR92m:18019}, which we have
reviewed in some detail, also due to the use of different
conventions.

In section~\ref{sec:butt-weak-morph} we define weak morphisms,
butterflies, and prove the main equivalence theorem. Then in
section~\ref{sec:fiber-bicat-cross} we describe the bicategory of
crossed modules, showing it is bi-equivalent to the 2-stack of
2-group stacks.

Sections~\ref{sec:exact-sequences-gr},
\ref{sec:abelian-butterflies}, and~\ref{sec:abelian-categories}
are devoted to applications. In
section~\ref{sec:exact-sequences-gr} we reexamine the notion of
exact sequence of 2-groups and obtain the long exact sequence in
non-abelian cohomology without the fibration assumption. We also
describe the homotopy fiber. This section requires more
background (especially on 2-stacks) than the rest of the present
work. Section~\ref{sec:abelian-butterflies} is devoted to the
various commutativity laws we can impose on a 2-group stack.  In
particular, for a braided 2-group stack, we obtain the braiding
bracket directly from the butterfly expressing the fact that the
multiplication law is a morphism of 2-groups.  We believe this is
new even in the set-theoretic case. In
section~\ref{sec:abelian-categories}, we discuss the connection
with Deligne's results, in particular theorem in ref.\
\cite{0259.14006} becomes a corollary of our main theorem in
sect.~\ref{sec:butt-weak-morph}. We terminate the discussion with
an exposition of the theory for modules in a ring site, devoting
special attention to locally split butterflies, in
section~\ref{sec:locally-split-butterflies}.

\subsection{Conventions and notations}
\label{sec:conv-notat}

We will work in the context of sheaves, stacks, etc. defined on a
site $\s$. It will be convenient to introduce the associated
topos $\cat{T}$ of sheaves on $\s$, and to say that $F$ is an
object (respectively, group object) of $\cat{T}$, rather than
specifying that $F$ is a sheaf of sets (respectively, groups) on
$\s$.  In a similar vein, we will usually adopt an ``element''
style notation by silently employing the device of identifying
objects of $\s$ with the (pre)sheaves they represent, thereby
identifying them with objects of $\cat{T}$, as per more or less
standard practice. Apart from this, we will not use the
properties of $\cat{T}$ as an abstract topos in any significant
manner.

We have tried not not make the paper dependent on specific
hypotheses on the nature of the underlying site $\s$.  We have
also tried to refrain from making cocycle-type arguments too
prominent. When we did have to run these type of arguments, we
used generalized covers and hypercovers. Using hypercovers does
not lead to a complication of the formalism, provided the right
simplicial one is used from the start.

As a general rule, objects of the underlying site $\s$ are
denoted by capital letters: $U$, $V$, etc.  Same for generalized
covers, and the various sheaves on $\s$, i.e. objects of
$\cat{T}$.  For categorical objects we use:
\begin{itemize}
\item $\cat{C}$, $\cat{D}$,... ``generic'' categories;
\item $\stx$, $\sty$, $\grg$,...  fibered categories, stacks, gr-stacks;
\item $\twocat{C}$, $\twocat{D}$,... ``generic'' 2- (or bi-)categories;
\item $\twostack{C}$, $\twostack{F}$,... 2-categories, 2-stacks,
  fibered bicategories.
\end{itemize}
Special items, such as the stack of $G$-torsors, for a group $G$,
are denoted by $\tors (G)$. Same if $\grg$ is a 2-group stack:
its 2-stack of torsors is denoted by $\tors (\grg)$.  $\prshv$
denotes the category of presheaves of sets on $\s$.

Complexes, and in particular simplicial objects, carry a bullet
for additional emphasis, so that, for example, hypercovers are
usually denoted by $U_\bullet,V_\bullet,\dotsc$ and so
on. Complexes always are cohomological, and usually placed in
\emph{negative} degrees. Except for the last section
(section~\ref{sec:abelian-categories}), this is not reflected in
the notation: for convenience, throughout most of the paper we
denote a crossed module by $G_\bullet\colon [G_1\to G_0]$.

\section{Background notions}
\label{sec:background-notions}

\subsection{Topology}
\label{sec:topology}

We will work on a fixed site $\s$, not assumed to necessarily
have fibered products.  It will be assumed that the topology on
$\s$ is subcanonical.

Recall that for an object $U$ of $\s$ a \emph{sieve} $R$ over $U$
is a collection of morphisms $i\colon V\to U$ of $\s$ which is
best described by saying that $R$ is a subfunctor of
$U=\Hom_\s(-,U)$.  A morphism $u\colon Y\to X$ in $\prshv$ is a
\emph{local epimorphism,} or a \emph{generalized cover,} if for
every morphism $U\to X$ in $\prshv$ with $U\in \Ob\s$ there
exists a covering sieve $R$ of $U$ such that for each $(i\colon
V\to U)\in R$ the composition $V\to U\to X$ lifts to $Y$.  A
generalized cover $u\colon Y\to X$ factors as
\begin{equation*}
  Y \lto \im(u) \lto X
\end{equation*}
where the first map is an epimorphism (hence a generalized cover)
and the second a local isomorphism. In particular, if $u\colon
Y\to U$, with $U\in \Ob\s$, is a generalized cover, then
$R=\im(u)$ is a sieve which is covering by definition: it is
precisely the sieve comprising morphisms $V\to U$ which lift to
$Y$ (hence the name \emph{local} epimorphims for $u$). This
correspondence allows to recast the axioms characterizing a
Grothendieck topology by reformulating them in terms of
generalized covers instead of sieves (see \cite{MR2182076} for
more details).
 
If $u\colon Y_\bullet\to X_\bullet$ is a simplicial morphism
between simplicial objects in $\prshv$, the modern point of view
is to say that $u$ is a \emph{hypercover} if all the maps
\begin{equation*}
  Y_n \lto (\cosk_{n-1}Y)_n\times_{(\cosk_{n-1}X)_n} X_n
\end{equation*}
are generalized covers (\cite{Jardine:Fields-spre}). It is shown
in ref.\ \cite{MR2034012} that this is equivalent to $u$ being a
local acyclic fibration.  More classically, following the
formally stated definition in \loccit and refs.\ \cite{MR0245577}
and \cite[Exp.\ V.\ 7]{MR0354653} one has:
\begin{definition}
  The augmented simplicial object $u\colon Y_\bullet\to U$, with
  $U\in \Ob\s$ is a \emph{hypercover} if:
  \begin{enumerate}
  \item $u$ is a local acyclic fibration in $\prshv$ ($U$ is
    regarded as a constant simplicial object), and
  \item each $Y_n$ is a coproduct of representable objects.
  \end{enumerate}
\end{definition}
One sees immediately that all maps $Y_n \to (\cosk_{n-1}Y)_n$ and
$Y_0\to U$ are local epimorphisms.

A hypercover $Y_\bullet \to U$ is \emph{bounded} or more
precisely, $p$-bounded, or a $p$-hypercover, for an integer
$p\geq 0$, if these maps are actually isomorphisms for $n\geq p$.
The \emph{\Cech covers} are the hypercovers in the sense of the
previous definition for which $p=0$, that is all maps as above
are isomorphisms. In general, for a morphism $u\colon Y\to X$ in
$\prshv$ we define the associated \emph{\Cech complex} to be the
simplicial object $\cech u$ (or $\cech_X Y$) defined by:
\begin{equation*}
  (\cech u)_n = 
  \underbrace{Y\times_X Y\times_X\dots \times_X Y}_{n+1}
\end{equation*}
Indeed, regarding $Y$ as a constant simplicial object in
$s\prshv$, we see that (\cite{MR0245577})
\begin{equation*}
  \cech u = \cosk_0 Y
\end{equation*}
Thus, upon considering a local epimorphism $u\colon Y\to U$ with
$U\in \Ob\s$, we see that a $0$-hypercover is precisely the
old-fashioned \Cech complex.

\subsection{Descent data}
\label{sec:descent-data}

We collect here a few reminders about the formalism of descent
data. We choose to formulate descent data using \Cech resolutions
of generalized covers and hypercovers. For this, one needs to define
 $\fibf (X)$ when $\fibf$ is a fibered category over $\s$ and $X$
 is an object of $\prshv$.
\begin{paragr}
  \label{prgr:1}
  Let $\fibf$ be a fibered category over $\s$.  For $X\in
  \Ob\prshv$ set
  \begin{equation*}
    \fibf (X) \eqdef \Lim_{(i\colon V\to
      X)\in (\s \downarrow X)} \fibf (V)
    \eqdef \catHom_\s(X,\fibf)
  \end{equation*}
  where $(\s \downarrow X)$ is the overcategory of objects of
  $\s$ over $X$ via the Yoneda embedding $\s\to \prshv$, and the
  right hand side is the category of morphisms of fibered
  categories. The functor $X\in\Ob\prshv$ is interpreted as a
  fibered category over $\s$ in the standard way (see e.g.\
  \cite{MR49:8992}).
\end{paragr}
For completeness let us recall the explicit form for objects and
morphisms in $\fibf (X)$, see, e.g.\ \cite{MR2182076}.  We will
not need to use the formulas in the sequel.
\begin{paragr}
  \label{prgr:3}
  An object of $\fibf (X)$ is a family $\lbrace x_i,\phi_j
  \rbrace$ of objects $x_i \in \Ob \fibf (V)$ parametrized by $(i
  \colon V\to X) \in \Ob (\s\downarrow X)$, and isomorphisms
  $\phi_j\colon j^*x_i \isoto x_{i j}$ for each
  $W\overset{j}{\to} V$, such that there is a commutative diagram
  \begin{equation*}
    \vcenter{%
      \xymatrix{%
        k^*j^*x_i \ar[r]^{k^*(\phi_j)} \ar[d]_{c_{j,k}} &
        k^*x_{ij} \ar[d]^{\phi_k} \\
        (jk)^*x_i  \ar[r]_{\phi_{jk}} & x_{ijk}
      }}
  \end{equation*}
  for any composable triple $Z\to W \to V\to X$. In the above
  diagram the vertical arrow to the left is the ``cleavage'' of
  the fibered category $\fibf$.  A morphism
  \begin{equation*}
    \lbrace x_i,\phi_j \rbrace \lto \lbrace x'_i,\phi'_j \rbrace
  \end{equation*}
  is a family $f_i\colon x_i\to x'_i$ for each $(i\colon V\to
  X)\in (\s\downarrow X)$ such that for any $j\colon W\to V$ the
  diagram
  \begin{equation*}
    \vcenter{%
      \xymatrix{%
        j^*x_i \ar[r]^{\phi_j} \ar[d]_{j^*(f_i)} & x_{ij}
        \ar[d]^{f_{ij}} \\
        j^*x'_i \ar[r]_{\phi'_j} & x'_{ij}
      }
    }
  \end{equation*}
  commutes.
\end{paragr}

For a hypercover $u\colon Y_\bullet\to U$, we define:
\begin{definition}[{\cite[\S 10]{MR0245577}}]
  \label{def:1}
  A \emph{descent datum} for $\fibf$ over $U$ relative to $u$ is
  given by an object $x\in \fibf (Y_0)$ and an isomorphism
  $\phi\colon d^*_0x\isoto d^*_1x$ in $\fibf (Y_1)$ satisfying
  the cocycle condition:
  \begin{equation*}
    d_1^*\phi = d_2^*\phi \circ d_0^*\phi\,.
  \end{equation*}
\end{definition}
It turns out that for a \emph{pre-stack} $p\colon\stx\to \s$ the
categories of descent data with respect to hypercovers are
equivalent to those determined by their $0$-coskeleta, that is,
\Cech covers.  This is explicitly proved in \cite[Proposition
10.3]{MR0245577} when $\s$ has fiber products, but the argument
goes through for generalized covers as well, or it follows more
generally from the results of \cite{MR2034012}.  On the other
hand, if working simplicially there really is no additional
complication in working with hypercovers---even the notation
would be the same.
\begin{paragr}
  \label{prgr:4}
  It is more appropriate to talk about the \emph{category} of
  descent data---the notion of \emph{morphism} between descent
  data $(x,\phi)$ and $(x',\phi')$ being defined by a morphism
  $\psi\colon x\to x'$ in $\fibf (Y_0)$ such that
  \begin{equation}
    \label{eq:3}
    d_1^*\psi\, \phi = \phi'\, d_0^*\psi
  \end{equation}
  in $\fibf (Y_1)$.  Let us denote by $\Desc(u,\fibf)$ the
  category of descent data for $\fibf$ relative to the hypercover
  $u\colon Y_\bullet \to U$.
\end{paragr}

\subsection{2-Descent data}
\label{sec:2-descent}

We will need to discuss the analog of the descent condition for
fibered 2-categories (or even bicategories), a concept for which
we refer to ref.\ \cite{MR0364245} (see also \cite[Chapter
1]{MR95m:18006}).
\begin{paragr}
  \label{prgr:5}
  Let $\twofib{C}$ be a fibered 2-category, and let $X\in
  \prshv$. Define, analogously to~\ref{prgr:1}
  \begin{equation*}
    \twofib{C} (X) \eqdef \twocatHom_{\s} (X,\twofib{C}).
  \end{equation*}
  where $X$ is interpreted as a fibered 2-category. This also
  equals $\Lim_{(i\colon V\to X)\in (\s \downarrow X)} \twofib{C}
  (V)$.  (See ref.\ \cite{MR0364245} for details.)
\end{paragr}

\begin{definition}
  \label{def:18}
  Let $\twofib{C}$ be a fibered 2-category. Let $U\in \Ob \s$,
  and $u\colon Y_\bullet\to U$ a hypercover. A \emph{2-descent
    datum} over $U$ relative to $u$ is given by an object $x\in
  \Ob \twofib{C} (Y_0)$, an isomorphism $\phi\colon d^*_0x\isoto
  d^*_1x$ in $\twofib{C} (Y_1)$, and a 2-morphism
  \begin{equation*}
    \alpha \colon 
    d_1^*\phi \lto d_2^*\phi \circ d_0^*\phi,
  \end{equation*}
  over $Y_2$, satisfying the cocycle condition visualized by the diagram:
  \begin{equation*}
    \bigl( (d_2d_3)^*\phi * d_0^*\alpha\bigr)
    \circ d_2^*\alpha =
    \bigl( d_3^*\alpha * (d_0d_1)^*\phi\bigr)
    \circ d_1^*\alpha
  \end{equation*}
  over $Y_3$.
\end{definition}

\subsection{Stack associated to a prestack}
\label{sec:stack-assoc-prest}

Recall that for any prestack $\stx$ there is canonically
associated a stack $\stx\sptilde$ and a morphism (the
``stackification'') $a\colon \stx \to \stx\sptilde$, such that
for every morphism (of prestacks) $F\colon \stx\to \sty$ to a
stack $\sty$ there is a factorization
\begin{equation*}
  \vcenter{%
    \xymatrix{%
      \stx \ar[r]^a \ar@/_/[dr]_F \drtwocell<\omit>{<-1>}
      & **[r]\stx\sptilde \ar@{.>}[d]^{F^a}  \\
      & \sty 
    }
  }
\end{equation*}
$\stx\sptilde$ is determined up to equivalence. The previous
diagram expresses the universal property of the associated stack.

There are explicit constructions of $\stx\sptilde$, which involve
``adding descent data.'' Given $\stx$, one defines
\begin{equation*}
  \stx^+ (U) = \colim_{Y\to U} \Desc ( \cech Y\to U, \stx),
\end{equation*}
where $U$ is an object of $\s$.  This leads to the explicit
description as found e.g.\ in \cite{MR1771927}: an object $(Y\to
U, x, \phi)$ of $\stx^+$ over $U$ comprises a generalized cover
and a descent datum relative to it.  A morphism
\begin{equation*}
  (Y\to U, x, \phi)\lto (Y'\to U, x', \phi')
\end{equation*}
is a morphism of descent data over $Y\times_U Y'\to U$.
Equivalently, one could use homotopy classes of hypercovers. Our
main example will be the gr-stack associated to a crossed module,
for which there exists an explicit model (see below).
\begin{theorem}[\cite{MR49:8992,MR1771927}]
  If $\stx$ is a prestack over $\s$, then $\stx^+$ is a stack.
\end{theorem}
There is an obvious morphism $\stx\to \stx^+$ which consists in
sending the object $x$ over $U$ to $(\id\colon U\to U, x, \id)$.
This allows us to take $\stx\sptilde$ to be $\stx^+$ and the
morphism $a\colon \stx\to \stx^+$ to be the one just described.

\section{Recollections on gr-categories, gr-stacks, and crossed
  modules}
\label{sec:cross-modul-their}

\subsection{Gr-categories and gr-stacks}
\label{sec:gr-categories}
The reference for gr-categories is the not easily accessible
thesis \cite{sinh:gr-cats} (see also \cite{MR0338002}). The basic
facts are recalled in ref.\ \cite{MR93k:18019}, which we follow
for terminology and conventions (see also \cite{MR2072344} and
\cite{MR1935985}).
\begin{paragr}
  A 2-group, or gr-category, is a monoidal, unital, compact
  groupoid, that is a groupoid $\cat{C}$ equipped with a
  composition law, a unit object $I$, and for each object $X\in
  \cat{C}$ a choice of (right) inverse $X^*$, respectively.  The
  composition law is a functor
  \begin{equation*}
    \otimes \colon \cat{C} \times \cat{C} \lto \cat{C}
  \end{equation*}
  obeying an associativity constraint: for each triple $X,Y,Z\in
  \Ob\cat{C}$ there is a functorial isomorphism (the associator)
  \begin{equation*}
    a_{X,Y,Z} \colon (X\otimes Y)\otimes Z \lisoto
    X\otimes (Y\otimes Z)
  \end{equation*}
  required to satisfy a coherence condition expressed by the
  well-know Mac~Lane's pentagon diagram.  Furthermore, for each
  object $X\in \Ob\cat{C}$ there are functorial isomorphisms
  \begin{equation}
    \label{eq:67}
    l_X\colon X \isoto I\otimes X\,,\quad
    r_X\colon X\isoto X\otimes I
  \end{equation}
  required to satisfy the compatibility diagram
  \begin{equation}
    \label{eq:5}
    \vcenter{%
      \xymatrix@C-2pc@R-0.5pc{%
        (X\otimes I) \otimes Y \ar[rr]
        & & X\otimes ( I \otimes Y) \\
        & X \otimes Y \ar@/^0.5pc/[ul] \ar@/_0.5pc/[ur] 
      }}
  \end{equation}
  The inverse map $X\mapsto X^*$ is a functor
  $\cat{C}^{\mathrm{op}}\to \cat{C}$, and one has $(X\otimes Y)^*
  \isoto Y^*\otimes X^*$.  Furthermore, there is an isomorphism
  \begin{equation*}
    X\otimes X^*\isoto I\,.
  \end{equation*}
  The choice of the latter determines the arrow $I\isoto X^*
  \otimes X$. For all the remaining properties, as well as the
  compatibility diagrams not displayed here, we refer to the
  above mentioned works.
\end{paragr}
\begin{paragr}
  To a gr-category $\cat{C}$ are associated its group of
  isomorphism classes of objects, $\pi_0(\cat{C}) = \Ob
  \cat{C}/\sim$, and the group of automorphisms of the identity
  object, $\pi_1(\cat{C}) = \Aut(I)$.  The latter is an abelian
  group owing to the fact that for any object $x\in \Ob\cat{C}$,
  left (say) multiplication by $x$ is an equivalence of
  $\cat{C}$, hence it allows to coherently identify $\Aut (x)$
  with $\Aut (I)$. This implies abelianness (cf.\
  \cite{sinh:gr-cats,MR1702420}).  One also has that
  $\pi_1(\cat{C})$ carries a (right) $\pi_0(\cat{C})$-action,
  induced by right multiplication by objects of $\cat{C}$.
\end{paragr}
Let now $\cat{C}$ and $\cat{D}$ be two gr-categories.
\begin{paragr}
  An additive functor from $\cat{C}$ to $\cat{D}$ is a pair
  $(F,\lambda)$, where
  \begin{equation*}
    F\colon \cat{C} \lto \cat{D}
  \end{equation*}
  is a functor between the underlying groupoids, and for each
  pair of objects $X,Y\in \Ob\cat{C}$ there is a functorial
  (iso)morphism
  \begin{equation*}
    \lambda_{X,Y}\colon F(X)\otimes F(Y) \lisoto F(XY)\,.
  \end{equation*}
  (Since in a gr-category the multiplication functor by every
  object is an equivalence, the condition $I\isoto F(I)$ follows
  from the existence of $\lambda$, cf.\ the above quoted
  references.)  The isomorphisms $\lambda$ must be compatible
  with the associativity morphism, in the sense that the
  following diagram must commute:
  \begin{equation}
    \label{eq:6}
    \vcenter{%
      \xymatrix@-1ex{%
        (F(X)\otimes F(Y))\otimes F(Z) \ar[d] \ar[r] &
        F(X\otimes Y)\otimes F(Z) \ar[r] &
        F((X\otimes Y)\otimes Z) \ar[d] \\
        F(X)\otimes (F(Y)\otimes F(Z))  \ar[r] &
        F(X) \otimes F(Y\otimes Z) \ar[r] &
        F(X\otimes (Y\otimes Z))}}
  \end{equation}
  The diagrams resulting from the compatibility between the
  isomorphism $I\isoto F(I)$ and $l_X$ and $r_X$ (for any object
  $X$) must commute as well.
\end{paragr}
\begin{paragr}
  A natural transformation of additive functors $(F,\lambda)$ and
  $(G,\mu)$ consists of a natural transformation of the
  underlying functors $\theta \colon F \Rightarrow G$ in the
  standard sense, such that the diagram
  \begin{equation}
    \label{eq:76}
    \vcenter{%
      \xymatrix{%
        F(X)\otimes F(Y) \ar[r]^(.6){\lambda_{X,Y}}
        \ar[d]_{\theta_{X}\otimes \theta_{Y}} & F(XY) 
        \ar[d]^{\theta_{XY}} \\
        G(X)\otimes G(Y) \ar[r]_(.6){\mu_{X,Y}} & G(XY)
      }}
  \end{equation}
  commutes. (Note that the diagram expressing the compatibility
  between $\theta$ and the associators obtained by combining the
  previous two diagrams is automatically commutative.)  
\end{paragr}
\begin{paragr}
  There is a canonical way of composing additive
  functors.  The composition of
  $(F_1,\lambda^1)\colon \cat{C}_0\to \cat{C}_1$ and
  $(F_2,\lambda^2)\colon \cat{C}_1\to \cat{C}_2$ is $F_2\circ
  F_1$ equipped with $\lambda^2*\lambda^1$ given by
  \begin{equation*}
    {\lambda^2*\lambda^1}_{X,Y} =
    F_2({\lambda^1}_{X,Y})\circ {\lambda^2}_{F_1(X),F_1(Y)}.
  \end{equation*}
  Note that this composition \emph{is} associative.
\end{paragr}
\begin{paragr}
  The preceding constructions carry over to the sheaf-theoretic
  context.  Suppose we are given a stack $\grg$ in groupoids on
  the site $\s$.  Following ref.\ \cite{MR93k:18019}, we will say
  that $\grg$ is a 2-group stack, or a gr-stack, if, again, it is
  equipped with a composition law embodied by morphisms of stacks
  \begin{equation*}
    \otimes \colon \grg\times \grg \lto \grg
  \end{equation*}
  and
  \begin{equation*}
    (\cdot)^*\colon \grg \lto \grg \,, \quad x \lmto x^*
  \end{equation*}
  plus a global identity object $I$.  These data will be required
  to satisfy the same formal properties as for a gr-category.  A
  morphism $F\colon \grg \to \grh$ of gr-stack is actually an
  additive functor, that is, a pair $(F,\lambda)$, where the
  underlying functor $F$ is a morphism of stacks.  Again, $F$ and
  $\lambda$ are required to satisfy the same properties listed
  for gr-categories.  The same definitions hold with the word
  ``stack'' replaced by ``pre-stack.''  Our main examples of
  gr-(pre)stacks will arise from crossed-modules, whose main
  definitions and properties we are going to recall below.
\end{paragr}
\begin{paragr}
  For $\grg$ a gr-stack we define $\pi_0(\grg)$ (or simply
  $\pi_0$, for short, when no danger of confusion can arise) to
  be the sheaf associated to the presheaf
  \begin{equation*}
    U\rightsquigarrow \pi_0(\grg_U),
  \end{equation*}
  where $U$ is an object of the underlying site, so that
  $\pi_0(\grg_U)$ is the group of isomorphism classes of objects
  of the gr-(fiber)-category $\grg_U$ over $U$. It is known (and
  easy to see, cf.\ refs.\ \cite{MR1771927,MR95m:18006}) that the
  projection
  \begin{equation*}
    \grg \lto \pi_0
  \end{equation*}
  makes $\grg$ a gerbe over $\pi_0$.  We also set $\pi_1(\grg) =
  \shAut(I)$ (or simply $\pi_1$ when possible), the sheaf of
  automorphisms of the identity object.  The coherence argument
  mentioned above remains valid in this case, implying that
  $\pi_1$ is a sheaf of abelian groups.  Moreover, it is the band
  of the gerbe $\grg\vert_{\pi_0}$ (\cite{MR95m:18006}).
\end{paragr}

\subsection{Crossed modules}
\label{sec:crossed-modules}

The notion of crossed module is of course by now well known.  We
will recall the main definitions here to merely establish the
necessary conventions.  Following ref.\ \cite{0259.14006}, a
crossed module will be considered as a complex
\begin{equation*}
  G^\bullet  \colon [ G^{-1} \overset{\del}{\lto} G^0 ]_{-1,0}
\end{equation*}
placed in (cohomological) degrees $-1,0$.  For notational
convenience, we will use a homological (subscript indices)
notation via the standard re-indexing $G_i=G^{-i}$.  (All actions
to be considered in this paper will be on the right, and crossed
modules will be no exception.)
\begin{definition}
  \label{def:2}
  A crossed module in $\cat{T}$ is a homomorphism of group
  objects
  \begin{equation*}
    \del \colon G_1\lto G_0
  \end{equation*}
  together with a right action
  \begin{equation*}
    G_1\times G_0\lto G_1
  \end{equation*}
  written as $(g,x)\mapsto g^x $ in set-theoretic terms, for
  $g\in G_1$ and $x\in G_0$, satisfying
  \begin{equation}
    \label{eq:7}
    \begin{aligned}
      \del ( g^x ) & = x^{-1}\, \del (g)\, x,\\
      {g_0}^{\del (g_1)} & = g_1^{-1}\, g_0\, g_1\,,
    \end{aligned}
  \end{equation}
  for $x\in G_0$ and $g,g_0,g_1\in G_1$.
\end{definition}
\begin{remark}
  \label{rem:1}
  The use of set-theoretic element-notation in eqns.~\eqref{eq:7}
  can of course be avoided. The axioms can be written in a purely
  arrow-theoretic way:
  \begin{equation}
    \label{eq:8}
    \vcenter{%
      \xymatrix{%
        G_1\times G_0 \ar[r]
        \ar[d]_{\id \times \del} &
        G_1 \ar[d]^\del \\
        G_0 \times G_0
        \ar[r] & G_0}
    } \qquad
    \vcenter{%
      \xymatrix@-0.5pc{%
        & G_0 \ar@/^/[dr]^\jmath   & \\
        G_1 \ar@/^/[ur]^\del \ar[rr]_{\imath_{G_1}}
        && \Aut(G_1)}
    }
  \end{equation}
  In the diagram to the left the top horizontal arrow is simply
  the action of $G_0$ on $G_1$, whereas the bottom one
  corresponds to the (right) action of $G_0$ on itself given by
  conjugation.  In the diagram to the right $\jmath$ is the
  morphism corresponding to the action of $G_0$ on $G_1$, and
  $i_{G_1}$ is homomorphisms given by the inner conjugation,
  namely $g\mapsto \imath_g\colon g' \mapsto g^{-1}g'g$.
\end{remark}
\begin{definition}
  \label{def:3}
  A \emph{strict} morphism of crossed modules is a diagram
  \begin{equation}
    \label{eq:9}
    \vcenter{%
      \xymatrix{%
        H_1 \ar[r]^{f_1} \ar[d]_{\del_G} & G_1 \ar[d]^{\del_H} \\
        H_0 \ar[r]_{f_0} & G_0 
      }}
  \end{equation}
  of group homomorphisms, where the columns are crossed modules,
  and $f_1$ is $f_0$-equivariant, that is:
  \begin{equation}
    \label{eq:10}
    f_1 (h^x) = f_1 (h)^{f_0(x)}
  \end{equation}
  for $h\in H_1$, $x\in H_0$.
\end{definition}
As the usage of the qualifier ``strict'' in the previous
definition suggests, there exist also \emph{weak} morphisms,
where conditions \eqref{eq:9} and \eqref{eq:10} are substantially
relaxed. They are defined to be simply additive functors between
the corresponding gr-categories, cf.\ refs.\
\cite{Noohi:notes,Noohi:weakmaps}.  They will be treated in
detail in a later section, from a rather different perspective
than the one adopted in \loccit

As one may expect, there are also morphisms between (strict)
morphisms (i.e. natural transformations), defined as follows.
\begin{definition}
  \label{def:4}
  Given two morphisms $f$, $f'$ as in \eqref{eq:9}, a
  \emph{homotopy} $\gamma \colon f \Rightarrow f'$ between them
  is a map
  \begin{equation}
    \label{eq:11}
    \gamma \colon H_0 \longrightarrow G_1
  \end{equation}
  satisfying the following relations:
  \begin{subequations}
    \label{eq:12}
    \begin{align}
      \label{eq:13}
      f_0 (x)\, \del_G(\gamma_x) & = f'_0 (x) \\
      \label{eq:14}
      \gamma_x\, f'_1(h) &= f_1(h)\, \gamma_y \\
      \label{eq:15}
      \gamma_{xx'} & = \gamma_x^{f_0(x')} \, \gamma_{x'}
    \end{align}
  \end{subequations}
  for all $h\in H_1$, $x,y\in H_0$ such that $x\, \del_H h = y$,
  and $x'\in H_0$.
\end{definition}
\begin{paragr}
  A crossed module gives rise to a groupoid
  \begin{equation*}
    \grpd{G}\colon
    \xymatrix@1{%
      G_0\times G_1 \ar@<0.3pc>[r]^(.6)s \ar@<-0.1pc>[r]_(.6)t & G_0
    }
  \end{equation*}
  where the source and target maps are:
  \begin{equation}
    \label{eq:16}
    s(x, g)=x\,,\quad t(x, g) = x\,\del(g)\,,
  \end{equation}
  where $x\in G_0$, $g\in G_1$.  This groupoid is in fact a
  strict gr-category, where the composition functor
  \begin{equation*}
    \otimes \colon \grpd{G}\times \grpd{G} \lto \grpd{G}
  \end{equation*}
  is given on objects (i.e.\ $G_0$) by the group law of $G_0$,
  and on morphisms by:
  \begin{equation}
    \label{eq:17}
    (x_0,g_0)\otimes (x_1,g_1) = (x_0 x_1, g_0^{x_1} g_1)\,,
  \end{equation}
  with obvious meaning of the variables.  It is easy to verify
  that this group law is strictly associative.  It is also easy
  to check that a strict morphism in the sense of
  Definition~\ref{def:3} gives an additive functor $F\colon
  \grpd{G}\to \grpd{H}$. Note that this functor will be additive
  in the strictest possible sense, namely all the isomorphism
  $\lambda_{x,y}$ are the identity. Finally, a homotopy as in
  Definition~\ref{def:4} gives rise to a morphism of such
  additive functors $F\Rightarrow F'$.
\end{paragr}
\begin{paragr}
  A crossed module $[\del\colon G_1\to G_0]$ gives rise to an
  obvious exact sequence
  \begin{equation*}
    0\lto A\lto G_1\overset{\del}{\lto}G_0\lto B\lto 1\,,
  \end{equation*}
  where $A=\Ker \del$ and $B=\Coker \del$. It is immediately
  verified that $B=\pi_0(\grpd{G})$, and $A = \pi_1(\grpd{G})$.
  It follows from the more general considerations about
  gr-categories, or from direct computations with~\eqref{eq:7},
  that $A$ is a $B$-module, and it is central in $G_1$, hence
  abelian.
\end{paragr}

\subsection{Cocycles}
\label{sec:cocycles}

To a crossed module $[G_1\to G_0]\sptilde$ there is a canonically
associated simplicial group object of $\cat{T}$, namely the nerve
of the groupoid $\grpd{G}\colon G_0\times G_1\rightrightarrows
G_0$. It is well known that this simplicial group, which we
denote $\smp{G}_\bullet$, is given by
\begin{equation*}
  \smp{G}_0 = G_0,
  \qquad
  \smp{G}_n = G_0\times \underbrace{G_1\times \dotsb \times
    G_1}_{n}\,, \quad n\geq 1.
\end{equation*}
with face and degeneracy maps $d_i \colon
\smp{G}_n\to \smp{G}_{n-1}$ and $s_i \colon \smp{G}_n\to
\smp{G}_{n+1}$:
\begin{align*}
  d_i (x, g_0, \dotsc, g_{n-1}) & =
  \begin{cases}
    (x\, \del g_0, g_1, \dotsc, g_{n-1}) & i=0\\
    (x, g_0, \dotsc, g_{i-1}g_{i}, \dotsc, g_{n-1}) & 0 < i < n\\
    (x, g_0, \dotsc, g_{n-2}) & i=n
  \end{cases} \\
  s_i (x, g_0, \dotsc, g_{n-1}) & = (x, g_0, \dotsc,
  g_{i-1},1,g_{i},\dotsc, g_{n-1}),\quad i=0,\dotsc,n.
\end{align*}
\begin{definition}
  \label{def:5}
  Let $Y_\bullet\to U$ be a hypercover of $\prshv$. A
  \emph{$0$-cocycle over $U$} is a simplicial map $\xi\colon
  Y_\bullet\to \smp{G}_\bullet$.  Two such cocycles $\xi,\xi'$
  are \emph{equivalent} if there is a simplicial homotopy $\alpha
  \colon \xi \Rightarrow \xi'\colon Y_\bullet \to
  \smp{G}_\bullet$.
\end{definition}
\begin{paragr}
  Computing with simplicial maps and simplicial identities, and
  the above definition of $\smp{G}_\bullet$, shows that there is
  a one-to-one correspondence between such simplicial maps
  $\xi\colon Y_\bullet \to \smp{G}_\bullet$ and pairs $(x,g)$,
  $x\colon Y_0\to G_0$ and $g\colon Y_1\to G_1$, satisfying
  \begin{equation}
    \label{eq:18}
    \begin{aligned}
      d_0^*x  & = d_1^*x\, \del g\\
      d_1^*g & = d_2^*g\, d_0^*g
    \end{aligned}
  \end{equation}
  and the normalization condition $s_0^*g =1$.  The simplicial
  map $\xi$ itself is given by:
  \begin{align*}
    \xi_0 & = x\\
    \xi_1 & = ( d_1^*x,g) \\
    \xi_2 & = ( (d_1d_2)^*x, d_2^*g, d_0^*g).
  \end{align*}
  (Like the nerve of any category $\smp{G}_\bullet$ is
  2-coskeletal, hence $\xi$ is completely determined by its
  2-truncation.) To express the correspondence between $\xi$ and
  $(x,g)$ we will simply write $\xi = (x,g)$.

  Similarly, another direct calculation using the definitions
  reveals that a simplicial homotopy $\alpha\colon \xi
  \Rightarrow \xi'$ is uniquely determined by an element $a\colon
  Y_0\to G_1$ such that
  \begin{equation}
    \label{eq:19}
    \begin{aligned}
      x' & = x\, \del a \\
      g\,d_0^*a & = d_1^*a \,g'.
    \end{aligned}
  \end{equation}
  According to the classical formulas found, e.g.\ in
  \cite{MR1206474} the homotopy $\alpha$ is concretely realized
  as a simplicial homotopy \emph{from $\xi'$ to $\xi$} in the
  sense of \loccit \S 5, and it consists of maps
  $\alpha^0_0\colon Y_0\to G_0\times G_1$ and
  $\alpha^1_0,\alpha^1_1\colon Y_1\to G_0\times G_1\times G_1$
  given by:
  \begin{align*}
    \alpha^0_0 & = (x, a) \\
    \alpha^1_0 & = (d_1^*x, d_1^*a, g') \\
    \alpha^1_0 & = (d_1^*x, g, d_0^*a).
  \end{align*}
\end{paragr}

\subsection{Gr-stacks associated to crossed modules}
\label{sec:gr-stacks-associated}

A sheaf of groupoids is in an obvious way a prestack
(\cite{MR1771927}).  Given a crossed module $[G_1 \to G_0]$ and
the groupoid
\begin{math}
  \grpd{G}\colon G_0\times G_1 \rightrightarrows G_0,
\end{math}
We usually indicate by
\begin{math}
  [G_1\to G_0]\sptilde
\end{math}
(rather than $\grpd{G}\sptilde$) the associated stack.  In
general we have:
\begin{proposition}
  \label{prop:2}
  If $\grg$ is a gr-prestack, then the associated stack
  $\grg\sptilde$ acquires the structure of gr-stack, and the
  stackification morphism $a\colon \grg\to \grg\sptilde$ becomes
  an additive functor.
\end{proposition}
\begin{proof}[Idea of proof]
  This can be seen by applying the diagram expressing the
  universal property of the associated stack at the beginning
  of~\ref{sec:stack-assoc-prest} to the morphism $\grg \times
  \grg \to \grg\to \grg\sptilde$ to obtain $\otimes \colon
  \grg\sptilde\times \grg\sptilde\to \grg\sptilde$; and similarly
  for the other diagrams expressing the associativity and
  inversion laws.
\end{proof}
It follows that $[G_1\to G_0]\sptilde$ is a gr-stack---the
associated \emph{gr}-stack to the crossed module $[G_1\to G_0]$.
Its gr-stack structure can be explicitly described in terms of
descent data.

Using Definition~\ref{def:1}, the maps in~\eqref{eq:16},
and eq.~\eqref{eq:18}, we see that in the present case descent
data just become cocycles with values in $[G_1\to G_0]$. (The
correspondence being $(x,y)\to (x, g^{-1})$, to be precise.)
Similarly, from~\eqref{eq:16} and~\eqref{eq:3} it follows that
morphisms of descent data correspond to their respective cocycles
being equivalent in the sense of~\ref{def:5}.
\begin{remark}
  Whenever the site $\s$ admits fiber products, and the topology
  on $\s$ is given in terms of covers, the cocycle
  relations~\eqref{eq:18} take the more familiar form
  \begin{equation*}
    \begin{aligned}
      x_j & = x_i \, \del g_{ij} \\
      g_{ik} & = g_{ij}\, g_{jk}\,,
    \end{aligned}
  \end{equation*}
  with respect to a cover $\lbrace U_i\to U\rbrace_{i\in I}$, see
  e.g.\ \cite[2.4.5.1, 2.4.5.2]{MR93k:18019}. In the same way,
  \eqref{eq:19} become:
  \begin{equation*}
    \begin{aligned}
      x'_i & = x_i\, \del a_i \\
      g_{ij}\, a_{j} & = a_i \, g'_{ij}
    \end{aligned}
  \end{equation*}
  expressing the familiar equivalence relation between cocycles
  with values in $[G_1\to G_0]$ over $U$, see \loccit In general
  one must take care that $Y=\sqcup_i U_i$ in $\prshv$ and
  similarly that $U_{ij}\coloneq U_i\times_UU_j$ exists a priori
  only in $\prshv$ as well, so $x_i$ and $g_{ij}$ should properly
  interpreted as morphisms $x_i\colon U_i\to G_0$ and
  $g_{ij}\colon U_{ij}\to G_1$ in $\cat{T}$.
\end{remark}
\begin{paragr}
  Now, given two cocycles $\xi,\xi'\colon Y_\bullet\to
  \smp{G}_\bullet$ there is an obvious definition of
  \begin{equation*}
    \xi\otimes \xi'\colon Y_\bullet \lto \smp{G}_\bullet
  \end{equation*}
  by ``pointwise'' multiplication using the simplicial group
  structure of $\smp{G}_\bullet$:
  \begin{equation}
    \label{eq:20}
    (\xi\otimes \xi')_n \coloneq \xi_n\xi'_n.
  \end{equation}
  Computing with~\eqref{eq:17}, we find that if $\xi = (x,g)$ and
  $\xi' = (x',g')$, then
  \begin{equation}
    \label{eq:21}
    \xi \otimes \xi' = ( xx', g^{d_1^*x}g' ).
  \end{equation}
  The unit is $(1,1)$ and inverse maps will be the obvious one
  computed from~\eqref{eq:21}.
\end{paragr}
It follows from the definitions in
section~\ref{sec:stack-assoc-prest} that objects of $[G_1\to
G_0]\sptilde$ over $U\in \Ob \s$ are pairs $X=(Y,\xi)$ where
$Y\to U$ is a generalized cover and $\xi \colon \cech Y\to
\smp{G}_\bullet$.  A better way would probably be to visualize
them as a fraction
\begin{equation*}
  X = \vcenter{
    \xymatrix@-1em{%
      & \cech Y \ar@/_{.2pc}/[dl] \ar@/^{.2pc}/[dr]^{\xi}& \\
      U     &   &  \smp{G}_\bullet
    }}
\end{equation*}
\begin{paragr}
  Given two such objects $X=(Y,\xi)$ and $X'=(Y',\xi')$ over $U$
  we define their product as:
  \begin{equation}
    \label{eq:22}
    X\otimes X' = \bigl( Y\times_U Y', p^*\xi\otimes p'^*\xi')
  \end{equation}
  where $p^*\xi$ is the pull-back of $\xi$ to $\cech
  (Y\times_UY')$ via $p\colon Y\times_UY'\to Y$, and similarly
  for $p'^*\xi'$. The $\otimes$-product in the right-hand side
  of~\eqref{eq:22} is the one computed via~\eqref{eq:20}.
  Considering that the simplicial map $\xi$ is itself determined
  by the pair $(x,g)$, we can just write the object $X$ as
  $X=(Y,x,g)$, where now $x\colon Y\to G_0$ and $g\colon
  Y\times_UY'\to G_1$. This is just the classical way to write
  descent data. Therefore given $(Y,x,g)$ and $(Y',x',g')$
  objects of $[G_1\to G_0]\sptilde$ over $U$, we can simply
  write~\eqref{eq:22} more classically as:
  \begin{equation}
    \label{eq:23}
    (Y,x,g)\otimes (Y',x',g') =
    (Z, x x', g^{d_0^*(x')}g'),
  \end{equation}
  where $Z\to U$ refines both $Y,Y'$, e.g. $Z=Y\times_U Y'$, in
  $\prshv$ and for simplicity on the right hand side we have
  suppressed the pullbacks to $Z$.  Similarly, if morphisms
  $(Y,x,g)\to (Y_1,x_1,g_1)$ and $(Y',x',g')\to (Y'_1,x'_1,g'_1)$
  are given by ``elements'' $a\colon Z\to G_1$ and $a'\colon
  Z'\to G_1$ as in~\eqref{eq:19}, their product is given by
  $a^{x'}a'$ over a refinement $W$ of $Z, Z'$.
\end{paragr}
The reader will be able to verify without difficulty:
\begin{proposition}[\cite{MR92m:18019}]
  \label{prop:3}
  The product~\eqref{eq:22} gives $[G_1\to G_0]\sptilde$ the
  structure of a gr-stack.
\end{proposition}
\begin{remark}
  Note that the group law on $[G_1\to G_0]\sptilde$ just
  introduced is \emph{not} strict, even though the one on
  $\grpd{G}$ is, due to the various pullbacks. Thus, for example,
  there will be an associativity morphism
  \begin{equation*}
    ((Y,x,g)\otimes (Y',x',g'))\otimes (Y'',x'',g'') \iso
    (Y,x,g)\otimes ((Y',x',g')\otimes (Y'',x'',g'') )
  \end{equation*}
  resulting from $(Y\times_U Y')\times_U Y''$ being isomorphic to
  $Y\times_U (Y'\times_U Y'')$.
\end{remark}
\begin{paragr}
  There is an equivalent but more geometric realization of the
  associated gr-stack of $[G_1\to G_0]$. Let $\delta\colon G\to
  H$ be a group homomorphism of $\cat{T}$. Following ref.\
  \cite{MR546620}, let us denote by $\tors(G,H)$ the stack of
  right $G$-torsors equipped with a trivialization of their
  extension to $H$-torsors. In other words an object of
  $\tors(G,H)$ is a pair $(P,s)$ where $P$ is a right $G$-torsor
  and $s$ is global isomorphism $s\colon P\cprod{G} H\isoto H$.
  An object of $\tors(G,H)$ will be called a $(G,H)$-torsor. The
  morphism $s$ will be identified with a $G$-equivariant morphism
  $s\colon P\to H$ where all the actions are on the right, namely
  $s(u g) = s(u)\delta (g)$.  With this convention, the precise
  correspondence
  \begin{align*}
    \shHom_G(P,H) &\lisoto P\cprod{G}H \\
    (P,s)&\longmapsto (u, s(u)^{-1})
  \end{align*}
  in set-theoretic notation.

  A morphism $f\colon (P,s)\to (Q,t)$ is a morphism $f\colon P\to
  Q$ of $G$-torsors compatible with the trivializations.
  Equivalently, the diagram
  \begin{equation}
    \label{eq:24}
    \vcenter{%
      \xymatrix@-.75pc{%
        P \ar[rr]^f \ar@/_/[dr]_s && Q \ar@/^/[dl]^t \\ & H &}}
  \end{equation}
  commutes.
\end{paragr}
\begin{paragr}
  All this becomes much more interesting when it is applied to
  the group homomorphism underlying a crossed module $[\del
  \colon G_1\to G_0]$.  It is shown in \cite{MR92m:18019} that in
  this situation each object $(P,s)$ of $\tors(G_1,G_0)$ is in
  fact a $G_1$-bitorsor with the left $G_1$-action defined
  (set-theoretically) by:
  \begin{equation*}
    g\star u = u\, g^{s(u)}.
  \end{equation*}
  As a consequence, posing, as in \loccit
  \begin{equation}
    \label{eq:25}
    (P,s)\otimes (Q,t) \coloneq (P\cprod{G_1} Q, s\wedge t)
  \end{equation}
  endows $\tors(G_1,G_0)$ with a gr-stack structure. Here
  $s\wedge t$ is the $G$-equivariant map from $P\cprod{G_1} Q$ to
  $G_0$ given by $s(u)t(v)$, where $(u,v)$ represents a point of
  $P\cprod{G_1} Q$.
\end{paragr}
Moreover, we have:
\begin{theorem}[{\cite[Théorème 4.6]{MR92m:18019}}]
  \label{thm:1}
  There is an equivalence of gr-stacks
  \begin{equation*}
    \tors(G_1,G_0)\lisoto [G_1\to G_0]\sptilde .
  \end{equation*}
\end{theorem}
\begin{proof}[Proof (Sketch)]
  We limit ourselves to an outline the argument leading to the
  equivalence of the product structures, referring to the
  original reference for the complete details.
  
  By the argument of \loccit an object $(P,s)$ of
  $\tors(G_1,G_0)$ determines descent data in the usual way. Let
  $Y\to *$ be a generalized cover of the terminal object $*\in
  \cat{T}$ with a trivialization of the underlying
  right-$G_1$-torsor $P$ via the section $u\colon Y\to P$.  These
  data determine an isomorphism of $(G_1,G_0)$-torsors, therefore
  the morphism $\phi\colon d_0^*P_Y\isoto d_1^*P_Y$ must also
  satisfy $s(d_0^*) = s(\phi(d_0^*u))$.  This determines $g\colon
  Y\times Y\to G_1$ such that $\phi (d_0^*u) = (d_1^*u)g$ and
  $x\coin s(u)\colon Y\to G_0$ such that~\eqref{eq:18} are
  satisfied.

  Assuming for convenience that $(P,s)$ and $(P',s')$ are
  trivialized over the same $Y\to *$ from~\eqref{eq:25} we have
  that
  \begin{equation*}
    \phi (d_0^*u)\wedge \phi' (d_0^*u')
    = (d_1^*u) g\wedge (d_1^*u') g'
    = (d_1^*u) \wedge g \star (d_1^*u') g'
  \end{equation*}
  and using the form of the left action given above
  \begin{equation*}
    g \star (d_1^*u') g' =  (d_1^*u') g^{s'(d_1^*u')} g'
  \end{equation*}
  so that
  \begin{equation*}
    \phi (d_0^*u)\wedge \phi' (d_0^*u')
    = (d_1^*u) \wedge (d_1^*u') (d_1^*u') g^{s'(d_1^*u')} g'
  \end{equation*}
  we conclude the morphism $\phi\wedge\phi'$ is represented by
  $g^{d_1^*x'}g'$.  Since obviously the value of $s\wedge s'$
  over $u\wedge u'$ is $xx'$, we finally have obtained that the
  cocycle corresponding to $(P,s)\otimes (P',s')$ is the product
  of the two cocycles in the sense of~\eqref{eq:21} (or, more
  precisely, \eqref{eq:22}).
\end{proof}
\begin{paragr}
  We conclude this section with the following observation, which
  will be useful elsewhere in this paper: if $\grg$ is the
  associated gr-stack to $[G_1\to G_0]$, then there is an exact
  sequence:
  \begin{equation}
    \label{eq:26}
    G_1\overset{\del}{\lto} G_0 \xrightarrow{\pi_G} \grg
  \end{equation}
  of gr-stacks over $\s$.  Here $G_1$ and $G_0$ are considered as
  gr-stacks in the obvious way. The map $\pi_\grg$ associates to
  the element $x\colon U\to G_0$ the trivial $(G_1,G_0)$-torsor
  $(G_1\rvert_U, x)$ over $U$, where $x$ is identified with the
  equivariant map sending the global section $1$ to $x$.
  Exactness is intended in the sense of stacks: there it is a
  pull-back square
  \begin{equation*}
    \xymatrix{%
      G_1\ar[r]^\del \ar[d] & G_0 \ar[d]^{\pi_\grg} \\
      \mathbf{1} \ar[r] & \grg
    }
  \end{equation*}
  which is 2-commutative. ($\mathbf{1}$ is the category with one
  object and one arrow.)  This is discussed in
  sections~\ref{sec:bist-cross-modul}, and with respect to the
  exactness question,~\ref{sec:exact-sequences}.

  In terms of the corresponding simplicial group objects, the
  above sequence corresponds to the highlighted portion of the
  following homotopy exact sequence\ \cite[eq.\
  (3.11.2)]{MR92m:18019}:
  \begin{equation*}
    \begin{xy}
      \xymatrix@1{ {*} \ar[r] & \Omega\smp{G}_\bullet \ar[r] &
        G_1 \ar[r]%
        \save []!<-1pc,1pc>.{[0,2]!<1pc,-1pc>}*[F--]\frm{}
        \restore%
        & G_0 \ar[r] & \smp{G}_\bullet \ar[r] & \B G_1 \ar[r] &
        \B G_0 \ar[r] & \W\smp{G}_\bullet}
    \end{xy}
  \end{equation*}
  where $G_0$ and $G_1$ are considered as constant simplicial
  groups.  We will say more about~\eqref{eq:26} further down in
  the paper.
\end{paragr}

\section{Butterflies and weak morphisms of crossed modules}
\label{sec:butt-weak-morph}

Morphisms of crossed modules as defined in~\ref{def:3} can be
generalized quite a bit, and the resulting theory has a more
geometric flavor.  Over the punctual topos, that is, when we are
dealing with groups and crossed modules in $\Set$, the notion of
weak morphism has been developed by the second author in
refs.~\cite{Noohi:notes,Noohi:weakmaps}.  As mentioned above, the
framework of weak morphisms of crossed modules translates into
the calculus of diagrams called ``butterflies,'' owing to their
shape.

In this section we recast this discussion in the sheaf theoretic
context of gr-stacks. As this is more than a mere translation,
our treatment is going to be quite different from that in the
above mentioned references.  From this more geometric point of
view we posit that weak morphisms of crossed modules are additive
functors between the associated gr-stacks. It is equivalent, in a
sense made precise below, to considering butterfly diagrams as
morphisms between crossed modules.

In a later section (sect.~\ref{sec:fiber-bicat-cross}) we will
show how crossed modules equipped with butterflies as their
morphisms (i.e.\ weak morphisms) form a fibered bicategory which
is biequivalent to the fibered 2-category of gr-stacks.

\subsection{General definitions}
\label{sec:general-definitions}

Let $[H_1\to H_0]$ and $[G_1\to G_0]$ be two crossed modules of
$\cat{T}$.
\begin{definition}
  \label{def:6}
  A \emph{weak morphism} $F\colon H_\bullet \to G_\bullet$ is an
  additive functor
  \begin{equation*}
    F\colon [H_1\to H_0]\sptilde \lto [G_1\to G_0]\sptilde
  \end{equation*}
  between the corresponding gr-stacks. A (weak) 2-morphism is a
  morphism of such additive functors (as in
  sect.~\ref{sec:gr-categories}).
\end{definition}
\begin{remark}
  Strict morphisms from $H_\bullet$ to $G_\bullet$, as defined in
  Definition~\ref{def:3}, give rise to weak morphisms in the
  obvious way, since they give rise to strict additive functors
  \begin{equation*}
    \left[
      \xymatrix@1{%
        H_0\times H_1 \ar@<0.3pc>[r]^(.6)s \ar@<-0.1pc>[r]_(.6)t & H_0
      }\right] \lto
    \left[
      \xymatrix@1{%
        G_0\times G_1 \ar@<0.3pc>[r]^(.6)s \ar@<-0.1pc>[r]_(.6)t & G_0
      }\right]
  \end{equation*}
  and therefore to morphisms between the associated stacks.
\end{remark}
\begin{definition}
  \label{def:7}
  A \emph{butterfly} from $H_\bullet$ to $G_\bullet$ is a
  commutative diagram of group homomorphisms of the form
  \begin{equation}
    \label{eq:27}
    \vcenter{%
      \xymatrix@R-0.5em{%
        H_1\ar[dd]_\del \ar@/_0.1pc/[dr]^\kappa  & &
        G_1 \ar@/^0.1pc/[dl]_\imath \ar[dd]^\del\\
        & E\ar@/_0.1pc/[dl]_\pi \ar@/^0.1pc/[dr]^\jmath &  \\
        H_0 & & G_0
      }}
  \end{equation}
  where $E$ is a group object of $\cat{T}$, the NW-SW sequence is
  a complex, and the NE-SW sequence is a group extension.  The
  various maps satisfy the equivariance conditions written
  set-theoretically as:
  \begin{equation}
    \label{eq:28}
    \imath (g^{\jmath(e)}) = e^{-1} \imath (g) e,\quad
    \kappa (h^{\pi (e)}) = e^{-1} \kappa (h) e
  \end{equation}
  where $g\in G_1, h\in H_1, e\in E$.
\end{definition}
Let us use the short-hand notation $[H_\bullet,E,G_\bullet]$ for
the butterfly diagram~\eqref{eq:27}, or even just $E$ when there
is no danger of confusion. As in~\cite{Noohi:weakmaps}, we have:
\begin{proposition}
  \label{prop:4}
  The images of $\kappa$ and $\imath$ commute in $E$.
\end{proposition}
\begin{proof}
  Easy consequence of~\eqref{eq:28}.
\end{proof}
\begin{definition}
  \label{def:19}
  A butterfly~\eqref{eq:27} is \emph{flippable,} or
  \emph{reversible,} if both diagonals are extensions.
\end{definition}

A slightly stronger version of the definition of a butterfly
plays a non-trivial role in some examples, most notably those
related to braidings.
\begin{definition}
  \label{def:8}
  A \emph{strong butterfly} is a butterfly~\eqref{eq:27} equipped
  with a global section $s\colon H_0\to E$ of $\pi\colon E\to
  H_0$ of underlying $\Set$-valued sheaves, namely such that
  $\pi\circ s = \id_{H_0}$.
\end{definition}

Morphism of butterflies are defined as follows:
\begin{definition}
  \label{def:9}
  A morphism of butterflies $\phi\colon
  [H_\bullet,E,G_\bullet]\to [H_\bullet,E',G_\bullet]$ is given
  by a group isomorphism $\phi\colon E\isoto E'$ such that
  \begin{equation*}
    \xymatrix@C+1pc{%
      H_1 \ar[r] \ar@/_0.1pc/[dr] \ar[dd] &  E' \ar@/^/[ddr]|(.35)\hole
      \ar@/_/[ddl]|(.35)\hole  & G_1 \ar[l] \ar[dd] \ar@/^0.1pc/[dl] \\
      &  E \ar@/_0.2pc/[dl] \ar@/^0.2pc/[dr] \ar[u] \\
      H_0  & & G_0}
  \end{equation*}
  commutes and is compatible with all the conditions
  in~\ref{def:7}.  Two morphisms
  \begin{equation*}
    \phi\colon
    [H_\bullet,E,G_\bullet]\to [H_\bullet,E',G_\bullet] \,,\quad
    \phi'\colon [H_\bullet,E',G_\bullet]\to
  [H_\bullet,E'',G_\bullet]
  \end{equation*}
  are composed in the obvious way.
\end{definition}
\begin{paragr}
  It is clear from Definitions~\ref{def:7} and~\ref{def:9} that
  butterflies from $H_\bullet$ to $G_\bullet$ and their morphisms
  form a groupoid. Let us denote it by $\cat{B}
  (H_\bullet,G_\bullet)$.

  Another groupoid naturally associated with two crossed modules
  $H_\bullet$ and $G_\bullet$ is the groupoid of weak morphisms
  as defined in sect.~\ref{sec:general-definitions}.  This
  groupoid will be denoted by $\cat{WM}(H_\bullet, G_\bullet)$.
\end{paragr}

\subsection{Remarks on the definition of weak morphism}
\label{sec:remarks-weak-morph}

In the set-theoretic case, the definition of weak morphism is
seemingly different. In \cite{Noohi:notes,Noohi:weakmaps} weak
morphisms from $H_\bullet$ to $G_\bullet$ are defined as
(pointed) lax functors from $\grpd{H}[1]$ to
$\grpd{G}[1]$. Recall that $\grpd{G}$ is the groupoid determined
by $G_\bullet$, and that $\grpd{G}[1]$ is the ``suspension'' of
$\grpd{G}$, namely the 2-category with only one object and
1-morphisms given by the objects of $\grpd{G}$, with composition
law given by the monoidal law of $\grpd{G}$.

\begin{paragr}
  In the context of sheaves over a site, given a crossed module
  $G_\bullet$ we have different notions of suspension:
  $\grpd{G}[1]$, the suspension of the groupoid $\grpd{G}$
  itself; $\grg[1]$, the suspension of the gr-stack $\grg$
  associated to $G_\bullet$; and finally $\tors (\grg)$, the
  2-gerbe of $\grg$-torsors (see \cite{MR95m:18006}). As remarked
  in \loccit and later on in section~\ref{sec:cone-butterfly},
  the latter is the correct one from a geometric point of view,
  as it is associated to $\grg [1]$ by a process of
  2-stackification. ($\grg [1]$ is a fibered bicategory over $\s$
  which deserves to be called a pre-bistack, since $\grg$ itself
  is a stack, but for which the 2-descent condition on objects
  does not hold; $\grpd{G}[1]$ is even less geometric: as
  $\grpd{G}$ itself is only a prestack, in the suspension only
  the 2-morphisms form a sheaf over $\s$.)
\end{paragr}
\begin{paragr}
  Given crossed modules $H_\bullet$ and $G_\bullet$, one can
  consider the following groupoids:
  \begin{enumerate}
  \item\label{item:1} $\operatorname{LaxFnct}_* (\grpd{H}[1], \grpd{G}[1])$:
    lax pointed 2-functors;
  \item\label{item:2} $\twocatHom_* (\grh [1], \grg [1])$:
    pointed Cartesian functors of fibered bicategories;
  \item\label{item:3} $\twocatHom_* (\tors (\grh), \tors (\grg))$: pointed
    Cartesian 2-functors of fibered 2-categories.
  \end{enumerate}
  A priori these are 2-groupoids, but since we are in the pointed
  case, they actually are equivalent to 1-groupoids. In the
  latter case, $\tors (\grg)$ is naturally pointed by the trivial
  torsor. Thus pointed morphisms send the trivial $\grh$-torsor
  to the trivial $\grg$-torsor up to equivalence.
\end{paragr}
\begin{paragr}
  There are equivalences:
  \begin{equation*}
    \cat{WM}(H_\bullet, G_\bullet) \lisoto
    \twocatHom_* (\grh [1], \grg [1])\lisoto
    \twocatHom_* (\tors (\grh), \tors (\grg)).
  \end{equation*}
  By Morita theory (see \cite{MR2183393}, or more precisely a
  categorification of it) the \emph{un-}pointed 2-groupoid
  $\twocatHom (\tors (\grh), \tors (\grg))$ consists of
  $(\grh,\grg)$-bimodules, namely stacks with simultaneous left
  $\grh$ and right $\grg$-actions, that are actually torsors for
  the right $\grg$-action. The \emph{pointed} ones are the ones
  for which the corresponding bimodule is actually equivalent to
  the trivial torsor. Hence they correspond to actual additive
  functors $\grh \to \grg$.

  The first equivalence between $\cat{WM}(H_\bullet, G_\bullet)$
  and the groupoid~\ref{item:2} is an application of the
  definitions.

  The groupoid~\ref{item:1} in the list is \emph{strictly}
  smaller, however. It is rather easy to see that it only leads
  to additive functors of the form $F(U)\colon \grpd{H}(U)\to
  \grpd{G}(U)$ for each object $U$ of the site $\s$. In other
  words, it gives rise to additive functors between the
  corresponding prestacks.  In light of
  sections~\ref{sec:weak-morph-butt}
  and~\ref{sec:proof-theor-refthm:3}, that choice only
  corresponds to \emph{strong} butterflies in the sense of
  Definition~\ref{def:8}.
\end{paragr}

\subsection{Weak morphisms and butterflies}
\label{sec:weak-morph-butt}
One of our main results is the theorem stating that the groupoid
of butterflies $\cat{B} (H_\bullet,G_\bullet)$ from $H_\bullet$
to $G_\bullet$ is equivalent to that of weak morphisms. More precisely:
\begin{theorem}
  \label{thm:2}
  There exists a pair of quasi-inverse functors
  \begin{gather*}
    \Phi\colon \cat{B} (H_\bullet, G_\bullet) \lto
    \cat{WM} (H_\bullet, G_\bullet) \\
    \intertext{and} \Psi \colon \cat{WM} (H_\bullet, G_\bullet)
    \lto \cat{B} (H_\bullet, G_\bullet).
  \end{gather*}
  defining an equivalence between $\cat{B} (H_\bullet,
  G_\bullet)$ and $\cat{WM} (H_\bullet, G_\bullet)$.
\end{theorem}
We will give a proof in~\ref{sec:proof-theor-refthm:3}.  It will
be simpler to directly show that $\Phi$ is fully faithful and
essentially surjective. However it still is worthwhile to have
the explicit definition of both functors at hand.

\subsubsection*{Definition of $\Phi$}

To define $\Phi$, we need to construct an additive functor
\begin{equation*}
  \Phi (E) \colon \grh \lto \grg
\end{equation*}
for each object of $\cat{B}(H_\bullet, G_\bullet)$, i.e.\ a
butterfly $[H_\bullet, E, G_\bullet]$.  Given an $H_1$-torsor
with an equivariant map $t\colon Q\to H_0$, consider the obvious
map
\begin{equation*}
  \pi_*\colon \shHom_{H_1}(Q,E) \lto \shHom_{H_1}(Q,H_0)
\end{equation*}
induced by $\pi\colon E\to H_0$ in the butterfly. $t$ is a global
section of $\shHom_{H_1}(Q,H_0)$, and we consider its local lifts
to $E$, that is the fiber over $t$:
\begin{equation*}
  \shHom_{H_1}(Q,E)_t =
  \bigl\lbrace e\in \Hom_{H_1}(Q\rvert_U,E\rvert_U) \,
  \big\vert\, \pi \circ e=t\rvert_U\bigr\rbrace
\end{equation*}
\begin{claim}
  $\shHom_{H_1}(Q,E)_t$ is a $G_1$-torsor.
\end{claim}

\begin{proof}
  Given two lifts $e,e'$ of $t$ there exists a unique $g\colon
  U\to G_1$ such that $e' = e \, \imath (g)$.  That $g$ is not a
  map from $Q$ to $G_1$, and only depends on $U$, follows from
  Proposition~\ref{prop:4}.

  $\shHom_{H_1}(Q,E)_t$ is locally non-empty since lifts exist,
  $\pi\colon E\to H_0$ being a sheaf epimorphism.
\end{proof}
Set $P=\shHom_{H_1}(Q,E)_t$. Now define $s\colon P\to G_0$ as
\begin{align*}
  s\colon \shHom_{H_1}(Q,E)_t & \lto G_0\\
  e &\longmapsto \jmath\circ e
\end{align*}
It follows from the equivariance of $e$ that $s$ is well-defined
map, that is, that it only depends on $U$, rather than the full
$Q\rvert_U$: indeed, one has, with set-theoretic notation:
\begin{equation*}
  \jmath (e (v h)) = \jmath (e (v) \kappa (h)) = \jmath (e (v)).
\end{equation*}
Moreover, if $e' = e\imath (g)$, for $g\in G_1$, then it
immediately follows that $s(e') = s (e) \del g$.
\begin{paragr}
  In sum, declare $(P,s)$ so constructed to be the object
  corresponding to $(Q,t)$.  If $\phi\colon (Q_1,t_1)\to
  (Q_0,t_0)$ is a morphism of $(H_1,H_0)$-torsors as
  in~\eqref{eq:24}, then the pull-back
  \begin{equation}
    \label{eq:29}
    \begin{aligned}
      (\phi^{-1})^* \colon \shHom_{H_1}(Q_1,E)_{t_1} & \lto
      \shHom_{H_1}(Q_0,E)_{t_0} \\
      e &\longmapsto e\circ \phi^{-1}
    \end{aligned}
  \end{equation}
  is clearly a morphism of $(G_1,G_0)$-torsors. Indeed
  eq.~\eqref{eq:24} is trivially satisfied simply because
  $\jmath\circ (e_0\circ \phi^{-1}) = (\jmath\circ e_0)\circ
  \phi^{-1}$, where $e_0\in \shHom_{H_1}(Q_0,E)$. Clearly this
  respects composition and identity objects.

  It is also clear that given two $(H_1,H_0)$-torsors $(Q_0,t_0)$
  and $(Q_1,t_1)$ there is an isomorphism of $(G_1,G_0)$-torsors
  \begin{equation}
    \label{eq:30}
    \Phi(E) (Q_0,t_0)\cprod{G_1} \Phi(E)(Q_1,t_1) \lisoto
    \Phi(E) \bigl(Q_0\cprod{H_1}Q_1,t_0t_1\bigr),
  \end{equation}
  in that given two lifts $e_0,e_1$ of $t_0,t_1$ to $E$ the
  product $e_0e_1$ is a lift of $t_0t_1$. It is verified at once
  that this isomorphism satisfies the required properties in the
  definition of additive functors (cf.\ sect.\
  \ref{sec:gr-categories}).
\end{paragr}
\begin{paragr}
  If $\alpha\colon E\to E'$ gives a morphism of butterflies, then
  there is an induced isomorphism
  \begin{equation*}
    \alpha_*\colon
    \shHom_{H_1}(Q,E)_t \lisoto \shHom_{H_1}(Q,E')_t
  \end{equation*}
  for each $(H_1,H_0)$-torsor $(Q,t)$, obtained by pushing along
  $\alpha$. Since this is clearly natural with respect to the
  pull-backs~\eqref{eq:29}, it provides a natural transformation
  $\Phi(E)\Rightarrow \Phi(E')$, which is immediately verified to
  be compatible with~\eqref{eq:30}, hence with the additive
  structure, as required.
\end{paragr}

\subsubsection*{Definition of $\Psi$}

Given a $F\colon \grh\to \grg$ be a morphism of gr-stacks over
$\s$, consider the stack fibered product:
\begin{equation*}
  \xymatrix{%
    H_0\times_{\grg}G_0 \ar[rr] \ar[d] && G_0 \ar[d]^{\pi_\grg}\\
    H_0 \ar[r]_{\pi_\grh}& \grh \ar[r]_F & \grg
  }
\end{equation*}
By definition of stack fibered product (\cite{MR1771927}),
$E=H_0\times_{\grg}G_0$ consists of elements $(y,f,x)$, where
$y\colon U\to H_0$ and $x\colon U\to G_0$, and $f$ is an
isomorphism $f\colon F(\pi_\grh(y))\isoto \pi_\grg (x)$.  Both
$G_0$ and $H_0$ are group objects, hence in particular spaces,
therefore so is the stack fiber product.  Moreover, putting
$(y_0,f_0, x_0)(y_1, f_1, x_1) = (y_0y_1, f_0f_1, x_0x_1)$, if
$y_0, \dotsc$ etc. are points of $H_0$ and $G_0$, obviously
endows it with a group structure. Here $f_0f_1$ stands for the
composition
\begin{equation*}
  F(\pi_\grh(y_0y_1)) \iso
  F(\pi_\grh(y_0))\otimes F(\pi_\grh(y_1))
  \xrightarrow{f_0\otimes f_1}
  \pi_\grg(x_0)\otimes \pi_\grg(x_1)
  = \pi_\grg (x_0x_1).
\end{equation*}
\begin{paragr}
  We have the following diagram over $\grg$:
  \begin{equation}
    \label{eq:31}
    \vcenter{%
      \xymatrix@R-1em@C+1em{%
        H_1\ar[dd]_{\del_H} \ar@/_0.1pc/[dr]^\kappa  & &
        G_1 \ar@/^0.1pc/[dl]_\imath \ar[dd]^{\del_G}\\
        & H_0\times_{\grg}G_0 \ar@/_0.1pc/[dl]_\pi
        \ar@/^0.1pc/[dr]^\jmath &  \\
        H_0 \ar@/_0.1pc/[dr]_{F\circ\pi_{\grh}} & &
        G_0 \ar@/^0.1pc/[dl]^{\pi_{\grg}} \\
        & \grg
      }
    }
  \end{equation}
  In~\eqref{eq:31} the maps $\pi$ and $\jmath$ are defined to be
  the canonical projections to the respective factors.

  The precise definitions of the maps $\kappa$ and $\imath$ are
  slightly more involved. Let
  \begin{equation*}
    \lambda_F\colon F(\pi_\grh(1))\to \pi_\grg(1)
  \end{equation*}
  be the isomorphism between the image of the unit object of
  $\grh$ and the unit object of $\grg$ (cf.\
  sect.~\ref{sec:gr-categories}). The homomorphism $\kappa$ is
  given by
  \begin{equation*}
    \kappa (h) = (\del_Hh, f_h, 1),
  \end{equation*}
  where $f_h\colon F(\pi_\grh(\del_Hh))\isoto \pi_\grg(1)$ is
  defined by:
  \begin{equation}
    \label{eq:32}
    \vcenter{%
      \xymatrix@-0.8pc{%
        **[l]F(\pi_\grh(\del h)) \ar@/^0.2pc/[dr]^{f_h} \ar[dd]_{F(h)} \\
        & \pi_\grg(1) \\
        **[l]F(\pi_\grh(1)) \ar@/_0.2pc/[ur]_{\lambda_F}
      }}
  \end{equation}
  where $h$ is regarded as a morphism
  \begin{equation*}
    (H_1\rvert_U,\del h)=
    \pi_\grh(\del h) \lto \pi_\grh(1)=(H_1\rvert_U,1)
  \end{equation*}
  between objects of $\grh$ via the identification
  $\shAut_{H_1}(H_1)\iso H_1$ (\cite[III.1.2.7 (ii)]{MR49:8992}).
  Similarly, the definition of $\imath$:
  \begin{equation*}
    \imath (g) = (1,f_g, \del g)
  \end{equation*}
  where $f_g \colon F(\pi_\grh(1))\isoto \pi_\grg(\del g)$ is
  defined by:
  \begin{equation*}
    \xymatrix@-0.8pc{%
      & \pi_\grg(\del g) \ar[dd]^g \\
      **[l]F(\pi_\grh(1)) \ar@/^0.2pc/[ur]^{f_g}
      \ar@/_0.2pc/[dr]_{\lambda_F} \\
      & \pi_\grg (1)
    }
  \end{equation*}
\end{paragr}
\begin{paragr}
  The NE-SW diagonal in~\eqref{eq:31} is exact, since it is the
  pull-back of the exact sequence
  \begin{equation*}
    G_1\overset{\del}{\lto} G_0 \xrightarrow{\pi_G} \grg
  \end{equation*}
  (recall~\eqref{eq:26} at the end of
  sect.~\ref{sec:gr-stacks-associated}) to $H_0$. The NW-SW
  diagonal is only a complex since it is the pull-back to $G_0$
  of the composite
  \begin{equation*}
    H_1\overset{\del}{\lto} H_0 \xrightarrow{\pi_H}
    \grh \overset{F}{\lto} \grg
  \end{equation*}
  which is itself only a complex. It is immediately verified that
  the various maps satisfy the conditions in def.~\ref{def:7}, so
  the diagram~\eqref{eq:31} is a butterfly from $H_\bullet$ to
  $G_\bullet$.
\end{paragr}
Therefore we define:
\begin{equation*}
  \Psi (F) = [H_\bullet, H_0\times_\grg G_0, G_\bullet]
\end{equation*}
Moreover, if $\theta\colon F\Rightarrow F'\colon \grh\to \grg$ is
a morphism of additive functors, and $E'$ is the fibered product
constructed as in~\eqref{eq:31}, with $F'$ in place of $F$, then
there is an induced isomorphism $E\isoto E'$, obtained by sending
the triple $(y,f,x)$ of $E$ to $(y,f',x)$, where $f'$ is the
composite of $f$ with the inverse of:
\begin{equation*}
  \theta_{\pi_\grh(y)}\colon F(\pi_\grh (y)) \lisoto
  F'(\pi_\grh(y)). 
\end{equation*}
One can easily check that with these definitions $\Psi$ is indeed
a functor.

\subsection{Proof of theorem \ref{thm:2}}
\label{sec:proof-theor-refthm:3}
\begin{lemma}
  \label{lem:9}
  $\Phi\colon \cat{B} (H_\bullet, G_\bullet) \to \cat{WM}
  (H_\bullet, G_\bullet)$ is fully faithful.
\end{lemma}
\begin{proof}
  Suppose first $\alpha$, $\beta\colon E\to E'$ are morphisms of
  butterflies such that
  \begin{equation*}
    \alpha_*=\beta_* \colon
    \shHom_{H_1}(Q,E)_t \lto \shHom_{H_1}(Q,E')_t
  \end{equation*}
  for each $(H_1,H_0)$-torsor $(Q,t)$. In particular, if
  $(Q,t)=\pi_\grh(y)$, for $y\colon U\to H_0$, we have
  \begin{equation*}
    \shHom_{H_1}(H_1\rvert_U,E)_y \iso E_y,
  \end{equation*}
  where $E_y$ is the fiber of $\pi\colon E\to H_0$ above $y$,
  which indeed is a $G_1$-torsor. In fact $E_y$ is a
  $G_1$-\emph{bi}torsor, and, according to the account of the
  Schreier theory in ref.~\cite{MR0354656-VII} (see
  also~\cite{MR92m:18019}), the ``bitorsor cocycle''
  \begin{equation}
    \label{eq:33}
    E_y \cprod{G_1}E_{y'} \lisoto E_{yy'}
  \end{equation}
  allows to recover $E$ as well as the extension $1\to G_1\to
  E\to H_0\to 1$.

  Thus the identity $\alpha_*=\beta_*$ reduces to two identical
  maps
  \begin{equation*}
    E_y\lto E'_y
  \end{equation*}
  for all $y$, compatible, by~\eqref{eq:30}, with~\eqref{eq:33}
  and the corresponding one for $E'$. It follows that
  $\alpha=\beta$.

  If, on the other hand, $[H_\bullet,E,G_\bullet]$ and
  $[H_\bullet,E',G_\bullet]$ are two butterflies and $\phi\colon
  \Phi(E)\to \Phi(E')$ a morphism of additive functors, by
  definition we have a natural morphism
  \begin{equation*}
    \phi_{Q,t}\colon
    \shHom_{H_1}(Q,E)_t \lto \shHom_{H_1}(Q,E')_t
  \end{equation*}
  for each $(H_1,H_0)$-torsor $(Q,t)$. Once again, when
  $(Q,t)=\pi_\grh(y)$ we obtain an isomorphism
  \begin{equation*}
    \phi_{H_1,y}\colon E_y\lto E'_y,
  \end{equation*}
  for all $y\colon U\to H_1$. The same arguments as above, in
  particular the compatibility with~\eqref{eq:33}, allow to
  conclude that the various $\phi_{H_1,y}$ glue into a
  homomorphism $E\to E'$. (That it indeed is a homomorphism, in
  particular, follows from the compatibility with~\eqref{eq:33}.)
\end{proof}

To prove essential surjectivity, we need to construct, for any
additive functor $F\colon \grh\to \grg$, a butterfly $E_F$ and a
morphism
\begin{equation*}
  \phi_F \colon F\lto \Phi(E_F) 
\end{equation*}
of additive functors.  This will follow from the following
\begin{proposition}
  \label{prop:5}
  For each $(H_1,H_0)$-torsor $(Q,t)$ there is a natural
  isomorphism
  \begin{equation*}
    \phi_{Q,t} \colon F(Q,t)\lisoto \shHom_{H_1}(Q,E)_t, \\
  \end{equation*}
  where
  \begin{equation*}
    E \eqdef H_0\times_{F,\grg} G_0.
  \end{equation*}
\end{proposition}
Note that in the previous statement $E$ is simply the butterfly
obtained by applying the functor $\Psi$ to $F$. We have
explicitly marked the dependency on $F$ in the notation for
clarity.
\begin{proof}[Proof of Proposition~\ref{prop:5}]
  Let us begin by assuming, as we have repeatedly done above,
  that $(Q,t)=\pi_\grh(y)$, for $y\colon U\to H_0$, so we have
  \begin{equation*}
    \shHom_{H_1} (H_1,H_0\times_{F,\grg} G_0)_y \iso
    (H_0\times_{F,\grg} G_0)_y.
  \end{equation*}
  The right hand side above consists of pairs $(f,x)$, where
  $x\colon U\to G_0$ is a point, and
  \begin{equation*}
    f \colon F(\pi_\grh(y)) \lisoto \pi_\grg (x),
  \end{equation*}
  the point $y$ being fixed. Incidentally, that
  $(H_0\times_{F,\grg} G_0)_y$ is a $G_1$-torsor directly results
  from the diagram of $(G_1,G_0)$-torsors:
  \begin{equation*}
    \xymatrix@C-2pc@R-1pc{%
      & \pi_\grg (x) \ar[dd]^{f'f^{-1}} \\
      **[l] F(\pi_\grh(tu)) \ar@/^0.1pc/[ur]^f \ar@/_0.1pc/[dr]_{f'}\\
      & \pi_\grg (x')
    }
  \end{equation*}
  $f'f^{-1}$, as a morphism between the underlying $G_1$-torsors,
  is identified with an element $g\colon U\to G_1$. Moreover,
  from~\eqref{eq:24} it follows that $x = x'\del g$. From this we
  immediately recognize that $(H_0\times_{F,\grg} G_0)_y$ has a
  structure of $(G_1,G_0)$-torsor, where the equivariant map to
  $G_0$ is simply the projection (Equivariance follows at once
  from the diagram above).

  If for simplicity we simply denote $F(\pi_\grh(y))$ by $(P,s)$,
  then from the previous paragraphs it follows that at the level
  of underlying $G_1$-torsors we have an isomorphism
  \begin{equation}
    \label{eq:34}
    (H_0\times_{F,\grg} G_0)_y \lisoto \shHom_{G_1}(G_1,P)
  \end{equation}
  obtained by sending $(f,x)$ to $f^{-1}$.  The right-hand side
  is a $G_1$-torsor in a trivial way (from the left-action of
  $G_1$ onto itself), and in addition we have an isomorphism
  \begin{equation}
    \label{eq:35}
    \shHom_{G_1}(G_1,P) \lisoto P,
  \end{equation}
  obtained by evaluating a map $m$ on the left-hand side at the
  unit $1\in G_1$ (see~\cite[III.1.2.7 (i)]{MR49:8992}). It is
  also clear that~\eqref{eq:34} is a morphism of
  $(G_1,G_0)$-torsors: the projection sending $(f,x)$ to $x$ maps
  to $s\circ f^{-1}$, and this goes to $s$ itself via the latter
  isomorphism.

  In summary, we set $\phi_{H_1,y}$ equal to the composition
  of~\eqref{eq:35} with~\eqref{eq:34}.
  
  We must verify the property that given a morphism
  \begin{equation*}
    h\colon \pi_\grh(y) \lto \pi_\grh (y'),
  \end{equation*}
  which corresponds to an element $h\colon U\to H_1$ such that
  $y=y'\,\del h$, the isomorphism just defined $\phi_{H_1,y}$
  behaves naturally with respect to it, namely that the diagram
  \begin{equation}
    \label{eq:36}
    \vcenter{%
      \xymatrix{%
        (H_0\times_{F,\grg} G_0)_y \ar[r] \ar[d]_{\phi_{H_1,y}}
        & (H_0\times_{F,\grg} G_0)_{y'} \ar[d]^{\phi_{H_1,y'}} \\
        F(\pi_\grh(y)) \ar[r]_{F(h)} 
        &  F(\pi_\grh(y')) 
      }}
  \end{equation}
  commutes. By $H_1$-equivariance, the top horizontal map is just
  the right $H_1$-action of $h^{-1}$ via $\kappa$, namely the one
  sending $(y,f,x)$ to
  \begin{equation*}
    (y,f,x)\,\kappa(h^{-1}) = (y', ff_h,1),
  \end{equation*}
  where $f_h$ was defined before along with $\kappa$. The
  isomorphisms $f$ and $ff_h$ fit in the commutative diagram
  \begin{equation*}
    \xymatrix@-0.8pc{%
      **[l]F(\pi_\grh(y)) \ar@/^0.2pc/[dr]^{f} \ar[dd]_{F(h)} \\
      & \pi_\grg(1) \\
      **[l]F(\pi_\grh(y')) \ar@/_0.2pc/[ur]_{ff_h}
    }
  \end{equation*}
  for $F$ is additive, and we can use~\eqref{eq:32}.  Recalling
  the definition of $\phi_{H_1,y}$, it is clear
  that~\eqref{eq:36} commutes.

  The case of a general $(H_1,H_0)$-torsor $(Q,t)$ is obtained by
  gluing local instances of the above construction via
  descent. That is, if $U$ is an object of $\s$, and $Y\to U$ a
  local epimorphism such that $u\colon Y\to Q$ gives a
  trivialization $(H_1)_Y\isoto Q_Y$ of the underlying
  $H_1$-torsor, we obtain a morphism of $(H_1,H_0)$-torsors
  \begin{equation*}
    (Q,t)_Y \lisoto (H_1,y)_Y
  \end{equation*}
  where $y = u^*t$, and in turn
  \begin{equation*}
    h\colon \pi_\grh (d_0^*y) \lisoto \pi_\grh (d_1^*y) 
  \end{equation*}
  is realized by an element $h$ of $H_1$ over $Y\times_UY$, so
  that the pair $(h,y)$ is a $0$-cocycle relative to $\cech (Y\to
  U)$ as seen in sect.~\ref{sec:cocycles}.

  By applying $F$, we obtain descent data for the
  $(G_1,G_0)$-torsors $F(\pi_\grh(y))$ via the morphisms $F(h)$.
  This reconstructs $F(Q,t)$, since locally we have
  $F(H_1,y)_Y\isoto F(Q,t)_Y$.

  Similarly, $\shHom_{H_1}(Q,H_0\times_{F,\grg} G_0)_t$ is
  obtained from the corresponding descent data for the various
  $(H_0\times_{F,\grg} G_0)_{y}$ resulting from the
  diagram~\eqref{eq:36}.

  This shows at once that the $\phi_{H_1,y}$ glue into a global
  \begin{equation*}
    \phi_{Q,t}\colon F(Q,t)\lisoto \shHom_{H_1}(Q,H_0\times_{F,\grg} G_0)_t.
  \end{equation*}
  That the morphism $\phi_{Q,t}$ is itself natural with respect
  to morphisms of $(H_1,H_0)$-torsors follows again by patching
  arguments.
\end{proof}
This concludes the proof of Theorem~\ref{thm:2}.\qed

\subsection{Strict morphisms and butterflies}
\label{sec:strict-morph-butt}

We have observed that strict morphisms of crossed modules induce
weak morphisms in an obvious way.  On the other hand,
Theorem~\ref{thm:2} entails the notion that butterflies ought to
be considered as (weak) morphisms.  It is natural to ask what
kind of butterfly diagrams correspond to the strict morphisms of
Definition~\ref{def:3}.
\begin{paragr}
  Let $F=(f_1,f_0)$ be a strict morphism $F\colon H_\bullet\to
  G_\bullet$ as in Definition~\ref{def:3}. The corresponding
  butterfly is
  \begin{equation}
    \label{eq:37}
    \vcenter{%
      \xymatrix@R-0.5em{%
        H_1\ar[dd]_\del \ar@/_0.1pc/[dr]^\kappa  & &
        G_1 \ar@/^0.1pc/[dl]_\imath \ar[dd]^\del\\
        & H_0\ltimes G_1\ar@/_0.1pc/[dl]_\pi
        \ar@/^0.1pc/[dr]^\jmath &  \\
        H_0 & & G_0
      }}
  \end{equation}
  where $\pi= \mathrm{pr}_1$, $\imath = (1,\id)$, $\kappa (h) =
  (\del (h), f_1(h^{-1}))$, and $\jmath (y, g) = f_0(y)\del (g)$.
  The semi-direct product at the center of the butterfly
  corresponds to the trivial extension of $H_0$ by $G_1$
  (\cite{maclane:hom}).  Equivalently, the NE-SW diagonal of the
  butterfly is a split extension. Also, the product law
  \emph{depends} on the actual (strict) morphism $F=(f_1,f_0)$:
  \begin{equation*}
    (y_0,g_0)(y_1,g_1) = (y_0y_1, g_0^{f_0(y_1)}g_1).
  \end{equation*}
  Therefore it would be more appropriate to record this
  dependency in the notation: $H_0\overset{\scriptscriptstyle
    F}{\ltimes} G_1$.
\end{paragr}

It is also clear that given the butterfly~\eqref{eq:37} one can
construct a unique strict morphism $(f_1,f_0)\colon H_\bullet\to
G_\bullet$.  This is due to the canonical splitting homomorphism
$s\colon H_0 \to H_0\ltimes G_1$ which sends $y$ to $(y,1)$.
\begin{paragr}
  More generally, a butterfly~\eqref{eq:27} is \emph{splittable}
  if there exists a homomorphism $s\colon H_0\to E$. As a result,
  from standard arguments, the NE-SW diagonal is in the same
  isomorphism class as the one in~\eqref{eq:37}.  Moreover, a
  unique strict morphism $(f_0,f_1)$ can be constructed from a
  splittable butterfly once it has been equipped with a specific
  choice of the splitting homomorphism $s$: $f_0=\jmath\circ s$,
  and $f_1$ is determined by the difference between $\kappa$ and
  $s\circ \del_H$, to wit:
  \begin{equation*}
    s(\del h) = \kappa (h) \imath (f_1 (h)), \quad h\in H_1.
  \end{equation*}
  One can easily check that the pair $(f_0,f_1)$ so determined
  has all the required properties~\eqref{eq:9} and~\eqref{eq:10}.
\end{paragr}
The next statement is therefore an immediate consequence of
Theorem~\ref{thm:2}:
\begin{proposition}
  \label{prop:6}
  A weak morphism $F \colon H_\bullet\to G_\bullet$ is equivalent
  to a strict one if and only if its butterfly is isomorphic to a
  split one.
\end{proposition}

If $\gamma\colon F\Rightarrow F'$ is a strict 2-morphism as in
Definition~\ref{def:4}, it is easy to see that the group
isomorphism
\begin{align*}
  \phi \colon H_0\overset{\scriptscriptstyle F}{\ltimes} G_1
  &\lto H_0\overset{\scriptscriptstyle F'}{\ltimes} G_1 \\
  (x,g) &\longmapsto (x, \gamma_x^{-1} g)
\end{align*}
is a morphism between the split butterflies corresponding to $F$
and $F'$.

In summary, we have a functor from the category of strict
morphisms between $H_\bullet$ and $G_\bullet$ and
$\cat{B}(H_\bullet, G_\bullet)$.

\subsection{Stacks of butterflies and weak morphisms}
\label{sec:stacks-butt-weak}

For two crossed modules $H_\bullet$ and $G_\bullet$ of $\cat{T}$
we have introduced the groupoid $\cat{B} (H_\bullet, G_\bullet)$
of butterflies from $H_\bullet$ to $G_\bullet$.  There is a
sheaf-theoretic counterpart, denoted by $\stb (H_\bullet,
G_\bullet)$, which is defined as usual by assigning to $U\in
\Ob\s$ the groupoid
\begin{equation*}
  \cat{B}(H_\bullet\rvert_U,G_\bullet\rvert_U),
\end{equation*}
and to every arrow $V\to U$ of $\s$ the functor
\begin{equation*}
  \cat{B}(H_\bullet\rvert_U,G_\bullet\rvert_U) \lto
  \cat{B}(H_\bullet\rvert_V,G_\bullet\rvert_V).
\end{equation*}
\begin{proposition}
  \label{prop:7}
  $\stb (H_\bullet, G_\bullet)$ is a stack over $\s$.
\end{proposition}
\begin{proof}
  Since we are restricting morphisms of ordinary group objects,
  so in particular sets, it is clear that $\stb (H_\bullet,
  G_\bullet)$ is fibered over $\s$ and it is a prestack.
  
  The same idea applies to proving that the descent condition on
  objects is effective. In slightly more details, let $Y\to U$ be
  a local epimorphism, and let us consider butterfly descent data
  along it.  Thus, let $E'$ be a butterfly on $Y$ from
  $(H_\bullet)_Y$ to $(G_\bullet)_Y$, and $\phi$ a morphism of
  butterflies
  \begin{equation*}
    \phi \colon d_0^*E' \lto d_1^*E'
  \end{equation*}
  over $\cech Y_1=Y\times_UY$ satisfying the cocycle condition
  $d_1^*\phi=d_2^*\phi\circ d_0^*\phi$ over $\cech Y_2$.

  Since sheaves of groups form a stack, these data determine a
  group $E$ object over $U$ such that $\psi\colon E' \isoto E_Y$
  and $d_1^*\psi\circ \phi = d_0^*\psi$. Moreover, it is easily
  seen that all the structural maps in the butterfly
  $[(H_\bullet)_Y,E',(G_\bullet)_Y]$ glue to provide the
  corresponding ones for a butterfly $[H_\bullet,E,G_\bullet]$
  over $U$: this follows at once from the fact that $\stb
  (H_\bullet, G_\bullet)$ is a prestack.  Consider, for instance,
  $\imath'\colon (G_1)_Y\to E'$ and the composite
  \begin{equation*}
    (G_1)_Y \overset{\imath'}{\lto} E' \overset{\psi}{\lto} E_Y.
  \end{equation*}
  Since $\phi$ is a morphism of butterflies, we have $\phi\circ
  d_0^*\imath' = d_1^*\imath'$, and we get immediately the
  equality
  \begin{equation*}
    d_0^*(\psi\circ \imath') = d_1^*(\psi\circ \imath')
  \end{equation*}
  from which it follows (since $\stb (H_\bullet, G_\bullet)$ is a
  prestack) that there is a morphism
  \begin{equation*}
    \imath\colon 
    G_1\lto E
  \end{equation*}
  of group objects over $U$. The remaining structural maps, as
  well as the relations~\eqref{eq:28} are handled in an entirely
  similar manner.
\end{proof}
If we start from $\cat{WM} (H_\bullet, G_\bullet)$, we can define
$\stwm (H_\bullet,G_\bullet)$ in the same way as we have done for
$\stb (H_\bullet, G_\bullet)$.  Then, as a consequence of
Proposition~\ref{prop:7} and the equivalence in
Theorem~\ref{thm:2}, we have the following:
\begin{corollary}
  \label{cor:1}
  $\stwm(H_\bullet,G_\bullet)$ is a stack over $\s$.
\end{corollary}

\subsection{Weak vs. Strict morphisms}
\label{sec:weak-vs.-strict}

We have seen strict morphisms $H_\bullet \to G_\bullet$ give rise
to special kinds of butterflies, namely the split ones (cf.\
Proposition~\ref{prop:6}).

To see how strict morphisms relate to the weak ones, and in
particular how they fit into the stack structure for the weak
morphisms introduced in sect.~\ref{sec:stacks-butt-weak}, fix the
crossed modules $H_\bullet$ and $G_\bullet$ and consider the
category $\cat{B} (H_\bullet,G_\bullet)$.  Within it, consider
the sub-category comprising strict morphisms and 2-morphisms from
$H_\bullet$ to $G_\bullet$, which we denote by
\begin{math}
  \cat {B} (H_\bullet,G_\bullet)_{\mathrm{str}}.
\end{math}
The notation means that we view strict morphisms as split
butterflies as explained in sect.~\ref{sec:strict-morph-butt}.

The obvious inclusion
\begin{equation*}
  \cat {B} (H_\bullet,G_\bullet)_{\mathrm{str}}
  \hookrightarrow
  \cat{B} (H_\bullet,G_\bullet)
\end{equation*}
extends to the fibered situation. Namely, we have an inclusion of
fibered categories.
\begin{equation*}
  \stb (H_\bullet,G_\bullet)_{\mathrm{str}}
  \hookrightarrow
  \stb (H_\bullet,G_\bullet),
\end{equation*}
where
\begin{math}
  \stb (H_\bullet,G_\bullet)_{\mathrm{str}}
\end{math}
is defined by repeating the procedure of
sect.~\ref{sec:stacks-butt-weak}. Thus its fiber category over
$U\in \Ob\s$ is simply category
\begin{equation*}
  \cat {B} (H_\bullet\rvert_U,G_\bullet\rvert_U)_{\mathrm{str}}.
\end{equation*}
That that this defines a fibered category is immediate, since
after all morphisms of (sheaves of) groups pull-back along arrows
in $\s$.

Consider the stack completion diagram
\begin{equation*}
  \xymatrix{%
    \stb (H_\bullet,G_\bullet)_{\mathrm{str}}
    \ar[r]^a \ar @/_1pc/ [dr]
    \drtwocell<\omit>{}
    &
    \stb (H_\bullet,G_\bullet)_{\mathrm{str}}\sptilde
    \ar@{.>}[d]\\
    & \stb (H_\bullet,G_\bullet)
  }
\end{equation*}
It is clear that the objects of the stack $\stb
(H_\bullet,G_\bullet)_{\mathrm{str}}\sptilde$ are butterflies,
since the stackification process happens ``inside'' $\stb
(H_\bullet,G_\bullet)$. Also, by the very nature of this process,
each object is locally isomorphic to one of $\stb
(H_\bullet,G_\bullet)_{\mathrm{str}}$, hence to a split one. It
follows that $\stb (H_\bullet,G_\bullet)_{\mathrm{str}}\sptilde$
is the stack of \emph{locally split} butterflies, in the sense of
the following definition:
\begin{definition}
  \label{def:10}
  A butterfly $E$ from $H_\bullet$ to $G_\bullet$ is
  \emph{locally split} if there is a generalized cover $V$ such
  that the extension on the NE-SW diagonal splits over $V$.
\end{definition}
Morphisms of locally split butterflies are butterfly morphisms
$\phi\colon E\to E'$ such that, if $E$ splits over $V$ and $E'$
splits over $V'$, then $\phi\rvert_{V\times V'}$ is a strict
2-morphism as expounded at the end of sect.\
\ref{sec:strict-morph-butt}.

\section{The bicategory of crossed modules and weak morphisms}
\label{sec:fiber-bicat-cross}

To analyze the global structure of crossed modules equipped with
weak morphisms provided by butterflies we need to use a few facts
regarding fibered bicategories.  Our goal is to prove that
crossed modules and weak morphisms over $\s$ comprise a fibered
bicategory which is in fact a bistack, in the sense made explicit
below.

The definitions one states in the context of 2-categories fibered
over a site, and the ensuing consequent distinctions based on what
actually glues---fibered, prestack, stack---, have a mirror in
the realm of fibered bicategories. (We refer to~\cite{MR0364245}
for fibered 2-categories over a site, and to~\cite[Chapter
1]{MR95m:18006} for a discussion of the extensions of these
concepts to the bicategorical situation.)  In particular, by a
pre-bistack, we mean a fibered bicategory where the descent
conditions are satisfied at all levels, except for objects. In
particular, the morphism (fibered) category between any pair of
objects will form a stack. A fibered bicategory is a bistack if
in addition the 2-descent condition on objects, which is
formulated in much the same way as for 2-categories, is
effective.

\subsection{Composition of butterflies and the bicategory of
  crossed modules}
\label{sec:bicat-cross-modul}

Recall that $\cat{B}(H_\bullet,G_\bullet)$ is the groupoid of
butterflies from $H_\bullet$ to $G_\bullet$. There is a
composition functor
\begin{equation*}
  \cat{B} (K_\bullet,H_\bullet) \times
  \cat{B} (H_\bullet,G_\bullet) \lto
  \cat{B} (K_\bullet, G_\bullet)
\end{equation*}
which is constructed in the following way (cf.\
ref.~\cite{Noohi:weakmaps}).
\begin{definition}
  \label{def:20}
  Given two butterflies
  \begin{equation*}
    \vcenter{%
      \xymatrix@R-0.5em{%
        K_1\ar[dd]_{\del_K} \ar@/_0.1pc/[dr]  & &
        H_1 \ar@/^0.1pc/[dl]_{\imath'} \ar[dd]^{\del_H}\\
        & F\ar@/_0.1pc/[dl] \ar@/^0.1pc/[dr]^{\jmath'} &  \\
        K_0 & & H_0
      }} \qquad
    \vcenter{%
      \xymatrix@R-0.5em{%
        H_1\ar[dd]_{\del_H} \ar@/_0.1pc/[dr]^\kappa  & &
        G_1 \ar@/^0.1pc/[dl] \ar[dd]^{\del_G}\\
        & E\ar@/_0.1pc/[dl]_\pi \ar@/^0.1pc/[dr] &  \\
        H_0 & & G_0
      }}
  \end{equation*}
  their composition is the butterfly (defined set-theoretically
  in \loccit):
  \begin{equation}
    \label{eq:38}
    \vcenter{%
      \xymatrix{%
        K_1\ar[dd]_{\del_K} \ar@/_0.1pc/[dr]  & &
        G_1 \ar@/^0.1pc/[dl] \ar[dd]^{\del_G}\\
        & {\displaystyle F\times_{H_0}^{H_1}E}
        \ar@/_0.1pc/[dl] \ar@/^0.1pc/[dr] &  \\
        K_0 & & G_0
      }}
  \end{equation}
  where the center is given by the following pull-back/push-out
  construction: the pull-back of the extension $1\to G_1\to E \to
  H_0\to 1$ along $\jmath' \colon F\to H_0$ gives the extension
  \begin{equation*}
    1\lto G_1 \lto F\times_{H_0}E \lto F \lto 1
  \end{equation*}
  Then the exact NE-SW diagonal of~\eqref{eq:38} arises as the
  cokernel of the morphism
  \begin{equation*}
    \xymatrix{%
      1 \ar[r] \ar[d] & H_1 \ar@{=}[r] \ar[d]^{(\imath',\kappa)} & H_1
      \ar[d]^{\imath'} \\
      G_1 \ar[r]      & F\times_{H_0}E \ar[r] & F
    }
  \end{equation*}
  All vertical maps are monomorphisms, and that the image of $H_1
  \xrightarrow{(\imath',\kappa)} F\times_{H_0}E$ is normal thanks
  to the properties \eqref{eq:28} of the maps in the butterflies.
\end{definition}
\begin{paragr}
  It is also clear that if
  \begin{equation*}
    [K_\bullet,F,H_\bullet]\to [K_\bullet,F',H_\bullet] \,,\quad
    [H_\bullet,E,G_\bullet]\to [H_\bullet,E',G_\bullet]
  \end{equation*}
  are morphisms of butterflies given by $\psi\colon F\isoto F'$
  and $\phi\colon E\isoto E'$, respectively, then there is a
  corresponding morphism
  \begin{equation*}
    [K_\bullet, F\times^{H_1}_{H_0} E, G_\bullet] \lto
    [K_\bullet, F'\times^{H_1}_{H_0} E', G_\bullet]
  \end{equation*}
  where
  \begin{equation*}
    F\times^{H_1}_{H_0} E \lisoto F'\times^{H_1}_{H_0} E'
  \end{equation*}
  is induced by
  \begin{equation*}
    (\phi,\psi)\colon F\times_{H_0} E \lisoto F'\times_{H_0} E'
  \end{equation*}
  after taking the quotient by the image of $H_1$ in both ends.
\end{paragr}
\begin{paragr}
  By the usual arguments, the construction of
  $F\times^{H_1}_{H_0} E$, is such that if we consider a third
  butterfly $[L_\bullet, M, K_\bullet]$, then there only is an
  isomorphism
  \begin{equation*}
    \bigl( M\times^{K_1}_{K_0}F \bigr) \times^{H_1}_{H_0} E \lisoto
    M\times^{K_1}_{K_0} \bigl( F \times^{H_1}_{H_0} E \bigr).
  \end{equation*}
\end{paragr}
As a result the composition law is only associative up to
isomorphism.  With these provisions, we immediately have the
following analog of\ \cite[Theorem 10.1]{Noohi:weakmaps}:
\begin{theorem}
  \label{thm:3}
  When equipped with the morphism groupoids $\cat{B}(-,-)$,
  crossed modules in $\cat{T}$ form a bicategory.
\end{theorem}
It is convenient to recall at this point the following special
cases. Their handling is unchanged from the set-theoretical
situation of ref.\ \cite{Noohi:weakmaps}, to which we refer for
more details.

\subsection{Special cases}
\label{sec:spec-cases-comp}

Composition of butterflies assumes a simpler form
than~\eqref{eq:38} when one of the morphisms is strict (cf.\
sect.\ref{sec:strict-morph-butt}).  If the morphism
$(q_1,q_0)\colon K_\bullet\to H_\bullet$ is strict, then the
composition is
\begin{equation*}
  \xymatrix{%
    K_1\ar[dd]_{\del_K} \ar@/_0.1pc/[dr]  & &
    G_1 \ar@/^0.1pc/[dl] \ar[dd]^{\del_G}\\
    &  q_0^*(E)
    \ar@/_0.1pc/[dl] \ar@/^0.1pc/[dr] &  \\
    K_0 & & G_0
  }
\end{equation*}
where the NE-SW diagonal is the pull-back of the extension
\begin{equation*}
  1\to G_1\lto E\lto H_0\lto 1
\end{equation*}
along the homomorphism $q_0\colon K_0\to H_0$.  In particular,
$q_0^*(E) = \displaystyle F\times_{H_0}E$ is the fiber product.

Similarly, if the second morphism $(p_1,p_0)\colon H_\bullet\to
G_\bullet$ is strict instead, then the composition is
\begin{equation*}
  \xymatrix{%
    K_1\ar[dd]_{\del_K} \ar@/_0.1pc/[dr]  & &
    G_1 \ar@/^0.1pc/[dl] \ar[dd]^{\del_G}\\
    & {p_1}_*(F)
    \ar@/_0.1pc/[dl] \ar@/^0.1pc/[dr] &  \\
    K_0 & & G_0
  }
\end{equation*}
where this time the NE-SW diagonal is the push-forward of the
extension
\begin{equation*}
  1\to H_1\lto F\lto K_0\lto 1
\end{equation*}
along the homomorphism $p_1\colon H_1\to G_1$.  Also, ${p_1}_*(F)
= {\displaystyle F\sqcup_{H_1}E}$ is the push-out.

We close with the following simple
\begin{lemma}
  \label{lem:8}
  The weak morphism $F$ determined by a butterfly~\eqref{eq:27}
  is an \emph{equivalence} if and only if the butterfly is
  \emph{flippable,} or \emph{reversible,} as in
  Definition~\ref{def:19}.
\end{lemma}
In this case the same butterfly, but read right-to-left:
\begin{equation*}
  \xymatrix@R-0.5em{%
    G_1\ar[dd]_\del \ar@/_0.1pc/[dr]^\imath  & &
    H_1 \ar@/^0.1pc/[dl]_\kappa \ar[dd]^\del\\
    & E\ar@/_0.1pc/[dl]_\jmath \ar@/^0.1pc/[dr]^\pi &  \\
    G_0 & & H_0
  }
\end{equation*}
is a diagram corresponding to the choice of an inverse functor
$F^*$ of $F$.
\begin{proof}[Proof of the lemma]
  If the butterfly~\eqref{eq:27} is reversible, composing it with
  its flipped counterpart according to~\eqref{eq:38} yields the
  isomorphisms
  \begin{equation*}
    E\times_{G_0}^{G_1} E \isoto H_0\ltimes^{\Id{}} H_1
    \,,\quad
    E\times_{H_0}^{H_1} E \isoto G_0\ltimes^{\Id{}} G_1.
  \end{equation*}
  so that $F$ is an equivalence.

  Conversely, if $F$ is an equivalence, then we have seen from
  the proof of Theorem~\ref{thm:2} that
  \begin{equation*}
    E = H_0\times_{\grg} G_0
  \end{equation*}
  with respect to $F\circ \pi_{\grh}\colon H_0\to \grg$. Since
  $F$ is an equivalence, the sequence
  \begin{equation*}
    H_1 \lto H_0 \xrightarrow{F\circ \pi_{\grh}} \grg
  \end{equation*}
  is still homotopy-exact. Hence, by pull-back along $G_0\to
  \grg$, the sequence
  \begin{equation*}
    H_1\lto E\lto G_0
  \end{equation*}
  is short-exact, that is, the butterfly is flippable.
\end{proof}

\subsection{The bistack of crossed modules}
\label{sec:bist-cross-modul}

Let $\CM (\s)$, or $\CM$ for short, the bicategory of crossed
modules of $\cat{T}$ in Theorem~\ref{thm:3}.  Define $\cm (\s)$,
or simply $\cm$, by assigning to every object $U$ of $\s$ the
bicategory
\begin{equation*}
  \CM (\s/U)
\end{equation*}
of crossed modules over $\s/U$, and to every arrow $V\to U$ of
$\s$ the homomorphism \footnote{We use the term ``homomorphism''
  in the sense of: $1$-morphism between bicategories, such that
  the structural 2-morphisms are
  isomorphisms---see~\cite{MR0220789}}
\begin{equation*}
  \CM (\s /U) \lto \CM (\s /V)
\end{equation*}
obtained by composition with the morphism $\s /V \to \s /U$.
\begin{proposition}
  \label{prop:8}
  $\cm$ is fibered over $\s$.  Moreover, we have an equivalence
  \begin{equation*}
    \CM \lisoto \Lim_{U\in \Ob (\s)} \cm
  \end{equation*}
  in the sense of bicategories.
\end{proposition}
\begin{proof}
  Sheaves of groups over $\s$ form a stack, hence in particular a
  fibered category over $\s$, which is in fact
  split~\cite{MR49:8992}.  Since all morphisms and 2-morphisms in
  $\CM (\s/U)$ are in effect diagrams of morphisms of sheaves of
  groups, it follows that pull-backs (as bifunctors) exist in
  $\cm$.
\end{proof}
Moreover, as an immediate consequence of Proposition~\ref{prop:7}
we obtain:
\begin{proposition}
  The fibered bicategory $\cm$ of crossed modules over $\s$ is a
  pre-bistack.
\end{proposition}
\begin{proof}
  Given two objects, $G_\bullet$ and $H_\bullet$, the fibered
  category of morphisms
  \begin{equation*}
    \shcatHom_\cm(H_\bullet, G_\bullet)
  \end{equation*}
  is precisely $\stb (H_\bullet, G_\bullet)$, which is a stack.
\end{proof}
We now turn to the question of whether 2-descent for objects of
$\cm$ is effective.

\begin{paragr}
  \label{prgr:2}
  For the sake of definiteness, let us provide a detailed list
  for a 2-descent datum for $\cm$ over $U\in \Ob\s$.
  \begin{enumerate}
  \item a hypercover $Y_\bullet \to U$ (e.g.\ the \Cech complex
    $\cech Y$ associated to a generalized cover $Y\to U$);
  \item a crossed module $G'_\bullet$ over $Y_0$;
  \item a reversible butterfly $[d_0^*G'_\bullet, E,
    d_1^*G'_\bullet]$ over $Y_1$;
  \item a morphism of butterflies $\alpha \colon d_1^*E
    \Rightarrow d_2^* E \circ d_0^*E$ over $Y_2$;
  \item a coherence condition for the $d_i^*\alpha$ over $Y_3$,
    $i=0,1,2,3$.
  \end{enumerate}
  The descent datum is effective if there exists a crossed-module
  $G\bullet$ over $U$, a reversible butterfly
  $[G'_\bullet,F,G_\bullet\rvert_{Y_0}]$, and a morphism of
  butterflies
  \begin{equation*}
    \beta \colon d_1^* F \circ E \Rightarrow d_0^*F
  \end{equation*}
  over $Y_1$, which is coherent over $Y_2$.
\end{paragr}
To prove that all 2-descent data in $\cm$ are effective (so that
$\cm$ is a bistack) it is best to exploit the relationship of
$\cm$ with the 2-category of gr-stacks (as opposed to giving a
direct proof).

For this, first consider the 2-category $\twocat{Stacks}(\s)$
of stacks over the site $\s$. It gives rise to a fibered
2-category $\stacks(\s)$ over $\s$ whose fiber over $U\in
\Ob\s$ is $\twocat{Stacks}(\s/U)$.  To every morphism $V\to U$ in
$\s$ it corresponds a Cartesian 2-functor
\begin{equation*}
  f^*\colon 
  \twocat{Stacks}(\s/U) \lto \twocat{Stacks}(\s/V)
\end{equation*}
arising from the pull-back $\stf \rightsquigarrow f^*\stf$, where
$\stf$ is a stack over $\s/U$. It is  known that $\stacks
(\s)$ is a 2-stack over $\s$ (see \cite{MR95m:18006}).

Gr-stacks form a sub-2-category $\twocat{Gr\mbox{-}Stacks}(\s)$
of $\twocat{Stacks}(\s)$ in the obvious way: the $1$-morphisms
are the additive functors and the 2-morphisms are the natural
transformations of additive functors. There is an obvious
forgetful functor
\begin{equation*}
  \twocat{Gr\mbox{-}Stacks}(\s)\lto \twocat{Stacks}(\s)
\end{equation*}
which simply forgets the additive structure.  Once again, these
considerations extend to the fibered situation to yield
$\grstacks(\s)$, the fibered 2-category of gr-stacks over the
site $\s$, as well as the forgetful functor
\begin{equation*}
  \grstacks(\s)\lto \stacks(\s).
\end{equation*}
We have the following important result:
\begin{theorem}
  \label{thm:4}
  $\grstacks(\s)$ is a 2-stack.
\end{theorem}
A direct proof is sketched in Appendix~\ref{sec:2-stack-gr}. A
more conceptual proof will be available in a forthcoming paper by
the authors.

Note that Corollary~\ref{cor:1} also follows directly from
Thm.~\ref{thm:4}, which is logically independent of any statement
about crossed modules.
\begin{definition}
  \label{def:11}
  Let $F$ be the homomorphism
  \begin{equation}
    \label{eq:39}
    F\colon \cm (\s) \lto \grstacks (\s)
  \end{equation}
  defined by sending the crossed module $[G_1\to G_0]$ to its
  associated gr-stack $[G_1\to G_0]\sptilde$, and, for two
  crossed modules $H_\bullet$ and $G_\bullet$, $\stb
  (H_\bullet,G_\bullet)$ to $\stwm (H_\bullet, G_\bullet)$.
\end{definition}
On the right-hand side of~\eqref{eq:39} we consider $\grstacks
(\s)$ as a bicategory (and in fact, as a bistack) in the obvious
way.

It immediately follows from Theorem~\ref{thm:2} that $F$ is
2-faithful---or fully faithful---in the sense of bicategories
(cf.~\cite{MR0364245} for the definition of
$i$-faithful---$i=0,\dots, 3$---in the context of
2-categories). The main results is:
\begin{theorem}
  \label{thm:5}
  The homomorphism $F$ in~\eqref{eq:39} is a
  biequivalence. Therefore $\cm (\s)$ is a bistack.
\end{theorem}
Since we have already remarked that $F$ is 2-faithful, the only
thing to be proved is essential surjectivity.  We state it
separately in the next proposition, which is also of independent
interest.
\begin{proposition}
  \label{prop:9}
  Let $\grg$ be a gr-stack. Then there exists a crossed module
  $[G_1\to G_0]$ such that $\grg$ is equivalent to the gr-stack
  $[G_1\to G_0]\sptilde$.
\end{proposition}
\begin{proof}
  The first step is to construct an additive functor
  \begin{equation*}
    \pi\colon G_0 \lto \grg
  \end{equation*}
  where $G_0$ is a group-object in $\cat{T}$. $G_0$ is
  considered, of course, as a gr-stack in the obvious way.
  First, we have the following:
  \begin{lemma}
    \label{lem:1}
    There exists an essentially surjective map
    \begin{equation*}
      \pi_0\colon X \lto \grg,
    \end{equation*}
    where $X$ is a \emph{space} over $\s$.
  \end{lemma}
  \begin{proof}[Proof of the lemma]
    Choose a skeleton $\sk \grg$ of $\grg$.  (Recall that a
    skeleton is a full subcategory having one object for each
    isomorphism class of objects of the ambient category,
    cf.~\cite{MR0349793}.) The inclusion $\imath \colon \sk \grg
    \to \grg$ is an equivalence, and it is easy to show that $\sk
    \grg$ is fibered, and in fact split, over $\s$. Therefore
    $\Ob(\sk \grg)$ is a presheaf of sets.  It is actually a
    \emph{separated} presheaf, as follows.

    A descent datum for $\Ob\sk\grg$ is given by a pair $(x,
    V_\bullet)$, where $V_\bullet \to U$ is a (hyper)cover of
    $U\in \Ob \s$, and $x$ is an object of $\sk\grg$ over $V_0$
    such that $d_0^*x=d_1^*x$ over $V_1$.  This defines a descent
    datum for $\grg$ (with identity maps as morphisms), so that
    there exists an object $y$ of $\grg$ over $U$ such that its
    pullback to $V_0$ is isomorphic to $x$. The object $y$ is
    defined up to isomorphism (in $\grg_U$), which shows that if
    it is in $\sk\grg_U$, then it must be unique. (Incidentally,
    this argument also shows why we ought not expect $\Ob\sk\grg$
    to be a sheaf.)

    We define $X$ to be the associated sheaf: since $\Ob\sk\grg$
    is separated to begin with, the ``plus'' construction will
    have to be done only once. Define $\pi_0$ to be the extension
    of $\imath$ to $X$. It exists, because an element of $X(U)$
    is given by a pair $\xi=(x,V_\bullet)$ as above, with gluing
    object $y\in \Ob\grg_U$, so we set $\pi_0(\xi)= y$.
  \end{proof}
  For any set $S$, let $F\langle S\rangle$ denote the free group
  over $S$.  Define $G_0$ as
  \begin{equation*}
    G_0 = F\langle X\rangle\sptilde,
  \end{equation*}
  where the tilde denotes sheafification. In other words, the
  right-hand side is the sheafification of the presheaf
  $U\rightsquigarrow F\langle X(U)\rangle$.
  \begin{lemma}
    \label{lem:2}
    The map $\pi_0$ extends to an additive functor $\pi\colon G_0
    \to \grg$.
  \end{lemma}
  \begin{proof}[Proof of the lemma]
    One follows the same pattern used to show the free group
    $F\langle S\rangle$ has the universal property with respect
    to set morphisms from $S$ to groups. Since the group law of
    $\grg$ is only a weak one in general, an ordering problem
    arises.  Locally, given a word $x_1\dotsm x_n$ where
    $x_1,\dotsc, x_n$ are element of $X(U)$ over some $U\in
    \Ob\s$, we define $\pi (x_1\dots x_n)$ by choosing a specific
    nesting of parentheses and then using the group law of
    $\grg$. Specifically, we set:
    \begin{equation*}
      \pi_0 ( x_1\dots x_n) \eqdef  (\dotsm ((x_1 \otimes x_2)\otimes
      x_3)\otimes \dotsm \otimes x_n)
    \end{equation*}
    associating from the left. That this is well defined follows
    from results of Laplaza about coherence in
    gr-categories---cf.~\cite{MR723395}.
  \end{proof}

  $G_1$ is defined by the square
  \begin{equation*}
    \xymatrix{%
      G_1 \ar[d] \ar[r]^{\del} & G_0 \ar[d]^{\pi} \\
      \mathbf{1} \ar[r] & \grg
    }
  \end{equation*}
  of gr-stacks, where $\mathbf{1}\to \grg$ corresponds to the
  unit object ($\mathbf{1}$ is the category with one object and
  one arrow).  Thus, the elements of $G_1$ consist of pairs
  $(x,\alpha)$ where $x\in G_0$ and $\alpha\colon I\to \pi (x)$
  in $\grg$. The map $\del$ is the projection to $G_0$ sending
  $g=(x,\alpha)$ to $x$.

  The multiplicative structure of $G_1$ is given by
  \begin{equation*}
    (x,\alpha ) (y,\beta) = (xy, \alpha \beta ),
  \end{equation*}
  where $\alpha\beta$ is to be interpreted as the composite arrow
  \begin{equation*}
    I \lisoto I\otimes I
    \xrightarrow{\alpha \otimes \beta}
    \pi (x) \otimes \pi(y) \lisoto \pi (xy).
  \end{equation*}
  The unit of $G_1$ is the pair $(1,\mu)$, where $\mu$ is the
  morphism such that
  \begin{equation*}
    \mu\colon I \lisoto \pi (1),
  \end{equation*}
  and $1$ is the unit element of $G_0$.  The inverse of the pair
  $(x,\alpha)$ is given by the pair $(x^{-1},\hat\alpha )$, where
  $\hat\alpha$ is the composite
  \begin{equation*}
    I \xrightarrow{(\alpha^*)^{-1}}
    \pi (x)^* \lisoto \pi(x^{-1}).
  \end{equation*}
  (Recall that the choice of a quasi-inverse in a
  gr-category---and therefore in a gr-stack---gives a functor
  $*\colon \grg^{\mathrm{op}}\to \grg$.)  With these definitions
  the map $\del$ evidently is a group homomorphism.

  Let us define an action of $G_0$ on $G_1$ by setting:
  \begin{equation}
    \label{eq:40}
    (y,\beta)^x = (x^{-1}yx, x^{-1} \beta x).
  \end{equation}
  Here $x\in G_0$, $(y,\beta)\in G_1$, and $x^{-1}\beta x$ is
  defined to be the composite
  \begin{equation*}
    I \lto
    \pi(x)^* \otimes \pi(x) \lto
    (\pi(x)^* \otimes I)\otimes \pi(x) \lto
    (\pi(x)^* \otimes \pi (y))\otimes \pi(x)\lisoto
    \pi (x^{-1} y x)
  \end{equation*}
  where the star denotes the inverse operation in $\grg$, and the
  last arrow on the right is in turn the composite of
  \begin{equation*}
    (\pi(x)^* \otimes \pi (y))\otimes \pi(x) \iso
    (\pi(x^{-1}) \otimes \pi (y))\otimes \pi(x) \iso
    \pi(x^{-1}y)\otimes \pi(x) \iso
    \pi (x^{-1}yx).
  \end{equation*}
  The arrow in the middle in the definition of $x^{-1}\beta x$ is
  \begin{math}
    (\pi(x)^* \otimes \beta )\otimes \pi(x).
  \end{math}
  Had we chosen to use
  \begin{math}
    \pi(x)^* \otimes ( \beta \otimes \pi(x))
  \end{math}
  instead, then the commutativity of~\eqref{eq:5} in any
  gr-category, and the functoriality of the associator, ensure
  the definition of $x^{-1}\beta x$ is unaffected.

  We must verify that the two axioms~\eqref{eq:7} of a crossed
  module hold for the action~\eqref{eq:40}. The first,
  \begin{equation*}
    \del (h^x) = x^{-1} \del h\,x
  \end{equation*}
  with $h=(y,\beta)$, is obvious. For the second, namely
  \begin{equation*}
    h^{\del g} = g^{-1} h\,g,
  \end{equation*}
  with $g=(x,\alpha)$ and $h=(y,\beta)$, to hold, we must have
  \begin{equation*}
    \hat\alpha\beta\alpha = x^{-1}\beta x.
  \end{equation*}
  To see why this is true, consider the following diagram:
  \begin{equation*}
    \xymatrix{%
      I \ar[r]^\sim &
      I \otimes I \ar[r]^\sim
      \ar[d]|{(\alpha^*)^{-1} \otimes \alpha} &
      I \otimes (I \otimes I)
      \ar[r]^{(\alpha^*)^{-1} \otimes (\beta\otimes \alpha)}
      \ar[d]|{(\alpha^*)^{-1} \otimes (I\otimes \alpha)} &
      **[r]\pi (x)^* \otimes ( \pi(y) \otimes \pi(x)) \\
      & \pi (x)^* \otimes \pi(x) \ar[r]
      & \pi (x)^* \otimes ( I \otimes \pi(x))
      \ar@(r,d)[ur]_(.7){\;\;\;\pi(x)^* \otimes (\beta\otimes \pi(x))} }
  \end{equation*}
  The relevant part of $\hat\alpha\beta\alpha$ is the composition
  in the top horizontal line.  The left square is commutative due
  to functoriality and the definition of $\alpha^*$. The
  composition of arrows at the bottom gives $x^{-1}\beta x$.

  Finally, the equivalence between $\grg$ and $[G_1\to
  G_0]\sptilde$ is obvious: full faithfulness is built-in in the
  definition of $G_1$, whereas essential surjectivity follows from
  the definition of $\pi$.

  This concludes the proof of the proposition.
\end{proof}
\begin{remark}
  An entirely similar proof in the case of strictly Picard
  gr-stacks is due to Deligne, cf.\ ref.\ \cite[Lemme 1.4.13
  i]{0259.14006}.  Breen and Messing have sketched a proof for
  the general case, using a rigidification of the group law on
  the simplicial object determined by $\grg$, cf.~\cite[Appendix
  B]{MR2183393}.
\end{remark}

\subsection{Derived category of crossed modules}
\label{sec:deriv-categ-cross}

This section is devoted to some remarks and analogies with
standard (abelian) homological algebra.

The first observation is that butterflies really \emph{are}
fractions.

Let $[H_\bullet, E,G_\bullet]$ be a butterfly.  It follows
immediately from its properties, and explicitly from
\cite{Noohi:weakmaps}, that
\begin{equation*}
  E_\bullet \colon
  \xymatrix@1{%
    H_1\times G_1 \ar[r]^>>>>{\kappa\cdot \imath }& E}
  \qquad
  (h,g) \lmto \kappa (g) \, \imath (g),
\end{equation*}
is a crossed module. The butterfly thus gives rise to a
``fraction,'' that is, a diagram of \emph{strict} morphisms of
crossed modules
\begin{equation*}
  \xymatrix@1{%
    H_\bullet & E_\bullet \ar[l]_\sim^{Q} \ar[r]^{P}
    & G_\bullet}
\end{equation*}
or, explicitly,
\begin{equation*}
  \xymatrix{%
    H_1 \ar[d]_\del &
    H_1\times G_1 \ar[l]_{\mathrm{pr}_1}
    \ar[d]_{\kappa\cdot \imath} \ar[r]^>>>>{\mathrm{pr}_2}
    & G_1 \ar[d]^\del \\
    H_0 & E \ar[l]^\pi \ar[r]_\jmath & G_0
  }
\end{equation*}
The strict morphism $Q$ is a quasi-isomorphism, inducing an
isomorphism of homotopy groups, as it can be readily checked.

The following is very easy, but it is worthwhile to point the
statement out.
\begin{lemma}
  \label{lem:5}
  $F\colon \grh\to \grg$ is a quasi-isomorphism if and only if it
  is an equivalence.
\end{lemma}
\begin{proof}
  Since the homotopy groups are the same, following the ideas
  of~\cite[\S 7]{MR95m:18006}, we can consider $F$ as a morphism
  over a space (=sheaf) $\pi_0$. Both $\grh$ and $\grg$ are
  $\pi_1$-gerbes over $\pi_0$, hence they must be equivalent.

  The converse is obvious.
\end{proof}

It follows from the lemma that if we call $\grstack{E}$ the
gr-stack associated to the crossed  module $E_\bullet$, the
induced morphism $\grstack{E}\to \grh$ is an equivalence. Thus,
analogously to the abelian situation in~\cite{0259.14006} $F$
factorizes as $F=P\circ Q^*$, where $Q^*$ is a quasi-inverse to $Q$
(as morphisms of gr-stacks).

Continuing the analogy with \loccit, let us denote by
$\twocat{Gr\mbox{-}Stacks}^\flat(\s)$ the category whose objects are the
gr-stacks of $\s$ and whose morphisms are isomorphism classes of
additive functors.  Similarly, $\CM^\flat (\s)$ will denote the
category whose objects are crossed modules and whose morphisms
are isomorphism classes of butterflies. Then Theorem~\ref{thm:5}
implies that there is an equivalence of categories
\begin{equation*}
  \CM^\flat (\s) \lisoto \twocat{Gr\mbox{-}Stacks}^\flat(\s).
\end{equation*}

\section{Exact sequences of gr-stacks and homotopy groups}
\label{sec:exact-sequences-gr}

In the set-theoretic context, given a butterfly
\begin{equation*}
  \xymatrix@R-0.5em{%
    H_1\ar[dd]_\del \ar@/_0.1pc/[dr]^\kappa  & &
    G_1 \ar@/^0.1pc/[dl]_\imath \ar[dd]^\del\\
    & E\ar@/_0.1pc/[dl]_\pi \ar@/^0.1pc/[dr]^\jmath &  \\
    H_0 & & G_0
  }
\end{equation*}
the complex $F_\bullet$ on the NW-SE diagonal participates in the
long exact sequence of homotopy groups:
\begin{equation*}
  1 \to \pi_2 (\twocat{F}) \to \pi_1 (\grpd{H}) \to \pi_1 (\grpd{G})
  \to \pi_1 (\twocat{F}) \to \pi_0 (\grpd{H}) \to \pi_0 (\grpd{G}) \to
  \pi_0(\twocat{F}) \to 1,
\end{equation*}
where $\grpd{H}, \grpd{G}$ are the groupoids corresponding to
$H_\bullet$ and $G_\bullet$, respectively, and $\twocat{F}$ is a
2-groupoid built from the complex $F_\bullet$, cf.\
\cite{Noohi:weakmaps}. In all instances these homotopy groups
coincide with the naïve (non-abelian) homology groups of the
complexes.  It is shown in \loccit that there is homotopy fiber
sequence of simplicial objects
\begin{equation*}
  \smp{F}_\bullet \lto \smp{H}_\bullet\lto \smp{G}_\bullet.
\end{equation*}
All terms are the nerves of the corresponding (2-)groupoids: for
$\smp{H}_\bullet$ and $\smp{G}_\bullet$ this was briefly recalled
in sect.~\ref{sec:cocycles}. They also are simplicial groups. For
2-groupoids, see~\cite{MR1239560}. Note that $\smp{F}_\bullet$ is
\emph{not} a simplicial group.

We give a geometric account of this circle of ideas in the
sheaf-theoretic case, based on the correspondence between weak
morphisms and morphisms of gr-stacks.

\subsection{Homotopy kernel}
\label{sec:homotopy-kernel}

Let $F\colon \grh\to \grg$ be a morphism of gr-stacks. The
homotopy kernel of $F$, $\grk = \shcatKer (F)$, is a gr-stack
defined by the (stack) fiber product:
\begin{equation*}
  \xymatrix{%
    \grk \drtwocell<\omit>{\epsilon}
    \ar[r]^J \ar[d] & \grh \ar[d]^F \\
    \mathbf{1} \ar[r] & \grg
  }.
\end{equation*}
We have explicitly marked the 2-morphism $\epsilon \colon F\circ
J \Rightarrow I$, where $I$ is the null functor, sending every
object $I$ of $\grk$ to the unit object of $\grg$, and every
morphism to the identity morphism of the unit object.
Explicitly, the objects of $\grk$ are given by pairs $(Y,f)$,
where $Y$ is an object of $\grh$, and $f\colon F(Y)\isoto
I_\grg$. A morphism $a\colon (Z,g)\to (Y,f)$ is given by a
morphism $a\colon Z\to Y$ in $\grh$ such that the diagram
\begin{equation*}
  \xymatrix{%
    F(Z) \ar@/_/[dr]_g \ar[r]^{F(a)} & F(Y) \ar[d]^f \\
    & I_\grg
  }
\end{equation*}
commutes. The multiplication law is given by:
\begin{equation*}
  (Y,f)\otimes_\grk (Z,g) = (Y\otimes_\grh Z, fg),
\end{equation*}
where $fg$ is the morphism obtained by composing $f\otimes_\grg
g$ with the obvious structural ones. In fact the construction of
the group $G_1$ in the proof of Lemma~\ref{lem:2} is but an
instance of homotopy kernel.  There will be no difficulty in
realizing that the multiplication and inverse laws of $\grk$ are
just the obvious translations of those already analyzed in detail
in that case.

Let $H_\bullet$ and $G_\bullet$ be crossed modules corresponding
to $\grh$ and $\grg$, respectively. Let $[H_\bullet, E,
G_\bullet]$ be the butterfly corresponding to the weak morphism
$H_\bullet\to G_\bullet$ determined by $F\colon \grh\to \grg$.
We may take $E=H_0\times_{\grg,F}G_0$.  According to
Proposition~\ref{prop:9}, there exists a crossed module such that
its associated gr-stack is equivalent to $\grk$.  More precisely,
consider the map $\jmath\colon E\to G_0$ for the butterfly
$[H_\bullet, E, G_\bullet]$.  We have:
\begin{proposition}
  \label{prop:10}
  The kernel $\grk$ of $F$ is equivalent to the gr-stack
  associated to the crossed module $[H_1\to
  \Ker\jmath]$. Moreover, $J\colon \grk\to \grh$ corresponds to
  the obvious \emph{strict} morphism of crossed modules
  \begin{equation*}
    \xymatrix{%
      H_1 \ar@{=}[r] \ar[d]_{\bar\kappa}  & H_1 \ar[d]^\del \\
      \Ker\jmath \ar[r] & H_0
    }
  \end{equation*}
\end{proposition}
\begin{proof}
  It is immediately checked that $[H_1\to \Ker\jmath]$ is a
  crossed module.
  
  Now, assume we are given an $(H_1,H_0)$-torsor $(Q,t)$ such
  that its image by $F$ is isomorphic to the unit object of
  $\grg$, which we identify with $(G_1,1)$ (the trivial
  $G_1$-torsor, equipped with the equivariant map $1\colon G_1\to
  G_0$).  Recall that the image of $(Q,t)$ by $F$ is the
  $G_1$-torsor $\shHom_{H_1}(Q,E)_t$ of local lifts of $t$ to
  $E$, equipped with the global equivariant map sending each
  local lift $e$ to $\jmath (e)$.  To say that this is isomorphic
  to the identity object $(G_1, 1)$ means that there exists a
  global lift $e\colon Q\to E$ such that $\jmath (e) =1$.  Thus
  the objects of $\grk$ are $(H_1,\Ker\jmath)$-torsors.  It is
  clear that these constructions are functorial.

  The second part of the proposition is obvious.
\end{proof}
The butterfly corresponding to $J$ (from the proposition) and the
one corresponding to $F$ are composed according to the
prescription of~\ref{sec:spec-cases-comp}.  It is immediately
verified that the result, which corresponds to $F\circ J$, is
globally split---the splitting homomorphism $s\colon \ker\jmath
\to \ker\jmath \times_{H_0} E$ is simply the diagonal---and
moreover it has the two properties in Lemma 10.3
of~\cite{Noohi:weakmaps}. Hence it corresponds to the trivial
morphism $1\colon [H_1\to \ker\jmath] \to [G_1\to G_0]$.

\subsection{Exact sequences}
\label{sec:exact-sequences}
\begin{paragr}
  The sequence
  \begin{equation}
    \label{eq:41}
    \grk \overset{G}{\lto} \grh \overset{F}{\lto} \grg
  \end{equation}
  of morphisms of gr-stacks is a \emph{complex} if there is a
  2-morphism $\epsilon$ from $F\circ G$ to the null-functor:
  \begin{equation*}
    \xymatrix{%
      \grk \ar[r]^G \ar@/_1.9pc/[rr]_I
      \rrlowertwocell<\omit>{<3>\epsilon}
      & \grh \ar[r]^F & \grg
    }
  \end{equation*}
\end{paragr}
This idea of a complex of gr-stacks is standard. (In the
set-theoretic case, compare ref.\ \cite{MR1935985}.) For
geometric applications, see ref.\
\cite{doi:10.1016/j.jpaa.2007.07.020}.  From~\eqref{eq:41} there
results a functor:
\begin{equation*}
  \bar G\colon \grg\lto
  \shcatKer (F)
\end{equation*}
defined by sending $Y\in \Ob\grg_U$ to the pair
$(G(Y),\epsilon_Y)$.  While obvious, this allows to formulate the
notion of exactness in the middle of the sequence~\eqref{eq:41}
as follows:
\begin{definition}
  \label{def:12}
  The sequence~\eqref{eq:41} is exact at $\grh$ if the functor
  $\bar G$ is full and essentially surjective.
\end{definition}
This definition can be found, for instance, in ref.\
\cite{MR1935985} in the context of gr-categories.  The
formulation for gr-stacks is the same, while taking care that
``full'' and ``essentially surjective'' must be intended in the
appropriate context.

For the definition of ``short exact,'' still in the context of
gr-categories, see refs.\ \cite{MR93k:18019,MR2072344}.  We
repeat it in the context of gr-stacks
\begin{definition}
  \label{def:13}
  The sequence~\eqref{eq:41} is:
  \begin{itemize}
  \item \emph{left exact} if it is exact at $\grh$ and $\bar
    G\colon \grk \to \shcatKer (F)$ is an equivalence;
  \item an extension of gr-stacks if it is both left exact and
    $F$ is essentially surjective.
  \end{itemize}
\end{definition}
We leave out the fibration condition found in \cite{MR92m:18019},
see section~\ref{sec:exact-sequence-non} below for more details.
\begin{paragr}
  Recall that if $\grg$ is a gr-stack we define $\pi_1 (\grg) =
  \shAut_{\grg}(I)$ and $\pi_0(\grg)$ as the sheaf associated to
  the presheaf $U\rightsquigarrow \Ob\grg_U/\sim$.  Given the
  sequence~\eqref{eq:41} one obtains corresponding sequences
  \begin{equation}
    \label{eq:42}
    \pi_i (\grk) \lto \pi_i(\grh) \lto \pi_i(\grg),\qquad i=0,1.
  \end{equation}
\end{paragr}
We have the following easy lemma (cf.\ \cite{MR1935985}).  We
sketch its proof anyway, in view of subsequent applications and
the fact it is being formulated for gr-stacks, as opposed to
gr-categories, as in \loccit
\begin{lemma}
  \label{lem:3}
  If~\eqref{eq:41} is exact at $\grh$, the
  sequences~\eqref{eq:42} are exact at $\pi_i(\grh)$, $i=0,1$.
\end{lemma}
\begin{proof}
  For $i=1$, we have that if $f$ is an automorphism of $I_\grh$
  over $U$, then its image $\pi_1(F)(f)$ is defined by:
  \begin{equation*}
    \xymatrix{%
      F(I_\grh) \ar[d]_{\eta_F} \ar[r]^{F(f)} & F(I_\grh) \ar[d]^{\eta_F} \\
      I_\grg \ar@{.>}[r]_{\pi_1(F)(f)} & I_\grg
    }
  \end{equation*}
  where $\eta_F$ is the morphism resulting from the additivity.
  Therefore, if $\pi_1(F)(f)$ is the identity, then by definition
  $f$ is an automorphism of $(I_\grh,\eta_F)$, considered as an
  object of $\shcatKer (F)$. Since $\bar G$ is full, there exists
  a generalized cover $V$ and an automorphism $g$ of $I_\grk$
  over $V$ such that $\pi_1(g)$ becomes equal to $f\rvert_V$ via
  $G(I_\grk)\isoto (I_\grh,\eta_F)$, as wanted.

  For $i=0$, let $\xi$ be a section of $\pi_0(\grh)$ over $U\in
  \Ob\s$. We can assume, up to refining $U$, that $\xi = [X]$,
  where $X$ is an object of $\grh_U$. Then
  $\pi_0(F)(\xi)=[F(X)]$.  If this is equal to $1\in
  \pi_0(\grg)$, then we must have that there is an isomorphism
  $f\colon F(X)\isoto I_\grg$, or in other words, $(X,f)$ is an
  object of $\shcatKer (F)_U$.  Since $\bar G$ is essentially
  surjective, there exists a (generalized) cover $V\to U$ and an
  object $Y$ of $\grk_V$ such that $\bar G (Y)\isoto
  (X,f)\rvert_V$.  So this means there exists $a\colon G(Y)\isoto
  X$ in $\grh_V$ such that
  \begin{equation*}
    \xymatrix{%
      F(G(Y)) \ar[r]^{F(a)}
      \ar@/_/[dr]_{\epsilon_Y}
      & F(X) \ar[d]^d \\
      & I_\grg
    }
  \end{equation*}
  Thus $[X]\rvert_V = [G(Y)] = \pi_0(G)([Y])$.
\end{proof}
\begin{proposition}
  \label{prop:11}
  If the sequence~\eqref{eq:41} is left-exact, then there is a
  connecting homomorphism $\Delta \colon \pi_1(\grg)\to \pi_0
  (\grk)$ leading to a long exact sequence
  \begin{equation*}
    0 \lto \pi_1 (\grk) \lto \pi_1(\grh) \lto \pi_1(\grg)
    \overset{\Delta}{\lto}
    \pi_0 (\grk) \lto \pi_0(\grh) \lto \pi_0(\grg).
  \end{equation*}
\end{proposition}
\begin{proof}
  We follow the set-theoretic arguments of ref.\
  \cite{MR2072344}.

  Observe that there exists a \emph{functor}
  \begin{equation*}
    \Delta \colon \pi_1(\grg) \lto \shcatKer (F)
  \end{equation*}
  defined as follows.  Any $g\in \shAut_\grg (I)(U)$ is sent to
  $(I_\grh,g\circ \eta_F)$.  (See the proof of the previous lemma
  for the meaning of symbols.) We claim that the sequence of
  gr-stacks
  \begin{equation*}
    0 \lto \pi_1 (\grk) \lto \pi_1(\grh) \lto \pi_1(\grg)
    \overset{\Delta}{\lto}
    \grk \lto \grh \lto \grg
  \end{equation*}
  is exact in the sense of Definition~\ref{def:12}.  (The group
  objects in the sequence are considered as gr-stacks in the
  obvious way.)

  First of all, it is clear that the homomorphism $\pi_1 (\grk)
  \to \pi_1(\grh)$ is an injection, and we need to check
  exactness only at $\pi_1(\grg)$ and $\grk$. We may assume that
  $\grk = \shcatKer (F)$.

  Exactness at $\pi_1(\grg)$ holds because if $\Delta (g) \isoto
  (I_\grh,\eta_F)$, this means there exists $f\colon I_\grh\to
  I_\grh$ over $U$ such that
  \begin{equation*}
    \xymatrix{%
      F(I_\grh) \ar[r]^{F(f)} \ar[d]_{\eta_F} & F(I_\grh)
      \ar[d]^{\eta_F} \\
      I_\grg \ar[r]_g & I_\grg
    }
  \end{equation*}
  commutes, which by definition means $g$ is the image of $f$ by
  $\pi_1(F)$.

  Exactness at $\grk$ holds because if $(X,f)$ becomes isomorphic
  to $I_\grh$ in $\grh$, this means that there is already an
  isomorphism $a\colon X\isoto I_\grh$, and so the diagram
  \begin{equation*}
    \xymatrix{%
      F(I_\grh) \ar[r]^{F(a)} \ar[d]_{f} & F(I_\grh)
      \ar[d]^{\eta_F} \\
      I_\grg & I_\grg \ar@{.>}[l]^g
    }
  \end{equation*}
  defines the required automorphism of $I_\grg$.

  Having established the exactness of the above sequence, we need
  only apply lemma \ref{lem:3} to it to obtain the sequence in
  the statement.
\end{proof}
\begin{remark}
  The sequence in the proposition is exact at the rightmost place
  iff $F$ is essentially surjective.
\end{remark}

\subsection{The homotopy fiber of a butterfly}
\label{sec:cone-butterfly}

Let us return to the morphism $F\colon \grh\to \grg$ and the
corresponding butterfly, which we rewrite for convenience:
\begin{equation*}
  \xymatrix@R-0.5em{%
    H_1\ar[dd]_\del \ar@/_0.1pc/[dr]^\kappa  & &
    G_1 \ar@/^0.1pc/[dl]_\imath \ar[dd]^\del\\
    & E\ar@/_0.1pc/[dl]_\pi \ar@/^0.1pc/[dr]^\jmath &  \\
    H_0 & & G_0
  },
\end{equation*}
Let us denote by $\grk$ the homotopy kernel of $F$.  From
Proposition~\ref{prop:10} it follows that
\begin{align*}
  \pi_1 (\grk) & = \Ker (\bar\kappa\colon H_1\to \Ker\jmath) \\
  \intertext{and} \pi_0 (\grk) & = \Coker (\bar\kappa \colon
  H_1\to \ker\jmath)
\end{align*}
respectively coincide with the non-abelian cohomology sheaves
$\shH^{-2}(F_\bullet)$ and $\shH^{-1}(F_\bullet)$, where
$F_\bullet$ is the complex
\begin{equation*}
  \left[
    H_1\lto E\lto G_0 \right]_{-2,0}
\end{equation*}
placed in degrees $[-2,0]$.  Moreover, $\pi_i (\grg) =
\shH^{-i}(G_\bullet)$, $i=0,1$, and similarly for $H_\bullet$.
By the previous discussion we have the exact sequence of homology
sheaves:
\begin{equation*}
  0\lto \shH^{-2}(F_\bullet) \lto \shH^{-1}(H_\bullet) \lto \shH^{-1}
  (G_\bullet) \lto \shH^{-1}(F_\bullet) \lto \shH^{0}(H_\bullet)
  \lto \shH^{0} (G_\bullet).
\end{equation*}
Again, the above sequence will be exact on the right if the
butterfly corresponds to an essentially surjective morphism.
Otherwise, it is easy to see that the obstruction is
$\shH^0(F_\bullet)$.

The complex $F_\bullet$ itself begs for an explanation akin to
the one given in~\cite{Noohi:weakmaps}, but more in keeping with
the present geometric context.  We conclude this section by
providing the construction of a 2-stack which, in a somewhat
imprecise sense, takes the rôle of the cone in an exact triangle.
The discussion in the rest of this section requires more
background than we were able to provide in the rest of this work,
especially concerning 2-(gr)-stacks and their relations with
complexes of length 3, for which we refer to one of the sequels
of this work (\cite{ButterfliesIII,ButterfliesIV}).  Here we will
simply gloss over several details of the construction, and refer
to the quoted references for the necessary background and
additional details not explicitly included here.
\begin{paragr}
  Following~\cite{Noohi:weakmaps}, there is a (sheaf of)
  2-groupoid(s), call it $\twocat{F}$, determined by
  $F_\bullet$. Its objects are given by the sections of $G_0$,
  and in general a 2-cell is
  \begin{equation*}
    \xymatrix@+1pc{%
      x \rtwocell<5>^{e_0}_{e_1}{h} & x'
    }
  \end{equation*}
  with $x,x'\in G_0$, $e_0, e_1\in E$, and $h\in H_1$, satisfying
  the relations
  \begin{equation*}
    e_1 = e_0\, \kappa (h),\qquad
    x' = x\, \jmath (e_0) = x\, \jmath (e_1) .
  \end{equation*}
  The vertical composition of 2-arrows is given by multiplication
  in $H_1$.  For the horizontal composition we have
  \begin{equation*}
    \xymatrix@+1pc{%
      x \rtwocell<5>^{e_0}_{e_1}{h} &
      x' \rtwocell<5>^{e'_0}_{e'_1}{\phantom{x}h'} & x''
    }=
    \xymatrix@C+2pc{%
      x \rtwocell<5>^{e_0e'_0}_{e_1e'_1}{\phantom{x}h''} & x''
    }
  \end{equation*}
  with $h'' = h^{\pi(e'_0)}h'$.  This follows at once from the
  properties of the butterfly.  Observe that this 2-groupoid is
  strict, in the sense that the 1-morphisms are strictly
  invertible, as opposed to just being equivalences.
\end{paragr}
\begin{paragr}
  The homomorphism $\kappa\colon H_1\to E$ defines a crossed
  module, where the action of $E$ on $H_1$ occurs via $\pi\colon
  E\to H_0$.  Let $\grc$ be the associated gr-stack:
  \begin{equation*}
    \grc = \left[ H_1\overset{\kappa}{\lto} E\right ]\sptilde
    \iso \tors (H_1,E).
  \end{equation*}
  An object $(P,e)$ of $\grc$ is therefore an $H_1$-torsor
  equipped with an $H_1$-equivariant map $e\colon P\to E$. It
  follow at once that $\jmath\circ e$ is invariant under the
  action of $H_1$, hence it is an element of $G_0$. Moreover, for
  a morphism $\phi\colon (P,e)\to (P',e')$ in $\grc$ the
  corresponding element in $G_0$ is the same: $\jmath\circ e =
  (\jmath\circ e')\circ \phi$.  This can be summarized by saying
  that there is an additive functor
  \begin{equation*}
    J\colon \grc \lto G_0 
  \end{equation*}
  sending the object $(P,e)$ to $\jmath (e)$, where, again, $G_0$
  is regarded as a gr-stack in the obvious way.
\end{paragr}
Incidentally, it is easy to verify that the homotopy kernel of
$J$ is again $\grk$. There is also a morphism of gr-stacks
$\pi\colon \grc\to \grh$ induced by the obvious corresponding
strict morphism of crossed modules $[H_1\to E]\to [H_1\to H_0]$
contained in the butterfly.  The (homotopy) kernel of the latter
is $G_1$, as it is immediately checked. Thus we have the
following proposition, whose proof will be left to the reader,
giving a stacky dévissage of the butterfly:
\begin{proposition}
  \label{prop:14}
  The butterfly $E$ from $H_\bullet$ to $G_\bullet$ gives rise to
  the commutative diagram of gr-stacks
  \begin{equation}
    \label{eq:55}
    \vcenter{%
      \xymatrix{%
        & \grk \ar@{=}[r] \ar[d] & \grk \ar[d] \\
        G_1 \ar[r] \ar@{=}[d] & \grc \ar[r]^\pi \ar[d]^J & \grh \ar[d]^F \\
        G_1 \ar[r]_\del & G_0 \ar[r]_{\pi_\grg} & \grg
      }
    }
  \end{equation}
\end{proposition}
In order to proceed any further we need to introduce the notion
of torsor over a gr-stack and that of $(\grh, \grg)$-torsor, for
$F\colon \grh\to \grg$ is a morphism of gr-stacks.

If $\grh$ is a gr-stack, the notion of (right) $\grh$-torsor is
expounded in detail in\ \cite{MR92m:18019}, and essentially
categorifies that of torsor over a sheaf of groups.  Thus, a
right $\grh$-torsor is a stack $\stx$ equipped with a right
$\grh$-action, and it is locally non-empty in a way that makes it
locally equivalent to $\grh$.  Morphisms of $\grh$-torsors are
morphisms of the underlying stacks that weakly commute with the
$\grh$-action, that is, up to coherent natural
transformation. Similarly, 2-morphisms are again 2-morphisms
between the underlying stacks that are compatible with the
$\grh$-action.  We denote by $\tors (\grh)$ the fibered
2-category of $\grh$-torsors over $\s$. In fact it is a (neutral)
2-gerbe over $\s$.
\begin{paragr}
  A $(\grh, \grg)$-torsor is the categorification of the concept
  of $(A,B)$-torsor, for a homomorphism $A\to B$ of group
  objects, and its definition was given in ref.\
  \cite{doi:10.1016/j.jpaa.2007.07.020}---albeit in a fully
  abelian context.  A $(\grh, \grg)$-torsor is a pair $(\stx, S)$
  where $\stx$ is a (right) $\grh$-torsor equipped with a
  $\grh$-equivariant functor
  \begin{equation*}
    S\colon \stx \lto \grg.
  \end{equation*}
  This is equivalent to saying that $\stx$ is equipped with a
  trivializing equivalence $S\colon \stx\cprod{\grc} \grg\isoto
  \grg$.  (The contracted product for stacks equipped with an
  action by a gr-stack is defined in \cite{MR92m:18019}.)

  The notion of morphism of $(\grh,\grg)$-torsors is the one
  obtained by taking the obvious generalization
  (=categorification) of~\eqref{eq:24}. A morphism of
  $\grh$-torsors is a pair $(F,\mu)$ such that $F\colon \stx \to
  \sty$ is a morphisms of stacks and $\mu$ is the natural
  transformation expressing the weak compatibility of $F$ with
  the torsor structures. We further require that the diagram
  \begin{equation}
    \label{eq:58}
    \vcenter{%
      \xymatrix{%
        \stx \ar[r]^F \ar@/_/[dr]_S^(0.6){}="S" 
        & \sty \ar[d]^T_(0.4){}="T" \\
        & \grg \ar@{=>}@/_0.3pc/ "T";"S"_{\lambda_F}
      }}
  \end{equation}
  of $\grh$-equivariant morphisms 2-commutes.  Note there is no
  additional condition on $\mu$. Finally, a 2-morphism
  $\alpha\colon (F, \mu)\Rightarrow (G,\nu)\colon \stx\to \sty$
  is a 2-morphism of $(\grh,\grg)$-torsors if the natural
  transformation $\alpha$, in addition to satisfying the
  properties required in \cite[6.1.7]{MR92m:18019}, also fits in
  the diagrammatic equality of 2-morphisms
  \begin{equation}
    \label{eq:59}
    \vcenter{%
      \xymatrix@C+1pc{%
        \stx \rtwocell^F_G{\alpha}
        \ar@/_1pc/[dr]_S^(0.6){}="S" 
        & \sty \ar[d]^T_(0.4){}="T" \\
        & \grg \ar@{=>}@/_/ "T";"S"^{\lambda_G} }}
    =
    \vcenter{%
      \xymatrix@C+1pc{%
        \stx \ar[r]^F \ar@/_1pc/[dr]_S^(0.6){}="S"
        & \sty \ar[d]^T_(0.4){}="T" \\
        & \grg \ar@{=>}@/_/ "T";"S"_{\lambda_F} }}
  \end{equation}
\end{paragr}
$(\grh, \grg)$-torsors comprise a 2-stack denoted $\tors
(\grh,\grg)$.  Actually more is true: given the morphism $F\colon
\grh\to \grg$ of gr-stacks there is an induced functor
\begin{equation*}
  F_* \colon \tors (\grh) \lto \tors (\grg)
\end{equation*}
that sends the $\grh$-torsor $\stx$ to
$\stx\cprod{\grh}\grg$. This is the analog---or
``categorified''---notion of extension of the structural group
for ordinary torsors.  Then $\tors (\grh,\grg)$ is the ``homotopy
fiber'' of $F_*$. More precisely we have:
\begin{lemma}
  \label{lem:4}
  There is a pull-back diagram
  \begin{equation*}
    \xymatrix{%
      \tors (\grh,\grg) \ar[r] \ar[d] & \tors (\grh) \ar[d]^{F_*} \\
      \twocat{\mathbf{1}} \ar[r] & \tors (\grg)
    }
  \end{equation*}
  in the sense of fibered product for 2-categories, as in
  \cite{MR0364245}.
\end{lemma}
In the lemma we have indicated with $\twocat{\mathbf{1}}$ the
2-category with only one object and identity (2-)morphisms.
\begin{proof}[Proof of the lemma]
  The morphism $\twocat{\mathbf{1}}\to \tors (\grg)$ sends the
  unique object of $\twocat{\mathbf{1}}$ to the trivial
  $\grg$-torsor.  According to \cite{MR0364245}, an object of the
  fibered product
  \begin{equation*}
    \twocat{\mathbf{1}} \times_{\tors (\grg)}\tors (\grh)
  \end{equation*}
  is by definition given by an $\grh$-torsor $\stx$, plus an
  equivalence $F_* (\stx)\iso \grg$.  This is by definition an
  $(\grh, \grg)$-torsor.
\end{proof}
\begin{paragr}
  Returning now to $\twocat{F}$, note that it is a 2-category
  fibered in 2-groupoids over $\s$. Moreover, this 2-category is
  separated in the sense that it has a \emph{sheaf} of
  2-morphisms.  It is promoted to a 2-prestack by promoting the
  categories of morphism from simply being isomorphic copies of
  the groupoids associated to the crossed module $[H_1\to E]$ to
  be equivalent to the gr-stack $\grc$.  This 2-prestack, which
  should be more properly called a \emph{bi-}prestack, since it
  turns out to be fibered in bicategories, is then made into a
  2-stack by an additional sheafification process.  This latter
  step rigidifies it again, so that it is an actual 2-stack.
\end{paragr}
Rather than describe the various stages in detail, let us give a
convenient model for $\tstf$, the 2-stack associated to the
2-groupoid $\twocat{F}$.  We claim that
\begin{equation*}
  \tstf \iso \tors (\grc, G_0).
\end{equation*}
Thus objects are pairs $(\stx, s)$, where $\stx$ is a
$\grc$-torsor and $s\colon \stx \to G_0$ is a $\grc$-equivariant
map via the morphism $J\colon \grc \to G_0$.

To have at least an intuitive idea of why $\tstf$ should be the
2-stackification of the 2-groupoid $\twocat{F}$ at all, one may
look into the 2-descent datum determined by an object $(\stx, s)$
of $\tstf$ over $U$, resulting in a 0-cocycle over $U$ with
values in $F_\bullet$.
\begin{paragr}
  Since a torsor is by definition locally trivial, there will
  exist an object $X$ of $\stx$ over a cover $V_0\to U$
  determining an equivalence $\stx\iso \grc$ over $V_0$.
  Moreover, applying $s$ we obtain an element $x=s(X)$ of
  $G_0$. Now, assuming we are provided with a $V_1$ such that
  $V_1 \rightrightarrows V_0$, for instance by working with a
  \Cech resolution $\cech (V_0\to U)$, the fact that $\stx$ is a
  $\grc$-torsor further implies there is a morphism over $V_1$:
  \begin{equation*}
    f\colon d_1^*X \isoto d_0^*X\cdot (P,e),
  \end{equation*}
  for an appropriate object $(P,e)$ of $\grc_{V_1}$.  Again,
  applying $s$ results in
  \begin{equation}
    \label{eq:43}
    d_1^*(x) = d_0^*(x)\, \jmath (e).
  \end{equation}
  The morphism $f$ will have to satisfy the 1-cocycle conditions
  given in~\cite{MR92m:18019}, which amount to the existence of a
  morphism of $(H_1,E)$-torsors
  \begin{equation*}
    \phi\colon d_2^*(P,e) \otimes d_0^*(P,e) \lto d_1^*(P,e)
  \end{equation*}
  over $V_2$ satisfying a compatibility condition over $V_3$,
  which are ultimately the conditions given in \cite[6.2.8 and
  6.2.10]{MR92m:18019}.  Upon choosing trivializations for the
  $E$-torsors compatible with the cover, the morphism $\phi$
  above yields the relation
  \begin{equation}
    \label{eq:56}
    d_2^*e\, d_0^*e = d_1^*e \, \kappa (h),
  \end{equation}
  where $h\in H_1 (V_2)$ arises from expressing the morphism of
  the underlying torsors in terms of the trivializations. The
  coherence condition results in the following relation over
  $V_3$:
  \begin{equation}
    \label{eq:57}
    d_2^* h\, d_0^* h = d_1^*h\, (d_3^* h)^{(d_0d_1)^*e}.
  \end{equation}
  Equations~\eqref{eq:43}, \eqref{eq:56} and~\eqref{eq:57} give
  the required cocycle.
\end{paragr}
\begin{paragr}
  The homotopy sheaves $\pi_i$, for $i=0,1,2$, are defined for a
  2-stack (cf.\ ref.\ \cite[Chap. 8]{MR95m:18006}).  In fact a
  2-stack $\tstf$ is a 2-gerbe over the localized site $\s
  \downarrow \pi_0(\tstf)$, as shown in \loccit, much in the same
  way as a stack is a gerbe over its own $\pi_0$-sheaf.
  Moreover, $\tstf$ as a 2-gerbe over $\s \downarrow \pi_0
  (\tstf)$ is neutral, as $I$, the trivial $(\grc, G_0)$-torsor,
  provides the necessary global object.  It follows that as a
  2-gerbe it is equivalent to $\tors (\shcatAut (I))$, and
  therefore it is classified by the invariants of the gr-stack
  $\shcatAut (I)$.
\end{paragr}
\begin{paragr}
  It can be readily verified by means of a local calculation that
  there is an equivalence of gr-stacks
  \begin{equation*}
    \shcatAut (I) \lisoto \grk= \left[ H_1 \overset{\bar\kappa}{\to}
      \ker\jmath \right]\sptilde.
  \end{equation*}
  Hence we can use the results in \cite[Chap. 8]{MR95m:18006} to
  conclude that
  \begin{equation*}
    \shH^{-2}(F_\bullet) = \pi_2(\tstf) = \pi_1 (\grk) ,
    \qquad \shH^{-1}(F_\bullet) = \pi_1(\tstf) = \pi_0 (\grk)\,
  \end{equation*}
  as wanted. Moreover, since $\tstf$ is the 2-stack associated to
  the 2-groupoid $\twocat{F}$, it follows that $\pi_0(\tstf) =
  \shH^0(F_\bullet)$.
\end{paragr}
The previous discussion is subsumed by part of the following
statement whose proof will be sketched below.
\begin{theorem}
  \label{thm:6}
  Let $F\colon \grh\to \grg$ be a morphism of gr-stacks, where
  $\grh= [H_1\to H_0]\sptilde$ and $\grg=[G_1\to G_0]\sptilde$,
  and let $[H_\bullet, E, G_\bullet]$ be the corresponding
  butterfly.

  There exists a 2-stack $\twostack{F}$ such that $\pi_i(\tstf) =
  \shH^{-i}(F_\bullet)$, $i=0,1,2$, where $F_\bullet$ is the
  complex on the NW-SE diagonal of the butterfly.  Moreover,
  there is a homotopy fiber sequence of 2-stacks:
  \begin{equation}
    \label{eq:44}
    \grh[0] \xrightarrow{F[0]} \grg[0]
    \overset{\iota}{\lto} \tstf
    \overset{\Delta}{\lto}
    \tors ({\grh}) \overset{F_*}{\lto} \tors ({\grg})
  \end{equation}
  giving rise to the exact sequence of homotopy sheaves
  \begin{equation*}
    \xymatrix{%
      0 \ar[r] & \pi_{2}(\twostack{F}) \ar[r]&
      \pi_{1}(\grh) \ar[r] & \pi_{1} (\grg) \ar[r] &
      \pi_{1}(\twostack{F}) \ar `r[d] `[lll] `[dlll] [dlll]\\
      & \pi_{0}(\grh) \ar[r] & \pi_{0} (\grg) \ar[r] & \pi_0
      (\twostack{F}) \ar[r] & 1
    }
  \end{equation*}
\end{theorem}
The various morphisms in the sequence are as follows.  The suffix
$[0]$ is used to indicate that the affected objects are to be
considered ``discrete'' 2-stacks with no non-trivial 2-morphisms,
so $F[0]$ is just $F$. $F_*$ is the ``push-forward'' of torsors
along $F$, introduced before.  The morphism $\Delta$ sends a
$(\grc, G_0)$ torsor $(\stx, s)$ to the $\grh$-torsor
$\stx\cprod{\grc} \grh$ via the morphism of gr-stacks $\grc\to
\grh$.  Finally, the morphism $\iota\colon \grg[0]\to \tstf$ can
be seen as the one induced by the 2-stackification of the strict
morphism of 2-groupoids $\grpd{G}[0]\to \twocat{F}$ occurring in
the butterfly.

The sequence is a ``homotopy fiber sequence'' in the sense that
each term in \eqref{eq:44}, starting from $\tstf$ and moving to
the left, can be understood as a pull-back (fiber product)
diagram of 2-stacks in the sense of \cite{MR0364245}.  For
example:
\begin{equation}
  \label{eq:54}
  \vcenter{%
    \xymatrix{%
      \tstf \ar[r] \ar[d] & \tors (\grh) \ar[d] \\
      \twocat{\mathbf{1}} \ar[r] & \tors (\grg)
    }}
\end{equation}
where $\twocat{\mathbf{1}}$ is the 2-category with one object and
only identity morphism and 2-morphism.  A similar consideration
holds for the other three-part segments.  Note also that the
resulting exact sequence involving the last three terms is well
defined, since all the three 2-stacks appearing in the sequence
are naturally pointed by the trivial torsor.  (The others are
2-groups, so they are naturally pointed too.)

\begin{proof}[Proof of Theorem~\ref{thm:6}]
  From Lemma~\ref{lem:4}, $\tors (\grh, \grg)$ can be taken as
  the homotopy fiber of the functor $F_*$.  Moreover, it is easy
  to see that it is part of the full homotopy fiber sequence:
  \begin{equation*}
    \xymatrix@1{%
      \grh[0] \ar[r]^{F[0]} &
      \grg[0] \ar[r]^<<<{\iota} &
      \tors (\grh,\grg) \ar[r]^{\Delta} &
      \tors ({\grh}) \ar[r]^{F_*} & \tors ({\grg})}
  \end{equation*}
  This time, $\Delta$ forgets the trivialization, whereas $\iota$
  sends a $(G_1,G_0$)-torsor $(P,s)$ to the $(\grh,\grg)$-torsor
  \begin{equation*}
    (\grh, (P,s))
  \end{equation*}
  where $(P,s)$ is to be considered as the $\grg$-equivariant
  functor sending the unit object $I_\grh$ of $\grh$ to $(P,s)$.

  The homotopy fiber sequence above can be extended to the left
  as follows:
  \begin{equation*}
    \xymatrix@1{%
      0\ar[r] & \pi_1 (\grk) \ar[r] & \pi_1 (\grh) \ar[r] &
      \pi_1(\grg) \ar[r] & \grk \ar[r] & \grh \ar[r] & \dotsm }
  \end{equation*}
  The (abelian) groups to the left are considered as 2-stacks in
  the obvious way: they are totally discrete 2-stacks with only
  identity arrows and 2-arrows.

  By applying $\pi_0$ to the combined sequence, and observing
  that $\pi_0 (\tors (\grh))$ and $\pi_0 (\tors (\grg))$ are
  trivial, we get the homotopy sequence in the statement, with
  $\tors (\grh, \grg)$ as the relevant 2-stack instead of
  $\tstf$.

  The statement then follows from the following proposition:
  \begin{proposition}
    \label{prop:13}
    There is a 2-equivalence (in the sense of ref.\
    \cite{MR0364245}) $G\colon \tstf \lisoto \tors (\grh, \grg)$.
  \end{proposition}
  \begin{proof}
    \renewcommand{\qedsymbol}{} The 2-functor $G$ is defined as
    follows. Let $(\stx,s)$ be a $(\grc, G_0)$-torsor. Then
    $\stx\cprod{\grc}\!\grh$ is an $\grh$-torsor.  We claim there
    is a global, $\grg$-equivariant, functor
    \begin{equation*}
      S\colon \stx\cprod{\grc}\!\grh \lto \grg
    \end{equation*}
    given by the global map
    \begin{equation*}
      (\pi_\grg\circ s) \wedge F \colon
      \stx\cprod{\grc}\!\grh \lto \grg,
    \end{equation*}
    so that the pair $(\stx\cprod{\grc}\!\grh, S)$ is an $(\grh,
    \grg)$-torsor. Since an object of $\stx\cprod{\grc}\!\grh$ is
    a pair $(X,Y)$, where $X$ and $Y$ are objects of $\stx$ and
    $\grh$, respectively,\footnote{$Y$ is what we would normally
      write as $(Q,t)$ when emphasizing the fact is is a
      $(H_1,H_0)$-torsor, which is not important, at the moment.}
    $S$ is defined by sending $(X,Y)$ to $\pi_\grg (s(X))
    \otimes_\grg F(Y)$.

    Using the explicit characterization for morphisms $(X,Y)\to
    (X',Y') \in \Mor (\stx\cprod{\grc}\!\grg)$ found in \cite[\S
    6.7]{MR92m:18019}, it can be directly verified that $S$ is
    indeed a functor.  The same calculation shows $S$ is
    compatible with the equivalences in $\stx\cprod{\grc}\!\grh$
    \begin{equation*}
      (X,\pi (C)\otimes_\grh Y) \lisoto (X\cdot C, Y)
    \end{equation*}
    resulting from the action of $\grc$.  Thus $S$ is well
    defined.

    It is also easy to verify diagrams of $(\grc,G_0)$-torsors
    as~\eqref{eq:58} and~\eqref{eq:59} (with the additional
    simplification that most 2-morphisms are trivial since there
    are no non trivial morphisms in $G_0$) induce corresponding
    diagrams of $(\grh, \grg)$-torsors. For example, from a
    morphism
    \begin{equation*}
      \xymatrix{%
        \stx \ar[r] \ar@/_/[dr]_s & \sty \ar[d]^t \\ & G_0}
    \end{equation*}
    to which we append $\pi_\grg \colon G_0\to \grg$, we get
    \begin{equation*}
      \xymatrix{%
        \stx \cprod{\grc}\!\grh \ar[r]
        \ar@/_0.7pc/[dr]|{(\pi_\grg\circ s)\wedge F}
        & \sty \cprod{\grc}\!\grh \ar[d]^{(\pi_\grg\circ t)\wedge F} \\ & \grg}
    \end{equation*}
    The latter is 2-commutative, due to the definition of
    $\cprod{\grc}$ in the context of stacks.

    From the exact sequences of homotopy sheaves above it follows
    that
    \begin{equation*}
      \pi_0 (\tors(\grh,\grg)) \iso \pi_0 (\tstf),
    \end{equation*}
    so again following \cite[Chap. 8]{MR95m:18006} we have that
    both $\tors(\grh,\grg)$ and $\tstf$ can be considered as
    2-gerbes over the same base $\s\downarrow\pi_0$.

    Moreover, we have already observed the automorphism gr-stack
    (as automorphism gr-stack of the trivial torsor) of $\tstf$
    is $\grk$.  It is immediately verified the same is true for
    $\tors (\grh, \grg)$. Hence, as 2-gerbes over $\pi_0$, they
    are ``banded'' by the same gr-stack and there is a 2-functor
    $G \colon \tstf \to \tors (\grh, \grg)$. Hence they are
    necessarily 2-equivalent over $\s \downarrow {\pi_0}$. But
    this implies they are equivalent over $\s$ tout-court.
  \end{proof}
  This ends the proof of the proposition and hence of the
  theorem.
\end{proof}
\begin{remark}
  In the course of the proof a longer sequence was obtained,
  namely
  \begin{equation*}
    \xymatrix@1{%
      0\ar[r] & \pi_1 (\grk) \ar[r] & \pi_1 (\grh) \ar[r] &
      \pi_1(\grg) 
       \ar `r[d] `[ll] `[lld] [dll]\\
       & \grk \ar[r]  & \grh \ar[r]^{F} & \grg 
       \ar `r[d] `[ll] `[lld]_\iota [dll] \\
      & \tstf \ar[r]^{\Delta} &
      \tors ({\grh}) \ar[r]^{F_*} & \tors ({\grg})}
  \end{equation*}
  Note that the first three (non-zero) terms are (abelian)
  groups, the next three are 2-groups, whereas the last three are
  2-stacks but lack a group structure, weak or otherwise.

  This sequence ought to be considered (the geometric version of)
  the counterpart of the sequence \cite[(3.9.1)]{MR92m:18019} for
  2-groups. A more detailed analysis will appear in
  \cite{ButterfliesIV}.

  Combining with results and remarks from \cite{MR93k:18019}, one
  can argue for an extension one step to the right if $\tors
  (\grh, \grh)$, or equivalently, $\tstf$, is a \emph{3-}group,
  that is, if $F\colon \grh\to \grg$ is, in the appropriate sense
  discussed in \loccit, a crossed module of gr-stacks.
\end{remark}
\begin{remark}
  Sequence~\eqref{eq:44} is the exact counterpart of the homotopy
  fiber sequence of ref.\ \cite[Theorem 9.1]{Noohi:weakmaps},
  where one regards crossed modules as 2-categories with one
  object, namely considers their suspension.  When applied to a
  gr-stack $\grh$, this process gives rise to its the naïve
  suspension $\grh[1]$, namely the fibered bicategory with one
  object such that the composition of 1-morphisms is given by the
  group-like structure of $\grh$ (cf.  \cite{MR95m:18006}).  Note
  that in~\eqref{eq:44} we have the correct geometric suspensions
  $\tors (\grh)$ and $\tors (\grg)$ of $\grh$ and $\grg$,
  respectively, instead.  As mentioned in \cite{MR95m:18006},
  $\tors (\grh)$ is obtained by taking the associated 2-stack (in
  fact 2-gerbe) of the naïve suspension $\grh [1]$.

  These considerations and Theorem~\ref{thm:6} make it suggestive
  to consider some portion of the exact sequence~\eqref{eq:44} as
  a candidate to play the rôle of an ``exact triangle''
  \begin{equation*}
    \grh \overset{F}{\lto} \grg \lto \twostack{F} \lto
    \grh [1]
  \end{equation*}
  for the non-abelian derived category.  This is only suggestive
  in that the ``cone,'' that is $\tstf$, is a 2-stack---an object
  related to a complex of length three.  Moreover, as we have
  already observed it is not a group object.
\end{remark}
We conclude with the following observation. First, using
Proposition~\ref{prop:14}, the obvious morphism
\begin{math}
  \tors (\grk) \lto \tors (\grc)
\end{math}
determined by $\grk \to \grc$ factors through a morphism
\begin{equation*}
  \tors (\grk) \lto \tors (\grc, G_0).
\end{equation*}
Now, if $F\colon \grh \to \grg$ is essentially surjective, then
the induced map $\pi_0(\grh) \to \pi_0(\grg)$ is an
epimorphism. To put it differently, $\pi_0 (\tstf)=*$, and
$\tstf$ becomes locally connected, hence a 2-gerbe
directly over the site $\s$.  Therefore we obtain an equivalence
of 2-gerbes
\begin{equation*}
  \tors (\grk ) \lisoto \tstf.
\end{equation*}
In other words, when $\grh \to \grg$ is essentially
surjective---so that taken together with its homotopy kernel it
gives rise to a short exact sequence in the sense of
section~\ref{sec:exact-sequences}---the homotopy fiber can be
identified with $\tors (\grk)$, i.e.\ the suspension of the
kernel.

\subsection{Exact sequence in non-abelian cohomology}
\label{sec:exact-sequence-non}
We now consider again a short exact sequence
\begin{equation*}
  \grk \overset{\iota}{\lto} \grh \overset{p}{\lto} \grg
\end{equation*}
of gr-stacks and show there is a corresponding long exact
sequence in non-abelian cohomology.  The definition of
short-exact was given in section~\ref{sec:exact-sequences}.
Note, the fibration condition was not included, and it is our
purpose here to point out how our characterization of weak
morphisms allows us to dispense of the fibration condition.

For the definition of non-abelian cohomology, we use the one
given in \cite[\S 4]{MR92m:18019}, supplemented by the explicit
cocyclic formulas recalled in sections~\ref{sec:cocycles}.

We will make the assumption (which, thanks to
Proposition~\ref{prop:9} is not a restriction) that $\grk$,
$\grh$, and $\grg$ are associated to crossed modules $K_\bullet$,
$H_\bullet$, and $G_\bullet$, respectively.  By implicitly making
use of Proposition~\ref{prop:9} we indifferently write
$\H^i(G_\bullet)$ or $\H^i(\grg)$, for $i=-1,0,1$.

We will not dwell on the interpretation of $\H^1(\grg)$ except to
note, after \loccit, that it should be interpreted as $\pi_0
(\tors (\grg)(*))$ (classes of equivalences of global objects
over $\s$).  This part will be taken up in detail in the
forthcoming \cite{ButterfliesII}.

As for the other degrees, it follows from the considerations in
section~\ref{sec:cocycles} that $\H^0(\grg) = \pi_0 (\grg (*))$
(isomorphism classes of global objects of $\grg$), whereas
$\H^{-1}(\grg) \iso \H^0(\pi_1(\grg))$ (ordinary abelian sheaf
cohomology).  The latter identification follows from the
definition and \cite{MR991977}.
\begin{proposition}
  \label{prop:12}
  The short exact sequence
  \begin{equation*}
    \grk \overset{\iota}{\lto} \grh \overset{p}{\lto} \grg
  \end{equation*}
  induces an exact sequence in cohomology
  \begin{equation}
    \label{eq:61}
    \vcenter{%
      \xymatrix{%
        0 \ar[r] & \H^{-1} (K_\bullet) \ar[r] &
        \H^{-1} (H_\bullet) \ar[r] & \H^{-1} (G_\bullet) \ar[r] &
        \H^{0} (K_\bullet) \ar `r[d] `[lll] `[dlll] [dlll] \\
        &  \H^{0}(H_\bullet) \ar[r] & \H^{0} (G_\bullet) \ar[r] &
        \H^{1} (K_\bullet) \ar[r] & \H^{1} (H_\bullet) \ar[r] & \H^{1} (G_\bullet)
      }}
  \end{equation}
\end{proposition}
\begin{proof}
  Let $[H_\bullet, E, G_\bullet]$ be a butterfly corresponding to
  $p$, e.g.\ by the fiber product construction. Let
  \begin{equation*}
    H_\bullet \overset{Q}{\longleftarrow} E_\bullet
    \overset{P}{\lto} G_\bullet
  \end{equation*}
  be the fraction determined by it, and let $\grstack{E}$ be the
  gr-stack associated to $E_\bullet$. Since $\grh\to \grg$ is
  essentially surjective, then $\jmath$ is an
  epimorphism. Moreover, the projection $H_1\times G_1\to G_1$ is
  also evidently so. At the level of simplicial groups we have an
  epimorphism $\smp{E}_\bullet \to \smp{G}_\bullet$.  Moreover,
  it follows that the morphism of gr-stacks
  \begin{equation*}
    P\colon \grstack{E} \lto \grg
  \end{equation*}
  arising from the strict morphism $P'$ is actually not only
  essentially surjective, but also a fibration.

  With the observation that the homotopy kernel of $P$ is still
  $\grk$ (since that of $P'$ is the crossed module $[H_1\to \ker
  \jmath]$), we can apply the results of \cite[\S
  5.1]{MR92m:18019}, in particular the sequence (5.1.3). It
  follows there is a long exact cohomology sequence
  \begin{equation*}
    \xymatrix{%
      0 \ar[r] & \H^{-1} (K_\bullet) \ar[r] & \H^{-1} (E_\bullet) \ar[r] &
      \H^{-1} (G_\bullet) \ar[r] &
      \H^{0} (K_\bullet) \ar `r[d] `[lll] `[dlll] [dlll] \\
      & \H^{0} (E_\bullet) \ar[r] & \H^{0} (G_\bullet) \ar[r] & \H^{1} (K_\bullet)
      \ar[r] & \H^{1} (H_\bullet) \ar[r] & \H^{1} (G_\bullet)
    }
  \end{equation*}
  Now, $Q'$ is a quasi-isomorphism, or equivalently the
  corresponding morphism
  \begin{equation*}
    Q\colon \grstack{E} \lto \grh
  \end{equation*}
  is an equivalence. It follows that $\H^i (\grstack{E}) \iso
  \H^i (\grh)$ (one can use the ``intelligent filtration''
  \begin{equation*}
    \pi_1(\grh) [1] \lto \grh \lto \pi_0 (\grh) [0],
  \end{equation*}
  see \cite[\S 5.3]{MR92m:18019}, for this purpose) which proves
  the proposition.
\end{proof}
The point of the proposition, or rather of its proof, is that we
can dispense of the fibration condition in the definition of
exact sequence, since by the very construction of a weak morphism
we can always replace the essentially surjective morphism
$p\colon \grh \to \grg$ with a fibration.  The butterfly diagram
construction of the weak morphism offers a canonical way to
accomplish it.

\section{Braided and abelian butterflies}
\label{sec:abelian-butterflies}

In this section we specialize our discussion to butterflies in
the abelian category of abelian sheaves over $\s$.  We obtain an
explicit characterization of the derived category of complexes of
length two of abelian sheaves over $\s$, see \cite{0259.14006}.
We are going to do so by introducing various customary
commutativity conditions (braided, symmetric, Picard) on
gr-stacks over $\s$, before devoting ourselves to the fully
abelian situation. While these conditions are all well known, our
approach, we believe, is new, even in the set-theoretic case.

\subsection{Butterflies and braidings}
\label{sec:butt-braid}

The first, indeed, weakest, possible commutativity condition that
can be imposed on a gr-stack (or gr-category, for that matter) is
the simple existence of a \emph{braiding} isomorphism, i.e.\ a
natural isomorphism implementing a formal commutativity
condition.
\begin{paragr}
  A braiding takes the form of a collection of functorial
  isomorphisms
  \begin{equation}
    \label{eq:62}
    s_{X,Y} \colon  X\otimes_\grg Y \lisoto Y\otimes_\grg X
  \end{equation}
  satisfying the well-known two hexagon diagrams of ref.\
  \cite{MR1250465}.

  The braiding is called \emph{symmetric} or (symmetric
  monoidal) if in addition the condition
  \begin{equation}
    \label{eq:63}
    s_{Y,X}\circ s_{X,Y} = \Id_{X\otimes_\grg Y}
  \end{equation}
  is satisfied for all objects $X$ and $Y$ of $\grg$.
  Furthermore, we say the braiding is \emph{Picard,} or that
  $\grg$ is a \emph{Picard stack} (or gr-category in the
  pointwise case) if in addition to the symmetry condition it
  satisfies the condition
  \begin{equation}
    \label{eq:64}
    s_{X,X} = \Id_{X\otimes_\grg X}
  \end{equation}
  for all objects $X$ of $\grg$. For convenience, we use the
  terminology ``Braided,'' ``Symmetric'' and ``Picard,'' where
  others (notably, Breen, see, e.g.\
  \cite{MR95m:18006,MR1702420}) use ``Braided,'' ``Picard'' and
  ``Strictly Picard.''
\end{paragr}
\begin{paragr}
  Let $G_\bullet$ be a crossed module. The previous conditions
  make sense for the (strict) group law on the groupoid
  determined by $G_\bullet$, and~\eqref{eq:62}, in particular,
  gives rise to the braiding map $\braid{-}{-} \colon G_0\times
  G_0 \to G_1$ such that
  \begin{equation}
    \label{eq:65}
    \del\braid{x}{y} = y^{-1}x^{-1}yx
  \end{equation}
  for all $x,y\in G_0$. (This follows immediately from the
  request that $s_{x,y}$ be an isomorphism from $xy$ to $yx$.) As
  pointed out in \cite{MR95m:18006}, if all other commutativity
  conditions are imposed $s$ becomes a full lift of the
  commutator map from $G_0$ to $G_1$.
\end{paragr}
It is well known that for a group to be abelian is equivalent to
the condition that the multiplication map be a group
homomorphism.  A similar approach can also be adopted in the
context of gr-stacks.  Indeed it is a relatively simple exercise
to show that the existence of a braiding on $\grg$ is equivalent
to the fact that the tensor operation $\otimes_\grg\colon
\grg\times \grg\to \grg$ is a morphism of gr-stacks, that is, an
additive functor. (This uses the fact every object in a
gr-category or -gr-stack is regular, see \cite{sinh:gr-cats}.)
\begin{paragr}
  If $\otimes_\grg\colon \grg\times \grg\to \grg$ is a morphism
  of gr-stacks, by the equivalence in Theorem~\ref{thm:2} there
  must be a butterfly diagram
  \begin{equation}
    \label{eq:66}
    \vcenter{%
      \xymatrix@R-0.5em{%
        G_1\times G_1 \ar[dd]_\del \ar@/_0.1pc/[dr]^\alpha  & &
        G_1 \ar@/^0.1pc/[dl]_\beta \ar[dd]^\del\\
        & P\ar@/_0.1pc/[dl]_\rho \ar@/^0.1pc/[dr]^\sigma &  \\
        G_0\times G_0 & & G_0
      }}
  \end{equation}
  from $G_\bullet\times G_\bullet$ to $G_\bullet$.
\end{paragr}
\begin{paragr}
  The butterfly~\eqref{eq:66} has some interesting additional
  properties.  Let
  \begin{equation*}
    \iota_1,\iota_2 \colon \grg\lto \grg\times \grg
  \end{equation*}
  be the two injections sending $X$ to $(X,I_\grg)$ (respectively
  to $(I_\grg,X)$).  In any gr-stack or gr-category the existence
  of the functorial isomorphisms~\eqref{eq:67} can be recast as
  the requirement that the composite of the multiplication law
  with $\iota_1$ or $\iota_2$ be isomorphic to the identity
  functor of $\grg$ to itself.  Translated in the language of
  butterflies, this means that the butterfly obtained by
  pre-composing~\eqref{eq:66} with the (strict) morphisms of
  crossed modules $\iota_1,\iota_2\colon G_\bullet\to
  G_\bullet\times G_\bullet$ (defined in the same way as for
  $\grg$) must be isomorphic to the identity morphism.  Since
  pre-composition with a strict morphism means pulling back, we
  arrive at the conclusion that the extension on the NE-SW
  diagonal of~\eqref{eq:66} must split when restricted to either
  factor in $G_0\times G_0$, that is, when pulled back by either
  $\iota_1$ or $\iota_2$.\footnote{We commit the abuse of
    language of still denoting the components of the strict
    morphism $\iota_i\colon G_\bullet \to G_\bullet\times
    G_\bullet$ by the same letter} In other words, since the
  extension
  \begin{equation*}
    G_1\lto \iota_i^*P\lto G_0
  \end{equation*}
  $i=1,2$, splits, there must exist two group homomorphisms
  \begin{equation}
    \label{eq:68}
    s_1,s_2 \colon G_0 \lto P
  \end{equation}
  such that $\rho\circ s_i = \iota_i$, $i=1,2$, as maps $G_0\to
  G_0\times G_0$.

  Since the composed split butterfly corresponds to the identity
  morphism, we must have
  \begin{gather}
    \label{eq:79}
    \sigma\circ s_i=\id_{G_0}\\
    \intertext{and}
    \label{eq:80}
    s_i(\del g) = \alpha (\iota_i (g)) \beta (g),\quad g\in G_1,
  \end{gather}
  for $i=1,2$, where we have used the explicit relation between
  strict morphisms and split butterflies analyzed in
  sect~.\ref{sec:strict-morph-butt}.
\end{paragr}
\begin{paragr}
  The existence of the two homomorphisms $s_1$ and $s_2$
  implies~\eqref{eq:66} is a \emph{strong} butterfly in the sense
  of Definition~\ref{def:8}. Indeed, let
  \begin{equation*}
    \tau \colon G_0\times G_0 \lto P
  \end{equation*}
  be defined by $\tau (x,y) = s_1 (x) s_2 (y)$, for $x,y\in G_0$.
  We have
  \begin{equation*}
    \rho\circ \tau (x,y) = \rho (s_1(x)s_2(y)) = (x,1)(1,y) = (x,y),
  \end{equation*}
  therefore $\tau$ provides a global set-theoretic splitting, as
  required.
\end{paragr}
\begin{paragr}
  Analyzing how far $\tau$ is from being a group homomorphism,
  leads to consider the combination
  \begin{equation}
    \label{eq:69}
    s_2(y)^{-1}s_1(x)^{-1}s_2(y)s_1(x).
  \end{equation}
  It is immediate that
  \begin{equation*}
    \rho ( s_2(y)^{-1}s_1(x)^{-1}s_2(y)s_1(x))
    = (1,y^{-1})(x^{-1},1)(1,y)(x,1) = 1
  \end{equation*}
  so that there exists $c(x,y)\in G_1$ such that
  \begin{equation}
    \label{eq:70}
    \beta (c(x,y)) = s_2(y)^{-1}s_1(x)^{-1}s_2(y)s_1(x).
  \end{equation}
  Moreover, by applying $\sigma$:
  \begin{equation*}
    \del c(x,y) = \sigma\beta c(x,y) = \sigma
    (s_2(y)^{-1}s_1(x)^{-1}s_2(y)s_1(x))
    = y^{-1}x^{-1}yx,
  \end{equation*}
  which should be compared with~\eqref{eq:65}.
  Thus~\eqref{eq:70} defines a braiding in the standard way.
  Note that with the previous choices the failure for $\tau$ to
  be a homomorphism is measured as:
  \begin{equation*}
    \tau (x_0x_1,y_0y_1)  \beta( c(x_1,y_0)^{y_1}) = \tau (x_0,y_0) \tau
    (x_1,y_1).
  \end{equation*}
\end{paragr}
\begin{paragr}
  For the interpretation of~\eqref{eq:70} as a braiding to be
  complete, two more checks are necessary. First, the two hexagon
  diagrams of \cite{MR1250465} must be satisfied. It is well
  known that in the case of a strict 2-group, i.e.\ crossed
  module, they reduce to the cocycle conditions
  \begin{align*}
    \braid{x}{yz} & = \braid{x}{y}^z\braid{x}{z}\\
    \intertext{and} \braid{xy}{z} & = \braid{y}{z}\braid{x}{z}^y
  \end{align*}
  for the braiding map. With~\eqref{eq:70}, the above cocycle
  conditions become an immediate consequence of the fact that
  $s_1$ and $s_2$ are homomorphisms. This simple fact is left as
  an exercise to the reader.

  Second, the functoriality condition for the
  isomorphisms~\eqref{eq:62} is expressed in terms of the
  braiding by two relations
  \begin{align*}
    \braid{x}{\del h} = h^{-1} h^x \\
    \intertext{and} \braid{\del g}{y} = g^{-y} g,
  \end{align*}
  where $x,y\in G_0$ and $g,h\in G_1$.  It is immediately
  verified that both conditions are satisfied by~\eqref{eq:70} as
  a consequence of~\eqref{eq:80}.
\end{paragr}
We can summarize the situation so far in the following
proposition.
\begin{proposition}
  \label{prop:15}
  Let $\grg\iso [G_1 \xrightarrow{\del} G_0]$.  Then the
  following are equivalent.
  \begin{itemize}
  \item $\grg$ is braided.
  \item $G_\bullet$ is braided.
  \item There is a butterfly~\eqref{eq:66} equipped with
    prescribed splittings~\eqref{eq:68} such that $\rho\circ s_i
    = \iota_i$ and both~\eqref{eq:79} and~\eqref{eq:80} hold.
  \end{itemize}
\end{proposition}
\begin{remark}
  Changing either or both splittings~\eqref{eq:68} has the effect
  of replacing the ``braiding''~\eqref{eq:70} with an equivalent
  one.  However, it is more appropriate to consider $s_1$ and
  $s_2$ as part of the structure.
\end{remark}

\subsection{Symmetric crossed modules and 2-groups}
\label{sec:symm-picard-butt}

The exchange of $s_1$ and $s_2$ in~\eqref{eq:70} replaces
$c(x,y)$ with $c'(x,y)=c(y,x)^{-1}$.  According to
\cite{MR1250465}, this is still a braiding, and the underlying
tensor category is called \emph{symmetric} if it so happens that
$c'=c$.  We want then to interpret the symmetry and Picard
conditions in terms of the characterization provided by
Proposition~\ref{prop:15}.

Let $T\colon \grg\times \grg\to \grg\times \grg$ be the swap
functor which exchanges the factors: $(X,Y)\mapsto (Y,X)$.  The
same letter will denote the corresponding map for $G_\bullet$, as
well as $G_1$ and $G_0$ separately.

Since $T$ can of course be considered as a strict morphism of
crossed modules $G_\bullet\times G_\bullet \to G_\bullet\times
G_\bullet$, it can be composed with the butterfly~\eqref{eq:66}
to yield a the new one: $[G_\bullet\times G_\bullet, T^*P,
G_\bullet]$. Recall that since we are composing with a strict
morphism, the resulting object in the center is just a pullback.
The butterfly so obtained corresponds to the additive functor
$\grg\times \grg\to \grg$ given by the opposite multiplication
law: $(X,Y)\mapsto Y\otimes_\grg X$.

We consider the natural symmetry condition with respect to $T$,
as spelled in the following definition.
\begin{definition}
  \label{def:14}
  The butterfly diagram~\eqref{eq:66} is \emph{symmetric} if it
  is isomorphic to its own pullback under $T$.  In other words,
  if there exists a group isomorphism
  \begin{equation*}
    \psi \colon P\lto T^*P
  \end{equation*}
  realizing a morphism of butterflies from $G_\bullet\times
  G_\bullet$ to $G_\bullet$.
\end{definition}
It is easy to see that this definition is equivalent to the
standard notion of braided symmetric 2-group.  Indeed the
condition on the butterfly spelled in the previous definition is
the translation in terms of butterfly diagrams of the following
way to recast the symmetry condition for a braiding.  We state it
as a Proposition.
\begin{proposition}
  \label{prop:17}
  The braiding $s$ on $\grg$ is symmetric if and only if it is a
  morphism of additive functors
  \begin{equation*}
    s \colon \otimes_\grg \Longrightarrow \otimes_\grg\circ T
    \colon \grg\times \grg \lto \grg.
  \end{equation*}
\end{proposition}
As a corollary of Proposition~\ref{prop:17} we have:
\begin{proposition}
  \label{prop:16}
  The braiding on $\grg$ is symmetric if and only if the
  butterfly~\eqref{eq:66} has the symmetry property of
  Definition~\ref{def:14}.
\end{proposition}
\begin{proof}[Proof of Proposition~\ref{prop:17}]
  That the braiding $s$ can be seen as a natural transformation
  is an obvious fact. The point is of course that it is a natural
  transformation of \emph{additive} functors.  Writing
  diagram~\eqref{eq:76} for $\otimes_\grg$, $\otimes_\grg\circ
  T$, and $s$, and using the two hexagon diagrams shows the
  equivalence between~\eqref{eq:63} and the symmetry condition.
  This is most easy when working directly with a crossed module
  and the braiding $\braid{-}{-}$, and it is left as a task to
  the reader.
\end{proof}
It is instructive to deduce the symmetry of the braiding at the
crossed-module level directly from Definition~\ref{def:14}.
\begin{paragr}
  The existence of $\psi$ in Definition~\ref{def:14} is
  equivalent to saying that there should be an automorphism
  $\psi\colon P\to P$ such that
  \begin{equation}
    \label{eq:71}
    \vcenter{\xymatrix{%
        P \ar[r]^\psi{} \ar[d]_\rho & P \ar[d]^\rho \\
        G_0\times G_0 \ar[r]_T & G_0\times G_0
      }}
  \end{equation}
  commutes and it is compatible with all the morphisms
  of~\eqref{eq:66}. In particular, this implies that $\psi$ fixes
  $G_1$ inside $P$.

  Using Grothendieck's theory of extensions, the automorphism
  $\psi$ in the previous diagram can be equivalently understood
  as a collection of isomorphisms
  \begin{equation*}
    \psi_{x,y} \colon P_{x,y} \lto P_{y,x}
  \end{equation*}
  of $G_1$-bitorsors $E_{x,y}$ above each point $(x,y)\in
  G_0\times G_0$, compatible with the multiplication structure of
  bitorsors
  \begin{equation}
    \label{eq:72}
    \vcenter{%
      \xymatrix{%
        P_{x_0,y_0} \cprod{G_1} P_{x_1,y_1} \ar[r]
        \ar[d]_{\psi\wedge \psi} &
        P_{x_0x_1,y_0y_1} \ar[d]^\psi \\
        P_{y_0,x_0} \cprod{G_1} P_{y_1,x_1} \ar[r] & P_{y_0y_1,x_0x_1}
      }}
  \end{equation}
  The horizontal arrows are the bitorsor contractions
  corresponding to the multiplication law in $P$.
\end{paragr}
\begin{paragr}
  The butterfly $T^*P$ must split when restricted to one of the
  factors of $G_0\times G_0$. Thus, there will exist group
  homomorphisms ${\hat s}_i \colon G_0\to \iota_i^*T^*P =
  (T\circ\iota_i)^*P$, $i=1,2$.  Now, since
  $T\circ\iota_1=\iota_2$, we can rather think of ${\hat s}_i$ as
  homomorphisms
  \begin{equation*}
    {\hat s}_i \colon G_0 \lto P
  \end{equation*}
  such that $\rho\circ {\hat s}_1 = \iota_2$ and $\rho\circ {\hat
    s}_2 = \iota_1$. In other words, we must have
  \begin{equation*}
    \rho \circ {\hat s}_1 = \rho\circ s_2,\qquad
    \rho \circ {\hat s}_2 = \rho\circ s_1
  \end{equation*}
  In general, this implies that ${\hat s}_1$ differ from $s_2$ by
  multiplication of an element in $G_1$ provided by an
  appropriate crossed homomorphism, and similarly for ${\hat
    s}_2$ and $s_1$.  However, if we consider that $s_1$ and
  $s_2$ are part of the structure, as they implement the given
  functorial isomorphisms~\eqref{eq:67}, then they must be simply
  swapped by $T$, so that we have
  \begin{equation}
    \label{eq:73}
    {\hat s}_1 = s_2, \qquad {\hat s}_2= s_1. 
  \end{equation}
\end{paragr}
\begin{paragr}
  If the butterfly is symmetric, so that the automorphism
  $\psi\colon P\to P$ as in Definition~\ref{def:14} exists, we
  must have that the following diagram (which
  completes~\eqref{eq:71})
  \begin{equation}
    \label{eq:74}
    \vcenter{\xymatrix{%
        P \ar[rr]^\psi{} \ar[dd]_\rho & & P \ar[dd]^\rho \\
        & G_0 \ar@/_/[ul]^{s_{1,2}} \ar@/^/[ur]_{{\hat s}_{1,2}}
        \ar@/^/[dl]_{\iota_{1,2}} \ar@/_/[dr]^{\iota_{2,1}} \\
        G_0\times G_0 \ar[rr]_T & & G_0\times G_0}}
  \end{equation}
  commutes.  Taking~\eqref{eq:73} into account, we obtain that
  $\psi$ must swap $s_1$ and $s_2$:
  \begin{equation}
    \label{eq:75}
    \psi\circ s_1 = s_2, \qquad \psi\circ s_2 = s_1.
  \end{equation}
  Thus, in view of the remark at the beginning of this section,
  $\psi$ replaces the braiding $c(x,y)$ with the symmetric one
  $c(y,x)^{-1}$.

  On the other hand, we have observed that compatibility
  of~\eqref{eq:71} with the rest of the butterfly means that
  $\psi$ must fix $G_1$, so that
  \begin{equation*}
    c(x,y) = \psi (c(x,y))  = c(y,x)^{-1}
  \end{equation*}
  that is, the braiding is symmetric in the usual sense.
\end{paragr}

\subsection{Picard crossed modules and 2-groups}
\label{sec:picard-cross-modul}

Let $\Delta\colon \grg \to \grg\times \grg$ the diagonal
functor. It is obvious that $\Delta$ is a strict additive
functor. Let $\Delta$ also denote the corresponding diagonal
functor for $G_\bullet$, as well as the degree-wise diagonal
homomorphisms for $G_0$ and $G_1$.
\begin{paragr}
  We can pull back the butterfly~\eqref{eq:66} to $G_\bullet$ via
  $\Delta$. Since $\Delta$ is a strict morphism, the composed
  butterfly corresponds to the additive functor $\grg \to \grg$,
  $X\to X\otimes X$.  If the butterfly is symmetric as per
  Definition~\ref{def:14}, we obtain an automorphism
  \begin{equation}
    \label{eq:77}
    \Delta^*\psi \colon \Delta^*P \lto \Delta^*P.
  \end{equation}
\end{paragr}
\begin{definition}
  \label{def:15}
  A symmetric crossed module is \emph{Picard} if the
  automorphism in eq.~\eqref{eq:77} is the identity.
\end{definition}
\begin{paragr}
  This is equivalent to the standard notion of ``Picard'' since,
  using Proposition~\ref{prop:17} and pre-composing with
  $\Delta\colon \grg\to \grg\times \grg$, it becomes the
  statement that $s*\Delta$, as a transformation from $X\to
  X\otimes X$ to itself, is the identity---which is
  condition~\eqref{eq:64}.
\end{paragr}
\begin{paragr}
  The Picard condition simply means that the automorphisms
  \begin{equation*}
    \psi_{x,x} \colon P_{x,x}\lto P_{x,x},
  \end{equation*}
  $x\in G_0$, are equal to the identity.  Then,
  using~\eqref{eq:72} with $(x_0,y_0)=(x,1)$ and
  $(x_1,y_1)=(1,x)$ leads to the condition
  \begin{equation}
    \label{eq:78}
    s_1(x)s_2(x) = s_2(x)s_1(x),
  \end{equation}
  which implies that $c(x,x)=1$. It is also easy to see that if
  we assume condition~\eqref{eq:78}, and use the fact that
  $P_{x,1}$ and $P_{1,y}$ are \emph{canonically trivial} as
  $G_1$-bitorsors, we obtain an automorphism $\psi$
  satisfying~\eqref{eq:71} and $\psi_{x,x}=\id$ for all $x\in
  G_0$.
\end{paragr}
To conclude, let us remark that the stronger condition that the
homomorphisms $s_1$ and $s_2$ have commuting images, that is
\begin{equation*}
  [s_1(x), s_2(y)]=1
\end{equation*}
for all $x,y\in G_1$, implies that $\tau\colon G_0\times G_0\to
P$ is a homomorphism, so that the braided butterfly~\eqref{eq:66}
splits and~\eqref{eq:70} shows that the braiding is identically
equal to $1$.  As a consequence, we find that both $G_1$ and
$G_0$ are abelian and the action of $G_0$ on $G_1$ is
trivial. The crossed module reduces to a length-two complex of
abelian sheaves.

From a homological point of view, this reduces to the extension
problem
\begin{equation*}
  0\lto A \lto G_1 \lto G_1 \lto B \lto 0
\end{equation*}
of abelian sheaves.

\subsection{Braided butterflies}
\label{sec:braided-butterflies}

Assume $\grh$ and $\grg$ are braided gr-stacks, and let
$H_\bullet$ and $G_\bullet$ be corresponding braided crossed
modules (cf.\ Proposition~\ref{prop:15}).  The following
definition is, \emph{mutatis mutandis,} the same as the one in
\cite[Definition~12.1]{Noohi:weakmaps}:
\begin{definition}
  \label{def:16}
  A butterfly $[H_\bullet, E, G_\bullet]$ is \emph{braided} if
  the following condition is satisfied:
  \begin{equation*}
    \kappa \braid{\pi (x)}{\pi (y)}_H \, \imath
    \braid {\jmath (x)}{\jmath (y)}_G = y^{-1} x^{-1} y x,
  \end{equation*}
  for all $x,y\in E$.
\end{definition}
If the butterfly comes from a strict morphism, then being braided
corresponds to the usual notion of morphism of braided
categorical groups, that is:
\begin{equation*}
  \braid{f_0(x)}{f_0(y)}_G = f_1( \braid{x}{y}_H)
\end{equation*}
for all $x,y\in H_0$ (see e.g.\ \cite{MR1250465}).
\begin{paragr}
  One way to understand Definition~\ref{def:16} is to notice that
  \begin{equation}
    \label{eq:81}
    \braid{x}{y}_E \eqdef \bigl(
    \braid{\pi (x)}{\pi (y)}_H , 
    \braid {\jmath (x)}{\jmath (y)}_G \bigr)
  \end{equation}
  defines a braiding on the crossed module $H_1\times G_1\to E$
  compatible with the strict morphisms to $H_\bullet$ and
  $G_\bullet$ in the sense explained above.  The condition in the
  definition is just the statement that $\kappa\cdot\imath
  (\braid{x}{y}_E) = y^{-1}x^{-1}yx$.

  It is easy to verify that the formula~\eqref{eq:81} for the
  braiding on $[H_1\times G_1\to E]$ is actually dictated by the
  above requirements.
\end{paragr}
\begin{paragr}
  The weak morphism counterpart of a braided butterfly is that of
  a weak morphism (i.e.\ additive functor between associated
  stacks) that is compatible with the braidings---or braided, for
  short. So, $F\colon \grh\to \grg$ is braided if all objects
  $X,Y$ of $\grh$ we have
  \begin{equation*}
    \xymatrix{%
      F(X)\otimes_\grg F(Y) \ar[r]^{\lambda_{X,Y}}
      \ar[d]_{s_{F(X),F(Y)}} &
      F(X\otimes_\grh Y) \ar[d]^{F(s_{X,Y})} \\
      F(Y)\otimes_\grg F(X) \ar[r]_{\lambda_{Y,X}} &
      F(Y\otimes_\grh X)}
  \end{equation*}
\end{paragr}
\begin{proposition}
  \label{prop:18}
  The butterfly $[H_\bullet, E, G_\bullet]$ is braided, that is,
  it satisfies the condition of Definition~\ref{def:16} if and
  only if the corresponding weak morphism is a morphism of
  braided gr-stacks.
\end{proposition}
\begin{proof}
  Let
  \begin{equation*}
    \xymatrix@1{
      H_\bullet & E_\bullet \ar[l]_{Q'} \ar[r]^{P'} & G_\bullet
    }
  \end{equation*}
  be the factorization of the butterfly in terms of strict
  morphisms (with $Q'$ a quasi-isomorphism). Let
  \begin{equation}
    \label{eq:82}
    \xymatrix@1{
      \grh & \grstack{E} \ar[l]_Q \ar[r]^P & \grg
    }
  \end{equation}
  be the corresponding additive morphisms of gr-stacks, with $F$
  factorized as $P\circ Q^{-1}$. 
  
  If the butterfly is braided, then previous considerations show
  that $P'$ and $Q'$ are strictly braided morphisms and then $P$
  and $Q$ are braided morphisms of gr-stacks. Thus $F$ is
  braided.

  Conversely, let us assume $F$ is braided. First, in the
  decomposition~\eqref{eq:82}, $\grstack{E}$ is braided and $Q$
  is also braided. This actually follows from the diagram:
  \begin{equation*}
      \xymatrix{%
        \grh \times \grh \ar[d]_\otimes &
        \grstack{E}\times\grstack{E} \ar[d]^\otimes
        \ar[l]_<<<{Q\times Q} \\
        \grh & \grstack{E} \ar[l]^Q 
      }
  \end{equation*}
  (Note that the diagram will only be 2-commutative.)  Choosing a
  quasi-inverse $Q^*$ for $Q$ we obtain that
  $\otimes_{\grstack{E}}$ is an additive functor, or equivalently
  that $\grstack{E}$ is braided.  That $Q$ itself is braided
  follows from the next lemma, whose proof is left to the reader.
  \begin{lemma}
    \label{lem:6}
    Let $\cat{C}$ and $\cat{D}$ be braided gr-categories. An
    additive functor
    \begin{equation*}
      (F,\lambda)\colon \cat{D} \to \cat{C}
    \end{equation*}
    is braided if and only if the following diagram
    \begin{equation*}
      \xymatrix{%
        \cat{D}\times \cat{D} \ar[r]^{F\times F} \ar[d]_\otimes &
        \cat{C}\times \cat{C} \ar[d]^\otimes_{}="a" \\
        \cat{D}\ar[r]_F^{}="b" \ar@{=>}@/_/ "a";"b"_\lambda&
        \cat{C}
      }
    \end{equation*}
    (which defines $(F,\lambda)$ to be an additive functor) is in
    fact a morphism of additive functors.
  \end{lemma}
  Using the factorization~\eqref{eq:82} again, we conclude that
  $P \iso F\circ Q$ is braided.  It follows from
  Proposition~\ref{prop:15} that $H_\bullet$, $G_\bullet$, and
  $E_\bullet$ are braided.

  Now, both $P$ (resp.\ $Q$) arise from \emph{strict} morphisms
  $P'$ (resp.\ $Q'$), and we want to conclude that $P'$ and $Q'$
  themselves are braided.  Observe that $P'$ gives rise to a
  morphism of (necessarily braided) groupoids $\grpd{E}\to
  \grpd{G}$ and moreover one has the (2-commutative) diagram of
  prestacks
  \begin{equation*}
    \xymatrix{%
      \grpd{E} \ar[r]^{P'} \ar[d]_a & \grpd{G} \ar[d]^a \\
      \grstack{E} \ar[r]_P & \grg
    }
  \end{equation*}
  where the vertical arrows are the ``associated stack''
  functors.  Since these are equivalences, a moment's thought
  will convince the reader that $P$ braided implies that so is
  $P'$.  The situation with $Q$, $Q'$ is analogous. Thus $P'$ and
  $Q'$ are braided and strict, so the butterfly is braided, as wanted.
\end{proof}
The results of \cite[\S 12]{Noohi:weakmaps} are still valid in
the present context: thus braided butterflies compose according
to the rules of section~\ref{sec:general-definitions}; braided
crossed modules over $\s$ form a bicategory $\twocat{BrXMod}$,
and there is a forgetful functor $\twocat{BrXMod}\to \CM$. 

We note a particular case of braided butterfly. Assume that
$G_\bullet$ is fully abelian in the sense described in the last
couple of paragraphs in section~\ref{sec:symm-picard-butt}. Then
the condition on the braiding reduces to:
\begin{equation*}
  \kappa \braid{\pi (x)}{\pi (y)}_H
  = y^{-1} x^{-1} y x.
\end{equation*}
It is also easy to see that the NE-SW diagonal of the butterfly
is a \emph{central} extension of $H_0$ by $G_1$.  If the
butterfly is an equivalence, so that the other diagonal is also
an extension, it is easy to verify that the braiding
$\braid{-}{-}_H$ is Picard, as expected.  This remark will be of
some relevance in the discussion of butterflies in abelian
categories.  The following statements give a slightly more
general take on the same theme. The proof is very easy, using
Definition~\ref{def:16}, and we will leave it out.
\begin{lemma}
  \label{lem:7}
  Let $[H_\bullet, E, G_\bullet]$ be a braided butterfly.
  \begin{enumerate}
  \item If the corresponding weak morphism $\grh\to \grg$ is
    essentially surjective, or equivalently if $\jmath\colon E\to
    G_0$ is a sheaf epimorphism, then $H_\bullet$ symmetric
    (resp.\ Picard) implies $G_\bullet$ symmetric (resp.\ Picard).
  \item If $\grh \to \grg$ has trivial homotopy kernel, or
    equivalently if $\kappa \colon H_1\to E$ is injective, then
    $G_\bullet$ symmetric (resp.\ Picard) implies $H_\bullet$
    symmetric (resp.\ Picard).
  \end{enumerate}
\end{lemma}
Combining the two statements above into one, we obtain that if
$[H_\bullet, E, G_\bullet]$ is a braided \emph{reversible}
butterfly, then obviously $H_\bullet$ is symmetric (resp.\
Picard) if and only if $G_\bullet$ is symmetric (resp.\ Picard).

\section{Butterflies and abelian categories}
\label{sec:abelian-categories}

In this section we specialize our discussion to the case of
abelian crossed modules, that is, complexes of length two in an
abelian category. This topic has been treated in some generality
in ref.\ \cite[\S 12.3]{Noohi:weakmaps}.  We want to
expand on this topic in the case of $\cat{Ch}(\s)$, the abelian
category of complexes of abelian sheaves over $\s$.

\subsection{Crossed modules in an abelian category}
\label{sec:crossed-modules-ab}

Let $\cat{A}$ be an abelian category.
\begin{paragr}
  Crossed modules in $\cat{A}$ are simply (cohomological)
  complexes of length two of objects of $\cat{A}$ without
  additional requirements. Typically, a complex will be denoted
  as $A^\bullet \colon A^{-1}\to A^0$, restoring the upper-index
  convention---and following ref.\ \cite{0259.14006}.
\end{paragr}
\begin{paragr}
  Complexes of length 2 in $\cat{A}$, supported in degrees
  $[-1,0]$ form an abelian sub-category
  $\cat{Ch}^{[-1,0]}(\cat{A})$ of the abelian category $\cat{Ch}
  (\cat{A})$ of complexes of objects of $\cat{A}$. If $\cat{A}$
  equal the category of abelian sheaves over $\s$ we write
  directly $\cat{Ch}(\s)$ and $\cat{Ch}^{[-1,0]}(\s)$.

  It is immediate that the notions of (strict) morphism and
  2-morphism explained in sect.~\ref{sec:crossed-modules} simply
  reduce to the standard ones of morphism of complexes and
  (chain) homotopy between such morphisms, respectively.
\end{paragr}
\begin{paragr}
  Since we can consider a complex $A^\bullet$ as a crossed module
  with trivial braiding, it follows from our previous analyses
  that a ``braided'' butterfly from $B^\bullet$ to $A^\bullet$ is
  a diagram
  \begin{equation}
    \label{eq:83}
    \vcenter{%
      \xymatrix@R-0.5em{%
        B^{-1}\ar[dd]_d \ar@/_0.1pc/[dr]^\kappa  & &
        A^{-1} \ar@/^0.1pc/[dl]_\imath \ar[dd]^d\\
        & E\ar@/_0.1pc/[dl]_\pi \ar@/^0.1pc/[dr]^\jmath &  \\
        B^0 & & A^0
      }}
  \end{equation}
  of abelian sheaves such that the NE-SW diagonal is an
  extension.  The difference with Definition~\ref{def:7} is that
  all compatibility requirements for the various actions are
  dropped.
\end{paragr}
As such, this definition makes sense in any abelian category
$\cat{A}$, as noted in ref.\ \cite{Noohi:weakmaps}.
\begin{definition}
  \label{def:17}
  Let $\cat{A}$ be an abelian category.
  A butterfly in $\cat{A}$ is a diagram of the form~\eqref{eq:83}
  of objects of $\cat{A}$ such that the NE-SW diagonal is short
  exact, and the NW-SE diagonal is a complex.
\end{definition}

Comparing with section~\ref{sec:deriv-categ-cross}, we see that
for length-two complexes
the butterfly diagram~\eqref{eq:83} provides a canonical choice
for the complex $E^\bullet$ quasi-isomorphic to $B^\bullet$.

\subsection{2-categories of Picard and abelian objects}
\label{sec:2-categories-picard}

From now on, we set $\cat{A}$ equal to the category of abelian
sheaves over $\s$.

We have already mentioned the category $\cat{Ch}^{[-1,0]}(\s)$ of
length-two complexes. Morphisms are simply morphisms of complexes
in the usual sense.

$\cat{Ch}^{[-1,0]}(\s)$ neglects the homotopies. When they are
included, we actually obtain a 2-category, to be denoted
$\twocat{Ch}^{[-1,0]}(\s)_{\mathrm{str}}$.  Objects and
1-morphisms are the same as $\cat{Ch}^{[-1,0]}(\s)$, and
2-morphisms are chain homotopies between strict morphisms. Thus
$\cat{Ch}^{[-1,0]}(\s)$ is tautologically the 1-category obtained
from the 2-category $\twocat{Ch}^{[-1,0]}(\s)_{\mathrm{str}}$ by
simply forgetting the 2-morphisms.

The bicategory $\bicat{Ch}^{[-1,0]}(\s)$ has still the same
objects, plus abelian butterflies as 1-morphisms, and morphisms
of (abelian) butterflies as 2-morphisms. (This is modeled on the
definitions of section~\ref{sec:bicat-cross-modul}.) Let us
denote by $\catHom(B^\bullet, A^\bullet)$ the morphism groupoid
from $B^\bullet$ to $A^\bullet$ in $\bicat{Ch}^{[-1,0]}(\s)$.
In a similar way, let us denote by
$\catHom(B^\bullet,A^\bullet)_{\mathrm{str}}$ the morphism
groupoid of the 2-category
$\twocat{Ch}^{[-1,0]}(\s)_{\mathrm{str}}$.

We mention that the groupoids $\catHom (B^\bullet, A^\bullet)$
acquire an extra structure: they are symmetric gr-categories,
since, thanks to the abelianness of everything involved,
butterflies such as~\eqref{eq:83} can be added, there is an
inverse, and an identity (the zero butterfly corresponding to the
identity morphism). The formulas are identical to the ones worked
out in \cite{Noohi:weakmaps} and will not be repeated here: there
is no change in passing from the set-theoretic context to that of
sheaves over the site $\s$.

$\twocat{Pic}(\s)$ will denote the 2-category of Picard stacks
over $\s$. This is a sub-2-category of $\twocat{Gr\mbox{-}Stacks}
(\s)$.

The gr-stack associated to a Picard crossed module is evidently a
Picard stack. In particular so is the stack associated to a
complex $A^\bullet\colon A^{-1}\to A^0$. (Considering $\tors
(A^{-1}, A^0)$, for instance, immediately gives the Picard
structure.)

\begin{paragr}
  Given two complexes $A^\bullet$ and $B^\bullet$ we define, in
  analogy to what was done in
  section~\ref{sec:general-definitions}, the groupoid
  $\cat{WM}(B^\bullet, A^\bullet)$ of weak morphisms from
  $B^\bullet$ to $A^\bullet$, as
  \begin{equation*}
    \cat{WM}(B^\bullet, A^\bullet) \eqdef
    \catHom_{\twocat{Pic}(\s)} ({B^\bullet}\sptilde, {A^\bullet}\sptilde),
  \end{equation*}
  that is, as the groupoid of additive functors of Picard stacks
  from $[B^{-1}\to B^0]\sptilde \to [A^{-1}\to A^0]\sptilde$.

  Thus, there is a natural homomorphism:
  \begin{equation}
    \label{eq:84}
    \bicat{Ch}^{[-1,0]} (\s)\lto \twocat{Pic} (\s).
  \end{equation}
  This is just the composition of the natural inclusion of
  $\bicat{Ch}^{[-1,0]} (\s) \to \CM (\s)$ with $\CM (\s) \to
  \twocat{Gr\mbox{-}Stacks} (\s)$, factoring through
  $\twocat{Pic} (\s)$.
\end{paragr}
Finally, we can specialize the construction of the bicategory
$\CM (\s)$ in section~\ref{sec:bicat-cross-modul} to Picard
crossed modules in the sense of
section~\ref{sec:symm-picard-butt}.  We obtain in this way a
bicategory denoted $\picCM (\s)$ whose objects are Picard crossed
modules, and morphism groupoids from $H^\bullet$ to $G^\bullet$,
denoted $\cat{PicB}(H^\bullet, G^\bullet)$, are Picard
butterflies and their morphisms.
\begin{paragr}
  We can very well give the same definition of weak morphism of
  Picard crossed modules, by setting
  \begin{equation*}
    \cat{WM}(H^\bullet, G^\bullet) \eqdef
    \catHom_{\twocat{Pic}(\s)} ({H^\bullet}\sptilde, {G^\bullet}\sptilde),
  \end{equation*}
  where $H^\bullet$ and $G^\bullet$ are two objects of $\picCM
  (\s)$. There also is an obvious homomorphism
  \begin{equation}
    \label{eq:85}
    \picCM (\s) \lto
    \twocat{Pic} (\s).
  \end{equation}
\end{paragr}

\subsection{Deligne's results in "La formule de dualité globale"}
\label{sec:delignes-results}

Theorem~\ref{thm:2} remains true in the present context, in the
following alternative form:
\begin{theorem}[Theorem~\ref{thm:2} for Picard stacks]
  \label{thm:7}
  For two objects $B^\bullet$, $A^\bullet$ of
  $\bicat{Ch}^{[-1,0]}(\s)$ there is an equivalence of
  groupoids:
  \begin{equation*}
    \cat{WM}(B^\bullet, A^\bullet)  \lisoto
    \catHom (B^\bullet, A^\bullet).
  \end{equation*}
\end{theorem}
In particular one direction---from weak morphisms to
butterflies---corresponds to \cite[Lemme 1.4.13
(II)]{0259.14006}. The proof is obtained by translating the proof
of Theorem~\ref{thm:2} to abelian butterflies by simply assuming
everything is abelian and neglecting the actions.  The reader
will be able to check that it reduces to Deligne's proof (in one
direction). In the other direction, our result provides a direct
converse to Deligne's 1.4.13~(II), different from \loccit,
Corollaire 1.4.17.

Proposition~\ref{prop:9} also remains true,
upon replacing ``gr-stack'' with ``Picard stack.'' Namely, we have:
\begin{proposition}[Proposition~\ref{prop:9} for Picard stacks]
  \label{prop:21}
  Let $\gra$ be a gr-stack. Then there exists a complex
  $A^\bullet \colon A^{-1}\to A^0$ such that $\gra$ is equivalent
  to the Picard stack $[A^{-1}\to A^0]\sptilde$.
\end{proposition}
This is actually Lemme
1.4.13 (I) of \loccit  Again, the proof can be obtained from the
proof of Proposition~\ref{prop:9} by: (1)  replacing the
sheafification of the free group on a sheaf of sets with the free
abelian group thereof; (2) replacing the coherence
argument from~\cite{MR723395} as used above with the one
from~\cite{MR0338002} found in~\cite{0259.14006}.  In this way
the proof becomes essentially the same as the one found in
Deligne's work.

\subsection{2-category of Picard crossed modules}
\label{sec:2-category-picard}

A slightly different point of view on the relationship between
Proposition~\ref{prop:9} and~\cite[Lemme 1.4.13 (I)]{0259.14006}
is as follows. Combining the former with
Proposition~\ref{prop:15} and the considerations on Picard
butterflies in sect.~\ref{sec:symm-picard-butt}, we can state an
alternative version of Proposition~\ref{prop:21}:
\begin{proposition}[Proposition~\ref{prop:9} for Picard
  stacks---2\textsuperscript{nd} version]
  \label{prop:19}
  Let $\gra$ be a Picard stack. Then there exists a Picard
  crossed module $[G^{-1}\to G^0]$ such that $\gra$ is equivalent
  to the Picard stack $[G^{-1}\to G^0]\sptilde$.
\end{proposition}
The crossed module obtained in this way is not necessarily an
object of $\bicat{Ch}^{[-1,0]}(\s)$.  However, it is
\emph{equivalent} to one. More precisely, the combination of
Propositions~\ref{prop:21} and~\ref{prop:19} guarantees that
given a Picard crossed module $[G^{-1}\to G^0]$ one can find a
complex $A^{-1}\to A^0$ of abelian sheaves such that there is an
equivalence
\begin{equation*}
  [G^{-1}\to G^0]\sptilde \lisoto [A^{-1}\to A^0]\sptilde
\end{equation*}
and that moreover this equivalence can be realized by a reversible
butterfly, say $[G^\bullet, P, A^\bullet]$.

The foregoing discussion has the following consequence, which can
be considered as an alternative statement for Theorem~\ref{thm:7}:
\begin{corollary}
  For two objects $H^\bullet$, $G^\bullet$ of $\picCM (\s)$ there
  is an equivalence of groupoids:
  \begin{equation*}
    \cat{WM} (H^\bullet, G^\bullet) \lisoto
    \cat{PicB}(H^\bullet, G^\bullet) .
  \end{equation*}
\end{corollary}
Putting all together we immediately obtain the following:
\begin{proposition}
  \label{prop:22}
  There are biequivalences
  \begin{equation*}
    \twocat{Pic} (\s) \iso \bicat{Ch}^{[-1,0]} (\s) \iso \picCM (\s)
  \end{equation*}
  induced by the homomorphisms~\eqref{eq:84}, \eqref{eq:85}.
\end{proposition}
Analogously to ref.\ \cite[Proposition 1.4.15]{0259.14006}, it
results from the above that there is an equivalence of categories
between any of $\twocat{Pic}(\s)^\flat$,
$\bicat{Ch}^{[-1,0]}(\s)^\flat$, or $\picCM (\s)^\flat$, and
$\cat{D}^{[-1,0]}(\s)$, the sub-category of the derived
category of $\cat{Ch}(\s)$ consisting of complexes $K^\bullet$
with $\H^i(K^\bullet)\neq 0$ only for $i= -1,0$.

\subsection{Stacky analogs}
\label{sec:stacky-analogs}

Many of the constructions of the previous sections can be
sheafified over $\s$.

For two objects $A^\bullet$ and $B^\bullet$ of
$\bicat{Ch}^{[-1,0]}(\s)$, the sheaf-theoretic counterpart of the
groupoid $\catHom (B^\bullet, A^\bullet)$, denoted $\shcatHom
(B^\bullet, A^\bullet)$, is obtained in the same way as its
non-abelian counterpart $\stb (-,-)$ in
section~\ref{sec:stacks-butt-weak}, namely by assigning to every
object $U$ of $\s$ the groupoid $\catHom (B^\bullet\rvert_U,
A^\bullet\rvert_U)$.

The same proof as the one for Proposition~\ref{prop:7} yields
\begin{proposition}
  \label{prop:20}
  $\shcatHom (B^\bullet, A^\bullet)$ is a stack over $\s$.
\end{proposition}
Note that by virtue of the remarks on the additivity of the
abelian butterflies, it follows that $\shcatHom (B^\bullet,
A^\bullet)$ is a symmetric gr-stack.

There is an obvious fibered analog $\ch (\s)$ of
$\bicat{Ch}^{[-1,0]}(\s)$ defined exactly as $\cm (\s)$ in
section~\ref{sec:bist-cross-modul}. Recall that for each object
$U$ of $\s$ its fiber over $U$ is $\bicat{Ch}^{[-1,0]}
(\s/U)$. Again, this is a fibration in bicategories and, by the
previous proposition, $\ch (\s)$ is a pre-bistack.  Analogously
to Theorem~\ref{thm:5}, we have
\begin{theorem}
  \label{thm:8}
  $\ch (\s)$ is a bistack.
\end{theorem}
\begin{proof}[Sketch of the proof]
  We need only show that the 2-descent condition for objects
  holds. To do this, we can rely upon Theorem~\ref{thm:5} to
  conclude that given a 2-descent data relative to $\ch$ are
  effective in $\cm (\s)$.

  Thus, let $\lbrace Y_\bullet, A^\bullet, E, \alpha\rbrace$ be a
  2-descent datum relative to a hypercover $Y_\bullet\to U$ as
  in~\ref{prgr:2}, but with all data in $\ch (\s)$. It glues to
  an object $G^\bullet$ of $\cm (\s)_U = \CM (\s /U)$.

  From~\ref{prgr:2}, there is a reversible butterfly
  \begin{math}
    [A^\bullet , F, G^\bullet\rvert_{Y_0}]
  \end{math}
  which implies, using Lemmas~\ref{lem:6} and~\ref{lem:7}, that
  $G^\bullet$ can be made Picard---and $F$ automatically braided.
  
  When we pull back to $Y_1$, using the descent condition, we
  potentially obtain \emph{two} different Picard structures on
  $G^\bullet\rvert_{Y_1}$. One results from $[d_0^*A^\bullet,
  d_0^*F, G^\bullet\rvert_{Y_1}]$, the other from the composition
  of $[d_0^*A^\bullet, E, d_1^*A^\bullet]$ with $[d_1^*A^\bullet,
  d_1^*F, G^\bullet\rvert_{Y_1}]$. It is easy to verify they are
  in fact the same. More generally:
  \begin{claim}
    Let $H^\bullet$ be braided by $\braid{-}{-}_H$, and
    $G^\bullet$ have \emph{two} braided structures
    $\braid{-}{-}_G$ and $\braid{-}{-}_G'$.  Let $E$ (resp.\
    $E'$) sit in a braided butterfly from $(H^\bullet,
    \braid{-}{-}_H)$ to $(G^\bullet, \braid{-}{-}_G)$ (resp.\
    $(G^\bullet, \braid{-}{-}_G)$).

    Assume $E\isoto E'$. If $E$ (hence $E'$) corresponds to an
    essentially surjective butterfly, then
    \begin{equation*}
      \braid{-}{-}_G = \braid{-}{-}_G'.
    \end{equation*}
  \end{claim}
  It follows from the claim that the Picard structure on
  $G^\bullet$ descends to $U$. Thus
  appealing to~\ref{prop:22}, we can conclude that $G^\bullet$ is
  equivalent to a complex $B^\bullet\colon B^{-1}\to B^0$, i.e.\
  an object of $\ch (\s)$ over $U$.
\end{proof}
There are obvious fibered analogs of what we have just discussed
for Picard crossed modules and Picard stacks, yielding the
bistack $\piccm (\s)$ and the 2-stack $\picstacks (\s)$,
respectively, whose fibers over $U$ are $\picCM (\s/U)$ and
$\twocat{Pic}(\s/U)$.

The homomorphisms~\eqref{eq:84}, \eqref{eq:85}, appropriately
localized, provide morphisms of bistacks
\begin{equation*}
  \ch (\s)\lto \picstacks (\s),\qquad
  \piccm (\s) \lto
  \picstacks (\s).
\end{equation*}
and finally the stack analog of Proposition~\ref{prop:22} is
\begin{proposition}
  \label{prop:23}
  There are biequivalences
  \begin{equation*}
    \picstacks (\s) \iso \ch (\s) \iso \piccm (\s).
  \end{equation*}
\end{proposition}

We conclude with the following analogs of \cite[Proposition
12.3]{Noohi:weakmaps}.  For two abelian groups $A,B$ over $\s$
define
\begin{equation*}
  \catExt (B,A)
\end{equation*}
to be the groupoid whose objects are extensions of $B$ by $A$ and
whose morphisms are isomorphisms of extensions. 

Let $A^\bullet \colon A^{-1}\to A^0$ and $B^\bullet \colon
B^{-1}\to B^0$ be two objects of $\bicat{Ch}^{[-1,0]}
(\s)$. There is an obvious forgetful morphism of groupoids
\begin{equation*}
  \catHom (B^\bullet, A^\bullet) \lto \catExt (B^0, A^{-1})
\end{equation*}
which sends a butterfly to its NE-SW diagonal.
\begin{proposition}
  \label{prop:1}
  The above map is a fibration. Its homotopy kernel is equivalent
  to $\catHom (B^\bullet, A^\bullet)_{\mathrm{str}}$ and the sequence
  \begin{equation*}
    \catHom (B^\bullet, A^\bullet)_{\mathrm{str}} \lto
    \catHom (B^\bullet, A^\bullet)\lto
    \catExt (B^0, A^{-1})
  \end{equation*}
  is a short-exact sequence (in the sense of
  section~\ref{sec:exact-sequences}) of symmetric gr-categories.
\end{proposition}
\begin{proof}
  For the additive structures, we refer
  to \loccit The rest is trivial, except perhaps the fibration
  condition. If one has an isomorphism $E\to E'$ fitting in an
  isomorphism of extensions of $B^0$ by $A^{-1}$, and there is a
  butterfly whose NE-SW diagonal is the extension $A^{-1}\to E'
  \to B^0$, the construction of a butterfly with center $E$ is
  just a calculation, and it is left to the reader.
\end{proof}
Note that an extension $0\to A\to E \to B\to 0$ can be thought of
as a butterfly from $[0\to B]$ to $[A\to 0]$, that is, as a weak
morphism $B[0] \to \tors (A)$. It follows that there is a
symmetric gr-stack
\begin{equation*}
  \shcatExt (B, A)
\end{equation*}
whose fibers over $U\in \Ob\s$ are $\catExt (B\rvert_U,
A\rvert_U)$.

Recall the \emph{locally split butterflies} of
Definition~\ref{def:10}. With the obvious changes in notations,
let $\shcatHom (B^\bullet, A^\bullet)_{\mathrm{str}}\sptilde$
denote the stack of locally split butteflies from $B^\bullet$ to
$A^\bullet$.  We obtain the following analog of
Proposition~\ref{prop:1}:
\begin{proposition}
  \label{prop:24}
  Let $A^\bullet$ and $B^\bullet$ be two (global) objects of
  $\bicat{Ch}^{[-1,0]}(\s)$.  There is a short exact sequence of
  symmetric gr-stacks
  \begin{equation*}
    \shcatHom (B^\bullet, A^\bullet)_{\mathrm{str}}\sptilde \lto
    \shcatHom (B^\bullet, A^\bullet)\lto
    \shcatExt (B^0, A^{-1})
  \end{equation*}
  The forgetful map from butterflies to extensions is a
  fibration.
\end{proposition}

\subsection{Ringed sites and locally split butterflies}
\label{sec:locally-split-butterflies}

Let us suppose $\rng$ is a sheaf of commutative unital rings over
$\s$, so that the pair $(\s, \rng)$ becomes a ringed site.  It is
known the category $\Mod (\rng)$ of $\rng$-modules is
abelian.
\begin{paragr}
  The constructions of butterflies, morphism of butterflies,
  etc., in general those of section~\ref{sec:2-categories-picard}
  make sense in the category of $\rng$-modules.  A crossed module
  in $\Mod (\rng)$ will simply be a length-two complex $M^\bullet
  \colon M^{-1}\to M^0$, where $M^{-1}$ and $M^0$ are
  $\rng$-modules. A butterfly from $N^\bullet$ to $M^\bullet$ is
  a diagram
  \begin{equation*}
    \xymatrix@R-0.5em{%
      N^{-1}\ar[dd]_\del \ar@/_0.1pc/[dr]^\kappa  & &
      M^{-1} \ar@/^0.1pc/[dl]_\imath \ar[dd]^\del\\
      & P\ar@/_0.1pc/[dl]_\pi \ar@/^0.1pc/[dr]^\jmath &  \\
      N^0 & & M^0
    }  
  \end{equation*}
  of $\rng$-modules such that the NE-SW diagonal is an extension
  in $\Mod (\rng)$, whereas the NW-SE one is a complex. A split
  butterfly is one where the extension is split in $\Mod
  (\rng)$. A \emph{locally split} one is a butterfly for which
  the splitting occur after restricting to a (generalized) cover.
\end{paragr}
Thus everything is as for the category of abelian sheaves on
$\s$, plus the additional requirement of compatibility with the
$\rng$-linear action. According to this new setting, for two
complexes $N^\bullet$ and $M^\bullet$ let us denote by
$\catHom_{\rng} (N^\bullet, M^\bullet)$ the groupoid of
butterflies, by $\catHom_{\rng} (N^\bullet,
M^\bullet)_{\mathrm{str}}$ that of butterflies corresponding to
strict morphisms. Their stack counterparts will be
$\shcatHom_{\rng}(N^\bullet, M^\bullet)$ and
$\shcatHom_{\rng}(N^\bullet, M^\bullet)_{\mathrm{str}}\sptilde$,
the latter denoting the stack of \emph{locally split}
butterflies.  We then have $\bicat{Ch}^{[-1,0]} (\rng)$, its
strict counterpart $\bicat{Ch}^{[-1,0]} (\rng)_{\mathrm{str}}$,
and the stack analog $\ch (\rng)$.
\begin{paragr}
  $M\in \Ob \Mod (\rng)$ is \emph{locally free} of rank $m$ if
  there is a generalized cover $Y\to *$ such that $M\rvert_Y
  \isoto \bigl( \rng\rvert_Y \bigr)^m$.  Let us say that
  $M^\bullet$ is locally free if both $M^{-1}$ and $M^0$ are.
\end{paragr}
The following is easy:
\begin{proposition}
  \label{prop:25}
  Let $N^\bullet, M^\bullet$ be two locally free complexes. Then
  every butterfly $[N^\bullet, P, M^\bullet]$ is a locally split
  butterfly of locally free objects.
\end{proposition}
\begin{proof}
  In the sequence $0\to M^{-1}\to P\to N^0\to 0$ $N^0$ locally
  free implies that the sequence is locally split. Then $M^{-1}$
  locally free implies that so is $P$.
\end{proof}
\begin{corollary}
  If $N^\bullet$ and $M^\bullet$ are locally free, then
  $\shcatExt_{\rng}(N^0, M^{-1})$ is equivalent to a point.
\end{corollary}
\begin{proof}
  Proposition~\ref{prop:25} gives
  \(%
  \shcatHom_{\rng}(N^\bullet,
  M^\bullet)\iso \shcatHom_{\rng}(N^\bullet,
  M^\bullet)_{\mathrm{str}}\sptilde,%
  \)
  then use Proposition~\ref{prop:24}.
\end{proof}
Restricting our attention to the locally free objects of $\ch
(\rng)$, we immediately obtain that they comprise a full fibered
sub-bicategory of $\ch (\rng)$, which we denote by
$\ch_{\mathrm{loc.fr.}} (\rng)$. The inclusion
\begin{equation*}
  \ch_{\mathrm{loc.fr.}} (\rng) \hookrightarrow \ch (\rng)
\end{equation*}
is a map of pre-bistacks.

\appendix

\section{The 2-stack of gr-stacks}
\label{sec:2-stack-gr}

\numberwithin{equation}{subsection}

In the main text, in Theorem~\ref{thm:4}, it is claimed that the
fibered 2-category $\grstacks(\s)$ is a 2-stack over $\s$.

Here we provide a sketch of a direct, brute-force approach, proof
of this fact. The main reason for providing one at all is that
although this fact should be well know to experts, the authors
could not find an adequate reference, let alone a proof, to it in
the literature.

\subsection{Preliminaries}
\label{sec:preliminaries}

We begin with writing down a few necessary diagrams, translating
some of the definitions recalled in
sect.~\ref{sec:gr-categories}.  Given two gr-stacks $\grc$ and
$\grd$, an additive functor $(F,\lambda) \colon \grc\to \grd$
corresponds to the diagram
\begin{equation*}
  \xymatrix{%
    \grc \times \grc \ar[r]^{F\times F} \ar[d]_{\otimes_\grc}
    \drtwocell<\omit>{\lambda} &
    \grd \times \grd \ar[d]^{\otimes_\grd} \\
    \grc \ar[r]_{F} & \grd}
\end{equation*}
The compatibility of $(F,\lambda)$ with the associator morphisms,
expressed in diagram~\eqref{eq:6}, can be written as:
\begin{equation}
  \label{eq:45}
  \begin{split}
    (\lambda*(\Id,\otimes_\grc)) \circ (\otimes_\grd*
    (\Id,\lambda))\circ
    (a_\grd * (F\times F\times F))  \\
    = (F*a_\grc)\circ (\lambda * (\otimes_\grc,\Id)) \circ
    (\otimes_\grd * (\lambda,\Id))
  \end{split}
\end{equation}
This expresses the commutativity of the cube one can construct
from the diagram above and the ones resulting from the
associativity morphisms for $\grc$ and $\grd$.

A morphism of additive functors $\theta\colon (F,\lambda)
\Rightarrow (G,\mu)$ translates into the equality of the
following two diagrams:
\begin{equation*}
  \xymatrix@+1.5pc{%
    \grc \times \grc \ar[r]^{F\times F} \ar[d]_{\otimes_\grc}
    \drtwocell<\omit>{\lambda} &
    \grd \times \grd \ar[d]^{\otimes_\grd} \\
    \grc \ar[r]^{F}
    \rlowertwocell_G{\theta} & \grd}
  \qquad
  \xymatrix@+1.5pc{%
    \grc \times \grc
    \ruppertwocell<10>^{F\times F}{{\;\;\;\theta\times \theta}}
    \ar[r]_{{G\times G}}
    \drtwocell<\omit>{\mu} \ar[d]^{\otimes_\grc}  &
    \grd \times \grd \ar[d]^{\otimes_\grd} \\
    \grc \ar[r]_{G}  & \grd}
\end{equation*}
or, in other words:
\begin{equation}
  \label{eq:46}
  (\theta * \otimes_\grc) \circ \lambda
  = \mu \circ (\otimes_\grd * (\theta\times \theta)),
\end{equation}
where $*$ denotes horizontal composition (pasting) of 2-arrows.

Let us say that a fibered 2-category $\twofib{F}$ over $\s$ is
\emph{separated} if the 2-morphisms glue, in other words if for
any two objects $X,Y$ of $\twofib{F}$, the fibered category
$\shcatHom_{\twofib{F}}(X,Y)$ is a prestack.  We will say that
$\twofib{F}$ is a \emph{2-prestack} if
$\shcatHom_{\twofib{F}}(X,Y)$ is in fact a stack, that is if
1-morphisms also glue.

\subsection{The proof}
\label{sec:proof}

\begin{claim}
  \label{cl:1}
  $\grstacks({\s})$ is separated.
\end{claim}
\begin{proof}
  Suppose we are given two gr-stacks $\grc$ and $\grd$, plus two
  additive functors $F,G\colon \grc\to \grd$. Let $U_\bullet\to
  *$ be a hypercover, and let
  \begin{equation*}
    \theta\colon F\rvert_{U_0}\Longrightarrow G\rvert_{U_0}
    \colon \grc\rvert_{U_0} \lto \grd\rvert_{U_0}
  \end{equation*}
  be a morphism of additive functors such that
  \begin{equation*}
    d_0^*\theta = d_1^*\theta
  \end{equation*}
  over $U_1$.  Since $\stacks (\s)$ is a 2-stack, there exists a
  2-morphism of stacks $\tilde\theta$ such that
  \begin{equation*}
    \tilde\theta \rvert_{U_0} = \theta,
  \end{equation*}
  and, by construction, $\tilde\theta$ satisfies~\eqref{eq:46} on
  $U_0$.  Since this is an identity between 2-morphisms, it
  follows that~\eqref{eq:46} is then satisfied globally.  This
  proves the claim.
\end{proof}
\begin{claim}
  \label{cl:2}
  $\grstacks(\s)$ is a 2-prestack.
\end{claim}
\begin{proof}
  Let $U_\bullet$ be as in the proof of the previous claim.
  
  This time, let us suppose we are given an appropriate descent
  datum for 1-morphisms, that is, suppose we are given
  1-morphisms
  \begin{equation*}
    (F,\lambda) \colon \grc\rvert_{U_0} \lto \grd\rvert_{U_0}.
  \end{equation*}
  over $U_0$, 2-morphisms
  \begin{equation*}
    \theta \colon d_0^*(F,\lambda) \Longrightarrow
    d_1^*(F,\lambda)
  \end{equation*}
  over $U_1$, such that the relation
  \begin{equation*}
    d_1^*\theta = d_2^*\theta\circ d_0^*\theta
  \end{equation*}
  holds over $U_2$.

  This implies that in $\stacks (\s)$ there exist a 1-morphism
  $\tilde F \colon \grc\to \grd$ and a 2-morphism $\tau \colon F
  \Rightarrow \tilde F\rvert_{U_0}$ such that
  \begin{equation}
    \label{eq:47}
    \vcenter{%
      \xymatrix@-1.5pc{%
        d_0^*F \ar@{=>}[rr]^\theta \ar@/_0.5pc/@{=>}[dr]_{d_0^*\tau}
        && d_1^*F \ar@/^0.5pc/@{=>}[dl]^{d_1^*\tau} \\
        & \tilde F\rvert_{U_1}}}
  \end{equation}
  is a commuting diagram of natural transformations over $U_1$.

  We can use these data to enforce the additivity condition on to
  $\tilde F$. Namely, let us \emph{define} the natural
  isomorphism
  \begin{equation*}
    \tilde \lambda \colon \otimes_\grd\circ
    (\tilde F\times \tilde F)
    \Longrightarrow \tilde F\circ \otimes_\grc
  \end{equation*}
  \emph{over $U_0$} by:
  \begin{equation*}
    \tilde\lambda \circ (\otimes_\grd * (\tau\times \tau))
    = (\tau * \otimes_\grc)  \circ \lambda.
  \end{equation*}
  (Note that the compatibility with the associators,
  namely~\eqref{eq:6}, or~\eqref{eq:45}), follows automatically.)
  An elementary manipulation using~\eqref{eq:47} then shows that
  $d_0^*\tilde\lambda = d_1^*\tilde\lambda$ and therefore, by the
  preceding claim, there exists a global $\tilde{\tilde \lambda}$
  satisfying
  \begin{equation*}
    \tilde{\tilde\lambda}\rvert_{U_0} = \tilde\lambda.
  \end{equation*}
  Again, the compatibility with the associators will follow.
  Thus, the morphism $\tilde F$ is actually additive: the pair
  $(\tilde F, \tilde{\tilde \lambda})$ satisfies the required
  properties, which proves the claim.
\end{proof}
Having so far shown that $\grstacks (\s)$ is a 2-prestack, the
hardest step is to show that $\grstacks (\s)$ actually is a
2-stack, as there are many more diagrams to check.

The problem is posed by giving ourselves an appropriate 2-descent
datum of gr-stacks. This means that if $U_\bullet$ is the
previously introduced hypercover, then we are given the following
data:
\begin{enumerate}
\item A gr-stack $\grc$ over $U_0$.
\item An additive functor $(F,\lambda) \colon d_0^* \grc \to
  d_1^*\grc$ over $U_1$.
\item A morphism of additive functors
  \begin{equation*}
    \theta \colon d_2^* (F,\lambda) \circ d_0^* (F,\lambda)
    \Longrightarrow d_1^*(F,\lambda)
  \end{equation*}
  over $U_2$.
\item A coherence diagram for $\theta$ over $U_3$, in other
  words, the diagram formed by the faces of the tetrahedron must
  commute:
  \begin{equation*}
    \vcenter
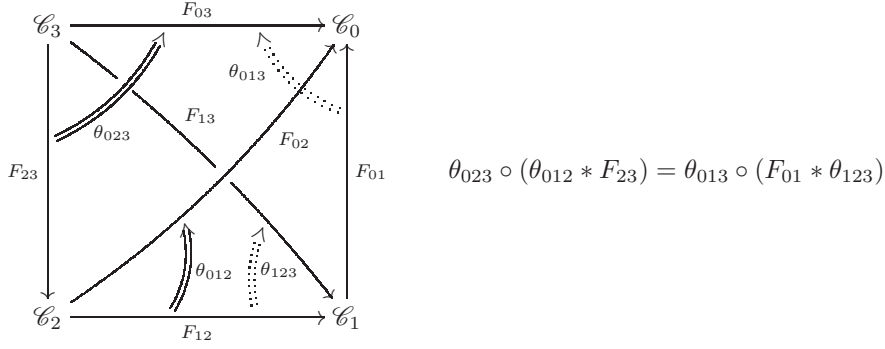
{%
      \xymatrix@+6pc{%
        \grc_3
        \ar[d]^(.4){}="23a"_{F_{23}}
        \ar[r]_(.4){}="03a"_(.7){}="03b"^{F_{03}}
        \ar@/^/[dr]|(.23)\hole|(.55)\hole_(.7){}="13"^(.4){F_{13}}
        \ar@{=>}@/_/ "23a";"03a"_(.3){\theta_{023}}
        &
        \grc_0 \\
        \grc_2
        \ar[r]^(.4){}="12a"^(.7){}="12b"_{F_{12}}
        \ar@/_0.7pc/[ur]_(.4){}="02"_(.7){F_{02}}
        \ar@{=>}@/_0.4pc/ "12a";"02"_{\theta_{012}}
        \ar@{:>}@/^0.3pc/ "12b";"13"_{\theta_{123}}
        &
        \grc_1
        \ar[u]^(.7){}="01"_{F_{01}}
        \ar@{:>}@/^/ "01";"03b"^(.7){\theta_{013}}
      }}
    \qquad
    \theta_{023}\circ (\theta_{012}*F_{23})
    = \theta_{013} \circ (F_{01} * \theta_{123})
  \end{equation*}
\end{enumerate}
We have used a ``missing index'' convention:
$\grc_0=(d_1d_2d_3)^*\grc$, $F_{01}=(d_2d_3)^*F$,
$\theta_{012}=d_3^*\theta$, and so on.

Again, since $\stacks (\s)$ is a 2-stack, there exists a
\emph{stack} $\tilde\stc$ over $\s$ with an equivalence of stacks
\begin{equation*}
  G\colon \grc \lisoto \tilde\stc\rvert_{U_0}
\end{equation*}
and a 2-morphism of stacks over $U_1$:
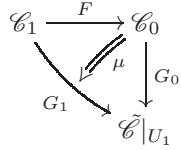
\begin{equation*}
  \xymatrix{%
    \grc_1
    \ar@/_/[dr]_{G_1}^{}="1"
    \ar[r]^{F} &
    \grc_0 \ar[d]^{G_0} \ar@{=>}@/_0.2pc/ [];"1"^\mu \\
    & \tilde\stc\rvert_{U_1}
  }
\end{equation*}
such that the system is coherent over $U_2$:
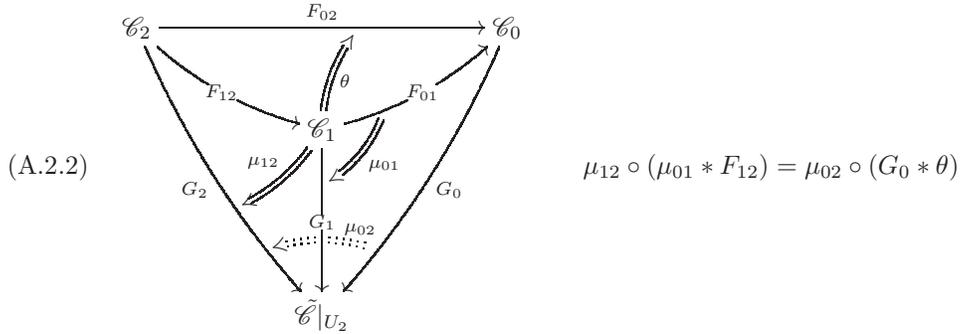
\begin{equation}
  \label{eq:48}
  \vcenter{%
    \xymatrix@C+2pc{%
      \grc_2
      \ar@/_/[dr]|{F_{12}}
      \ar@/_/[dddr]^(.6){}="2stc" ^(.75){}="2stca"_{G_2}
      \ar[rr]^{F_{02}}_(.6){}="02"
      &&
      \grc_0
      \ar@/^/[dddl]_(.75){}="0stc"^{G_0} \\
      & \grc_1
      \ar[dd]_(.3){}="1stc"|{G_1}
      \ar@/_/[ur]_(.3){}="01"|{F_{01}}
      \ar@{=>}@/^0.3pc/ [];"02"_{\theta}
      \ar@{=>}@/^0.3pc/ [];"2stc"_{\mu_{12}}
      \\ \\
      & \tilde\stc\rvert_{U_2}
      \ar@{=>}@/^0.3pc/ "01";"1stc"^{\mu_{01}}
      \ar@{:>}@/_0.3pc/ "0stc";"2stca"_<<{\mu_{02}}
    }}
  \qquad
  \mu_{12}\circ (\mu_{01} * F_{12}) =
  \mu_{02} \circ ( G_0 * \theta )
\end{equation}

We want to show these data can be used to induce a gr-structure
on $\tilde\stc$. This means we have to induce multiplication,
inverse, and unit object functors:
\begin{equation*}
  \otimes_{\tilde\stc} \colon \tilde\stc \times \tilde\stc  \lto \tilde\stc,
  \quad
  *\colon \tilde\stc^\mathrm{op} \lto  \tilde\stc,
  \quad
  \mathbf{{1}} \lto \tilde\stc
\end{equation*}
satisfying the usual diagrams, such as:
\begin{equation}
  \label{eq:49}
  \vcenter{%
    \xymatrix{%
      \tilde\stc \times \tilde\stc \times \tilde\stc
      \ar[r]^(.6){\otimes_{\tilde\stc}\times\Id}
      \ar[d]_{\Id\times\otimes_{\tilde\stc}} &
      \tilde\stc \times \tilde\stc
      \ar[d]^{\otimes_{\tilde\stc}}_{}="a" \\
      \tilde\stc \times \tilde\stc
      \ar[r]_{\otimes_{\tilde\stc}}^{}="b" & \tilde\stc
      \ar@{=>}@/_0.4pc/ "a";"b"_a
    }
  }, \quad 
  \vcenter{%
    \xymatrix{%
      \mathbf{1}\times \tilde\stc \ar[r] \ar@/_/[dr]
      \drtwocell<\omit>{<-1.8>l}
      &
      \tilde\stc\times \tilde\stc \ar[d]^\otimes_{\tilde\stc} \\
      & \tilde\stc
    }
  },\quad
  \vcenter{%
    \xymatrix{%
      \tilde\stc \times \mathbf{1} \ar[r] \ar@/_/[dr]
      \drtwocell<\omit>{<-1.8>r}
      &
      \tilde\stc\times \tilde\stc \ar[d]^\otimes_{\tilde\stc} \\
      & \tilde\stc
    }
  }
\end{equation}
and so on.  This is done locally (i.e. on $U_0$) using the
gr-structure of $\grc$. For example, consider the multiplicative
structure. We can induce one via the diagram:
\begin{equation*}
  \xymatrix{%
    \grc \times \grc \ar[r]^{G\times G} \ar[d]_{\otimes_\grc}
    \drtwocell<\omit>{\nu} &
    **[r]\tilde\stc\rvert_{U_0} \times \tilde\stc\rvert_{U_0}
    \ar[d]^{\tilde\otimes} \\
    \grc \ar[r]_{G} & **[r]\tilde\stc\rvert_{U_0}}
\end{equation*}
where $\nu$ is determined by the choice of a quasi-inverse of
$G$.  The multiplication morphism $\tilde\otimes$ defined by this
diagram lives over $U_0$, and we must show that the gluing data
for $\tilde\stc$ in $\stacks (\s)$ yield appropriate gluing data
for $\tilde\otimes$ along $U_\bullet$, so that it will glue to a
global monoidal functor $\otimes_{\tilde\stc}$.

The two possible pull-backs of the morphism $\tilde\otimes$ to
$U_1$ are related by a 2-morphism $\epsilon\colon
\tilde\otimes_1\Rightarrow \tilde\otimes_0$ which is defined by
the following diagram:
\begin{equation}
  \label{eq:50}
  \vcenter{%
    \xymatrix{%
      \grc_1 \times \grc_1
      \ar[rr]^{F\times F}_{}="0011"
      \ar[ddd]_{\otimes_1}
      \ar@/_/[dr]|{G_1\times G_1}="tttt1"
      &  &
      \grc_0 \times \grc_0
      \ar[ddd]^{\otimes_0}
      \ar@/^/[dl]|{G_0\times G_0}
      \\
      & \tilde\stc\times \tilde\stc
      \ar@{=>}@/^/ "0011";"tttt1"^(.3){\mu\times\mu}
      \ar@/_0.7pc/@<-1ex>[d]_(.3){\tilde\otimes_1}_{}="s1"^{}="a1"
      \ar@/^0.7pc/@<+1ex>[d]^(.3){\tilde\otimes_0}^{}="s0"_{}="a0"
      \ar@{=>} "a1";"a0"^\epsilon
      \\
      & \tilde\stc
      \\
      \grc_1
      \ar@/^/[ur]_{G_1}^{}="t1"_(.7){}="tt1"
      \ar[rr]_{F}^{}="01"
      \ar@{=>}@/_/ "01";"tt1"^{\mu}
      &  &
      \grc_0
      \ar@/_/[ul]^(.7){G_0}_{}="t0"
      \ar@{=>}@/_0.3pc/ "s1";"t1"_{\nu_1}
      \ar@{=>}@/^0.3pc/ "s0";"t0"^{\nu_0}
      \ar @{:>} {[]!<-1ex,8ex>};{"01"!<6ex,1ex>}^(.3)\lambda
    }}
\end{equation}
(In the previous diagram we are still using the ``missing index''
convention introduced above.)  We need the 2-arrow $\epsilon$
to be coherent on $U_2$, namely that, after pulling everything
back to $U_2$, it satisfies
\begin{equation}
  \label{eq:51}
  d_1^*\epsilon = d_2^*\epsilon\circ d_0^*\epsilon.
\end{equation}
To see how this may be true, we should consider the pasting
diagram resulting from the three possible pullbacks
of~\eqref{eq:50} to $U_2$.
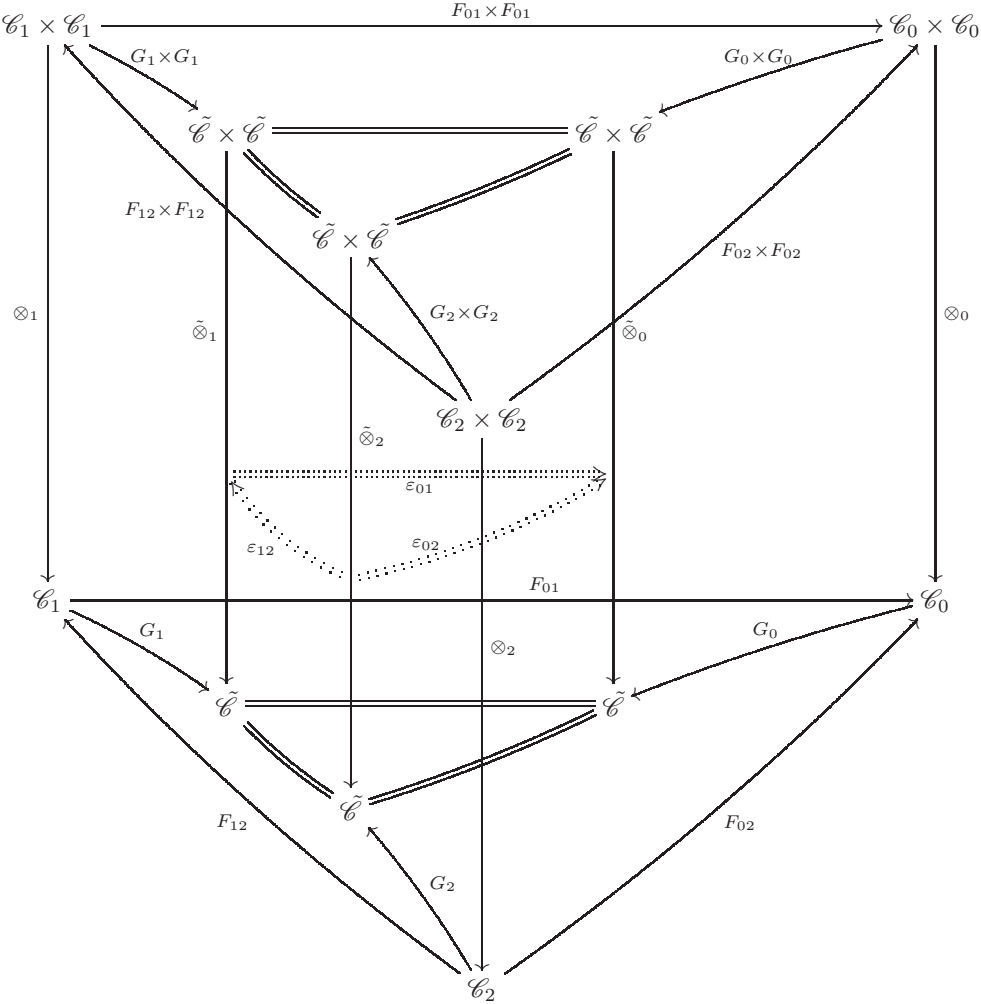
\begin{equation}
  \label{eq:52}
  \xymatrix@C-1pc{%
    \grc_1\times\grc_1 \ar [0,10]^{F_{01}\times F_{01}}
    \ar [6,0]_{\otimes_1}
    \ar @/^0.2pc/ [1,2]^{G_1\times G_1}
    & & & & & & & & & &
    \grc_0\times\grc_0 \ar [6,0]^{\otimes_0}
    \ar @/_0.2pc/ [1,-5]_{G_0\times G_0}
    \\
    & & \tilde\stc\times\tilde\stc
    \ar @{=} [0,3]
    \ar @{=} @/_0.2pc/[1,1]
    \ar [6,0]_(.35){\tilde\otimes_1}^(.6){}="1"
    & & & \tilde\stc\times\tilde\stc
    \ar [6,0]^(.35){\tilde\otimes_0}_(.6){}="0" \\ 
    & & & \tilde\stc\times\tilde\stc
    \ar @{=} @/_0.2pc/ [-1,2]
    \ar [6,0]^(.35){\tilde\otimes_2}_(.6){}="2" \\ \\
    & & & &\grc_2\times\grc_2 \ar @/_0.2pc/ [-2,-1]_{G_2\times G_2}
    \ar @/^/ [-4,-4]^(.6){F_{12}\times F_{12}}
    \ar [6,0]^(.4){\otimes_2}
    \ar @/_/ [-4,6]_{F_{02}\times F_{02}} \\
    \\
    \grc_1 \ar [0,10]^(.56){F_{01}}
    \ar @/^0.2pc/ [1,2]^{G_1}
    & & & & & & & & & &
    \grc_0 \ar @/_0.2pc/ [1,-5]_{G_0} \\
    & & \tilde\stc \ar @{=} [0,3] \ar @{=} @/_0.2pc/[1,1] 
    & & & \tilde\stc \\
    & & & \tilde\stc \ar @{= } @/_0.2pc/ [-1,2]\\ \\
    & & & & \grc_2 \ar @/_/ [-4,6]_{F_{02}}
    \ar @/_0.2pc/ [-2,-1]_{G_2}
    \ar @/^/ [-4,-4]^{F_{12}}
    \ar @{:>} "1";"0"_{\epsilon_{01}}
    \ar @{:>} @/^0.4pc/ "2";"1"^{\epsilon_{12}}
    \ar @{:>} @/_0.4pc/ "2";"0"^(.3){\epsilon_{02}}
  }
\end{equation}
For added clarity, we have not explicitly written the 2-arrows,
except for the $\epsilon_{ij}$, $ij=01, 02, 12$.  In the diagram,
the bottom pyramid is the diagram in eq.~\eqref{eq:48}, whereas
the the top pyramid is simply the product of same with
itself. The three lateral prisms are the three pull-backs
of~\eqref{eq:50}. All these blocks 2-commute, in the sense that
their faces form a system of commutative 2-arrows.  Therefore,
all the faces of these parts in the diagram in \eqref{eq:52}
commute, forcing the central triangular prism to commute as well.

Since relation~\eqref{eq:51} is satisfied, we can invoke
Claim~\ref{cl:2} to conclude that there exists a global functor
\begin{equation*}
  \otimes_{\tilde\stc} \colon \tilde\stc\times \tilde\stc \lto \tilde\stc
\end{equation*}
equipped with a 2-arrow
\begin{equation*}
  \rho\colon \tilde\otimes \Longrightarrow \otimes_{\tilde\stc}\rvert_{U_0}
\end{equation*}
such that the relation
\begin{equation*}
  \rho_0 \circ \epsilon = \rho_1
\end{equation*}
is satisfied over $U_1$ (again, $\rho_0=d_1^*\rho$ and
$\rho_1=d_0^*\rho$).  As a result, there is a morphism of
additive functors over $U_0$:
\begin{equation*}
  \xymatrix{%
    \grc \times \grc \ar[r]^{G\times G} \ar[d]_{\otimes_\grc}
    \drtwocell<\omit>{\tilde\nu} &
    **[r]\tilde\stc\rvert_{U_0} \times \tilde\stc\rvert_{U_0}
    \ar[d]^{\otimes_{\tilde\stc}\rvert_{U_0}} \\
    \grc \ar[r]_{G} & **[r]\tilde\stc\rvert_{U_0}}
\end{equation*}
where $\tilde\nu = \nu\circ (\rho^{-1} * (G\times G))$, such that
the following diagram holds
\begin{equation}
  \label{eq:53}
  \vcenter{%
    \xymatrix{%
      \grc_1 \times \grc_1
      \ar[rr]^{F\times F}_{}="0011"
      \ar[ddd]_{\otimes_1}
      \ar@/_/[dr]|{G_1\times G_1}="tttt1"
      &  &
      \grc_0 \times \grc_0
      \ar[ddd]^{\otimes_0}
      \ar@/^/[dl]|{G_0\times G_0}
      \\
      & \tilde\stc\times \tilde\stc
      \ar@{=>}@/^/ "0011";"tttt1"^(.3){\mu\times\mu}
      \ar[d]^(.3){\otimes_{\tilde\stc}}_{}="s1"^{}="s0"
      \\
      & \tilde\stc
      \\
      \grc_1
      \ar@/^/[ur]_{G_1}^{}="t1"_(.7){}="tt1"
      \ar[rr]_{F}^{}="01"
      \ar@{=>}@/_/ "01";"tt1"^{\mu}
      &  &
      \grc_0
      \ar@/_/[ul]^(.7){G_0}_{}="t0"
      \ar@{=>}@/_0.5pc/ "s1";"t1"_{\nu_1}
      \ar@{=>}@/^0.5pc/ "s0";"t0"^{\nu_0}
      \ar @{:>} {[]!<-1ex,8ex>};{"01"!<6ex,1ex>}^(.3)\lambda
    }  }
\end{equation}
in place of~\eqref{eq:50}, over $U_1$.

For $\otimes_{\tilde\stc}$ to be a true monoidal functor, there
must be an associator 2-morphism $\tilde a$, as in the first
diagram in~\eqref{eq:49}, satisfying the appropriate coherence
condition---the pentagon identity.  The cube
\begin{equation*}
  \xymatrix{%
    \grc\times \grc\times \grc
    \ar[rr]^{(\otimes,\Id)}
    \ar[dd]_(.3){(\Id,\otimes)}
    \ar@/_/[dr]
    & & \grc\times \grc
    \ar@/^/[dr] 
    \ar[dd]_(.3){\otimes}|\hole \\
    & \tilde\stc\times \tilde\stc\times \tilde\stc
    \ar[rr]^(.34){(\otimes_{\tilde\stc},\Id)}
    \ar[dd]^(.3){(\Id,\otimes_{\tilde\stc})}
    & & \tilde\stc\times \tilde\stc
    \ar[dd]^(.3){\otimes_{\tilde\stc}}
    \ar@{=>} +<-8ex,1ex>;!<-6ex,5ex>^(.6){\scriptscriptstyle(\tilde \nu,\Id_G)}
    \\
    \grc\times \grc
    \ar[rr]^(.4){\otimes}|(.53)\hole
    \ar@/_/[dr]
    & & \grc
    \ar@/^/[dr]|(.62)\hole
    \ar@{=>}@/_/ !<0ex,6ex>;!<-6ex,0ex>^a
    \\
    & \tilde\stc\times \tilde\stc
    \ar[rr]_{\otimes_{\tilde\stc}}
    \ar@{=>} +<-1ex,6ex>;+<-5ex,2ex>_(.6){\scriptscriptstyle(\Id_G,\tilde \nu)}
    & & \tilde\stc
    \ar@{=>}@/_/ !<0ex,7ex>;!<-7ex,0ex>^(.6){\tilde a}
    \ar@{=>} !<0ex,10ex>;!<-6ex,7ex>_(.6){\tilde \nu}
    \ar@{=>} !<-10ex,0ex>;!<-8ex,5ex>^(.6){\tilde \nu}
  }
\end{equation*}
will define the 2-arrow $\tilde a$ over $U_0$. Its two pullbacks
to $U_1$, together with the appropriate number of copies (four)
of~\eqref{eq:53}, and the cube resulting from the fact that
$(F,\lambda)$ is an additive functor (hence compatible with the
associators), give rise to a 2-commutative hypercube from which
it follows that $d_0^*\tilde a=d_1^*\tilde a$.  Applying
Claim~\ref{cl:1} implies that there exists a globally
well-defined 2-arrow
\begin{equation*}
  a^{\tilde\stc}\colon
  \otimes_{\tilde\stc}\circ (\otimes_{\tilde\stc}, \Id)
  \Longrightarrow \otimes_{\tilde\stc}\circ (\Id,\otimes_{\tilde\stc})
\end{equation*}
such that $a^{\tilde\stc}\rvert_{U_0} = \tilde a$.

As for the pentagon identity, it holds locally, that is, on
$U_0$, since it does for $a$. More precisely, from the cube
expressing the pentagon identity for $a$, and the five copies of
the above cube, we can form another hypercubic diagram. Of the
two remaining cubes of this tesseract, one corresponds to the
operation $(X,Y,Z,W)\to (X\otimes Y)\otimes (Z\otimes W)$,
whereas the latter corresponds to the pentagon identity for
$\tilde a$. Since all the first seven cubes are 2-commutative, so
must be the case for the latter. It follows that $a^{\tilde\stc}$
satisfies the pentagon identity.

This proves that $\tilde\stc$ has a monoidal functor. The rest of
the functors comprising the gr-stack structure are treated in an
analogous (and equally lengthy!) way. The same is true for the
coherence conditions that should be satisfied by the
multiplication functor, the inverse, and the identity. After
ref.\ \cite{MR723395}, this amounts to the commutativity of
\begin{equation*}
  \xymatrix{%
    (X\otimes X^*)\otimes X \ar[d]_{e\otimes X}
    \ar[rr]^a & & X\otimes (X^*\otimes X )
    \ar[d]^{X\otimes \eta} \\
    I\otimes X \ar[r]_{l_X} & X & X\otimes I \ar[l]^{r_X}
  }
\end{equation*}
written in an object-wise fashion. This corresponds to the
2-commutativity of
\begin{equation*}
  \xymatrix{%
    \tilde\stc \times \tilde\stc \times \tilde\stc
    \ar[rr]^{(\otimes,\Id)}_{}="es"
    \ar[dd]_{(\Id,\otimes)}^{}="etas"
    & &
    \tilde\stc \times \tilde\stc \ar[dd]^{\otimes}_(.7){}="as"\\
    & \tilde\stc
    \ar[ul]|{(\Id,*,\Id)}
    \ar@/_0.3pc/[ur]^(.25){(I,\Id)}_{}="ls"^{}="et"
    \ar@/^0.2pc/[dl]_(.25){(\Id,I)}^{}="rs"_{}="etat"
    \ar@/^0.2pc/[dr]^\Id_(.3){}="rt"^(.3){}="lt"
    \ar@{=>}@/^/ "ls";"lt"^l
    \ar@{=>}@/_/ "rs";"rt"_r
    \ar@{=>}@/^/ "etas";"etat"_\eta
    \ar@{=>}@/_/ "es";"et"^e
    \\
    \tilde\stc\times \tilde\stc \ar[rr]_{\otimes}^(.7){}="at"
    \ar@{:>}@/_/ "as";"at"^(.6)a
    & & \tilde\stc
  }
\end{equation*}
These observations complete the proof.\qed

\backmatter%
\bibliography{general}%
\bibliographystyle{smfalpha}%

\end{document}